\documentclass{amsart}

\usepackage{amsmath}
\usepackage{amsfonts}
\usepackage{amssymb}
\usepackage{amsthm}
\usepackage{comment}
\usepackage{epsfig}
\usepackage{psfrag}
\usepackage{mathrsfs}
\usepackage{amscd}
\usepackage[all]{xy}
\usepackage{rotating}
\usepackage{lscape}
\usepackage{amsbsy}
\usepackage{verbatim}
\usepackage{moreverb}
\usepackage{color}
\usepackage{bbm}
\usepackage{eucal}

\newtheorem{theorem}{Theorem}[section]
\newtheorem{prop}[theorem]{Proposition}
\newtheorem{lemma}[theorem]{Lemma}
\newtheorem{cor}[theorem]{Corollary}

\theoremstyle{definition}
\newtheorem{definition}[theorem]{Definition}

\newtheorem{observation}[theorem]{Observation}

\newtheorem{construction}[theorem]{Construction}
\newtheorem{terminology}[theorem]{Terminology}
\newtheorem{remark}[theorem]{Remark}

\newtheorem{example}[theorem]{Example}

\DeclareMathOperator{\Psh}{\sf PShv}
\DeclareMathOperator{\Aut}{\sf Aut}
\DeclareMathOperator*{\colim}{\sf colim}
\DeclareMathOperator*{\limit}{\sf lim}

\DeclareMathOperator{\Hom}{\sf Hom}

\DeclareMathOperator{\Fun}{\sf Fun}

\DeclareMathOperator{\Map}{\sf Map}
\DeclareMathOperator{\Mapc}{{\sf Map}_{\sf c}}
\DeclareMathOperator{\Gammac}{{\Gamma}_{\!\sf c}}

\DeclareMathOperator{\Space}{\sf Spaces}

\DeclareMathOperator{\sym}{\sf Sym}
\DeclareMathOperator{\MC}{\sf MC}

\DeclareMathOperator{\Cat}{\sf Cat_\infty}
\DeclareMathOperator{\uno}{\mathbbm{1}}
\DeclareMathOperator{\exit}{\sf Exit}

\DeclareMathOperator{\m}{\sf Mod}

\DeclareMathOperator{\shv}{\sf Shv}
\DeclareMathOperator{\csh}{\sf cShv}

\DeclareMathOperator{\Alg}{\mathsf{Alg}}

\DeclareMathOperator{\cAlg}{\mathsf{cAlg}}

\DeclareMathOperator{\op}{\mathsf{op}}
\DeclareMathOperator{\com}{\mathsf{Com}}

\DeclareMathOperator{\Top}{\mathsf{Top}}
\DeclareMathOperator{\Emb}{\mathsf{Emb}}

\DeclareMathOperator{\conf}{\mathsf{Conf}}

\DeclareMathOperator{\Spaces}{\mathsf{Spaces}}
\DeclareMathOperator{\spaces}{\mathsf{Spaces}}
\DeclareMathOperator{\spectra}{\mathsf{Spectra}}
\DeclareMathOperator{\mfld}{{\cM}\mathsf{fld}}
\DeclareMathOperator{\snglr}{{\cS}\mathsf{nglr}}
\DeclareMathOperator{\bsc}{{\cB}\mathsf{sc}}
\DeclareMathOperator{\dmfld}{\mathsf{Mfld}}

\DeclareMathOperator{\fr}{\sf fr}

\DeclareMathOperator{\Sing}{\mathsf{Sing}}

\DeclareMathOperator{\BO}{{\mathsf BO}}

\DeclareMathOperator{\Lie}{\sf Lie}

\def\ot{\otimes}

\DeclareMathOperator{\fin}{\sf Fin}

\DeclareMathOperator{\oo}{\infty}

\DeclareMathOperator{\hh}{\sf HH}

\DeclareMathOperator{\free}{\sf Free}

\DeclareMathOperator{\ran}{\sf Ran}

\DeclareMathOperator{\disk}{{\cD}\sf isk}
\DeclareMathOperator{\ddisk}{{\sf Disk}}
\DeclareMathOperator{\Disk}{{\cD}\sf isk}
\DeclareMathOperator{\Mfld}{{\cM}\sf fld}

\DeclareMathOperator{\bdelta}{\boldsymbol{\Delta}}

\newcommand{\ra}{\rightarrow}

\newcommand{\xra}{\xrightarrow}
\newcommand{\xla}{\xleftarrow}

\def\cA{\mathcal A}\def\cB{\mathcal B}\def\cC{\mathcal C}\def\cD{\mathcal D}
\def\cE{\mathcal E}\def\cF{\mathcal F}\def\cG{\mathcal G}
\def\cI{\mathcal I}
\def\cM{\mathcal M}\def\cO{\mathcal O}
\def\cS{\mathcal S}
\def\cU{\mathcal U}\def\cV{\mathcal V}\def\cX{\mathcal X}

\def\AA{\mathbb A}\def\DD{\mathbb D}
\def\EE{\mathbb E}\def\FF{\mathbb F}

\def\NN{\mathbb N}
\def\RR{\mathbb R}\def\SS{\mathbb S}\def\TT{\mathbb T}

\def\ZZ{\mathbb Z}

\def\sB{\mathsf B}\def\sC{\mathsf C}
\def\sH{\mathsf H}

\def\sN{\mathsf N}\def\sO{\mathsf O}\def\sP{\mathsf P}

\def\sU{\mathsf U}
\def\sZ{\mathsf Z}

\def\bH{\mathbf H}

\def\bDelta{\mathbf\Delta}

\begin{document}

\title{A factorization homology primer}
\author{David Ayala \& John Francis}
\address{Department of Mathematics\\Montana State University\\Bozeman, MT 59717}
\email{david.ayala@montana.edu}
\address{Department of Mathematics\\Northwestern University\\Evanston, IL 60208}
\email{jnkf@northwestern.edu}
\thanks{DA was supported by NSF awards 1507704. 
JF was supported by NSF award 1508040.}

\begin{abstract} 
This article amalgamates some foundational developments and calculations in factorization homology, including natural (co)filtrations thereof.
\end{abstract}



\maketitle


\section{Introduction}

This article is an introduction to factorization homology---or factorization algebras---in the topological setting; see also Lurie~\cite{dag}. For introductions in the closely related algebro-geometric or Riemannian settings, see Beilinson--Drinfeld~\cite{bd}, Francis--Gaitsgory~\cite{fg}, and Costello--Gwilliam~\cite{kevinowen}.

\medskip

Factorization homology takes:
\begin{itemize}
\item a geometric input: an $n$-manifold $M$;
\item an algebraic input: an $n$-disk algebra $A$, or a stack $X$ over $n$-disk algebras, in a symmetric monoidal $\oo$-category $\cV$;\footnote{$\cV$ is usually $\m_\Bbbk$, $\spectra$, or $\spaces$---respectively, chain complexes over a commutative ring $\Bbbk$, spectra, or spaces.
See~\S2.1.2 of~\cite{dag} for a thorough development of \emph{symmetric monoidal $\infty$-categories}.
}
\end{itemize}
The resulting factorization homology
\[
\int_M A
\]
is an object of $\cV$. It can be thought of as the integral of the algebra $A$ over the manifold $M$, in the same sense that ordinary homology $\sH_\ast(M,A)$ is given by integrating an abelian group $A$ over a space $M$.\footnote{This is the alpha version of the theory; there is a beta version, developed in \cite{emb1a}, which takes more general geometric and algebraic inputs.} The functor $\int A$ is covariant with respect to open embeddings in the manifold variable: for each fixed $M$, the functor on the poset of opens $\int A: {\sf Opens}(M) \ra \cV$, sending $U\subset M$ to $\int_U A$, defines a {\it factorization algebra} on $M$, both in the definitions of Beilinson--Drinfeld~\cite{bd} and Costello--Gwilliam~\cite{kevinowen}.

\medskip

\medskip

Factorization homology has three essential features making it technically advantageous:
\begin{enumerate}
\item {\bf Local-to-global principle}: $\ot$-excision, generalizing the Eilenberg--Steenrod axioms;
\item {\bf Filtration}: a generalization of the Goodwillie--Weiss embedding calculus;
\item {\bf Duality}: Poincar\'e/Koszul duality.
\end{enumerate}

The algebraic input $A$ can arise from a variety of sources: as an $n$-fold loop space in classical algebraic topology; or as a formal deformation of an $n$-Poisson algebra. This latter class of examples arises in higher quantization in mathematical physics, a major source of contemporary interest in factorization methods. 

\medskip

The following is a table of values for factorization homology, as the algebra input varies, which we will explain through this article.

\begin{center}
    \begin{tabular}{|p{2.8cm}|  p{4cm} | p{6.3cm}|}
    \hline
        {\bf Manifold input} & {\bf Algebra input}  &{\bf Factorization homology output}\\ \hline
    $M=S^1$ the circle & $A$ an associative algebra &$\displaystyle\int_{S^1}A \simeq \hh_\bullet(A)$ \newline Hochschild homology \\ \hline
    $M$ an $n$-manifold & $A$ an abelian group & $\displaystyle\int_MA \simeq \sH_\ast(M,A)$\newline ordinary homology  \\ \hline
        $M$ an $n$-manifold & $A$ a spectrum & $\displaystyle\int_MA \simeq \Sigma^{\oo}_\ast M \wedge A$\newline generalized homology  \\ \hline
        $M$ an $n$-manifold & $A$ a commutative algebra \newline (e.g., $\sym(V)$) & $\displaystyle\int_MA \simeq M\boxtimes A$ \newline categorical tensoring \newline (e.g., $\displaystyle\int_M \sym(V) \simeq \sym(\sC_\ast(M)\ot V)$ )  \\ \hline 
        $M$ an $n$-manifold & $A= \FF_nV$ a free $n$-disk algebra & $\displaystyle\int_M\FF_nV \simeq\bigoplus_{k\geq 0} \sC_\ast(\conf^{\fr}_k(M)) \underset{\Sigma_k \wr \sO(n)}\ot V^{\ot k}$  \\ \hline
        $M$ an $n$-manifold & $A=\Omega^n K$ the $n$-fold loop space of an $n$-connective space & $\displaystyle\int_M\Omega^n K\simeq \Mapc(M, K)$\newline \newline space of compactly-supported maps   \\ \hline
        $M$ an $n$-manifold & $A=\sU_n(\frak g)$ an $n$-disk enveloping algebra of a Lie algebra & $\displaystyle\int_M\sU_n (\frak g) \simeq \sC_\ast^{\Lie}\bigl(\sC_{\sf c}^\ast(M,\frak g)\bigr)$ \newline Lie algebra homology  \\ \hline
        $M$ an $n$-manifold & $A= {\sf Obs}(\RR^n)$ the local observables in a TQFT & $\displaystyle\int_M {\sf Obs}(\RR^n) \longrightarrow {\sf Obs}(M)$ \newline $\ot$-excisive left-approximation to global observables \\ \hline

    \end{tabular}
\end{center}

\medskip

The last row is not a genuine equivalence in general: 
the result of alpha factorization homology is not always the genuine global observables, but only those observables that are determined by the point-local ones. 
Additionally, the alpha version of factorization homology does not compute the state space of a field theory. This motivates a beta version of factorization homology, developed in \cite{emb1a}, where the algebraic input is a more general object, an $(\oo,n)$-category, and which satisfies a more general, and nuanced, form of excision.
As this beta version of factorization homology is more involved, it will not be discussed in this article.  

\medskip

While the subject of factorization homology is relatively new in name, it has important roots and antecedents. It derives foremost from the algebro-geometric theory of chiral and factorization algebras developed by Beilinson and Drinfeld in conformal field theory in \cite{bd}. Secondly, it has an antecedent in the labeled configuration space models of mapping spaces dating to the 1970s; it is closest to the models of Salvatore~\cite{salvatore} and Segal~\cite{segallocal}, but see also \cite{kallel}, \cite{bodig}, \cite{mcduff}, \cite{may}, and \cite{segal}. Aspects appear implicitly in other works, particularly Bott--Segal~\cite{bottsegal}, in realizing the Gelfand--Fuks cohomology of vector fields as the homology of a section space, and in foliation theory~\cite{mcduff2}. Factorization homology thus arises from the broad nexus of Segal's ideas on conformal field theory \cite{segalconformal}, on mapping spaces \cite{segal} and \cite{segallocal}, on foliations \cite{segal.foliations}, and on Lie algebra homology \cite{bottsegal}.\footnote{More recent works include \cite{akz}, \cite{andrade}, \cite{BZBJ1}, \cite{BZBJ2}, \cite{dck}, \cite{gtz1}, \cite{gtz2}, \cite{owen}, \cite{quoc}, \cite{klang}, \cite{ben}, \cite{ben.thesis}, among many others.}

\subsubsection{\bf Outline of contents}
We briefly outline the contents of this article.
{\bf Section 2} concerns $n$-manifolds and certain topological structures thereon. 
In {\bf~\S2.2-2.5} we discuss a classification of sheaves on an $\infty$-category $\Mfld_n$ of $n$-manifolds and embeddings among them: \emph{sheaves on $\Mfld_n$ are $n$-dimensional tangential structures}.  
In {\bf~\S2.5-2.7} we undergo a similar examination concerning \emph{Weiss sheaves}, thereby motivating and introducing an $\infty$-category $\Disk_n$, which plays a key role in factorization homology.  
In {\bf~\S2.8} we extend these developments to manifolds with boundary; in {\bf~\S2.9} we record a technical result concerning localizing with respect to isotopies, which plays a key role in essentially all of the technical results concerning factorization homology. 

{\bf Section 3} concerns homology theories for manifolds, and introduces factorization homology.  
In {\bf~\S3.2} we relate factorization homology to factorization algebras.  
The remainder of this section is devoted to establishing a pushforward formula for factorization homology, and characterizing homology theories in terms of factorization homology.
{\bf Section 4} applies this characterization to prove nonabelian Poincar\'e duality. 
{\bf Section 5} is devoted to a few formal calculations of factorization homology.  

{\bf Section 6} concerns filtrations and cofiltrations of factorization homology, whose layers are explicit in terms of configuration spaces.
These (co)filtrations offer access to identifying and controlling factorization homology.
{\bf Section 7} demonstrates this with a statement for how factorization homology intertwines Poincar\'e duality and Koszul duality.  
This \emph{Poincar\'e/Koszul duality} (Theorem~\ref{1main}) is the deepest result discussed in this article, and supplies a physical interpretation to factorization homology as observables of perturbative sigma-models. 

{\bf Section 8} is a synopsis of an adaptation of factorization homology for singular manifolds with coefficients in \emph{disk}-algebras and module data among such.

\subsubsection*{\bf Implementation of $\infty$-categories}
In this work, we use Joyal's {\it quasi-category} model  of $\oo$-category theory \cite{joyal}. 
Boardman and Vogt first introduced these simplicial sets in \cite{bv}, as weak Kan complexes, and their and Joyal's theory has been developed in great depth by Lurie in \cite{topos} and~\cite{dag}, our primary references; see the first article of \cite{topos} for an introduction. We use this model, rather than model categories or simplicial categories, because of the great technical advantages for constructions involving categories of functors, which are ubiquitous in this work.

More specifically, we work inside of the quasi-category associated to this model category of Joyal's.  In particular, each map between quasi-categories is understood to be an iso- and inner-fibration; (co)limits among quasi-categories are equivalent to homotopy (co)limits with respect to Joyal's model structure.
As we work in this way, we refer the reader to these sources for $\infty$-categorical versions of numerous familiar results and constructions among ordinary categories.  
In particular, we will make repeated use of the $\infty$-categorical adjoint functor theorem (Corollary~5.5.2.9 of~\cite{topos}); the straightening-unstraightening equivalence between Cartesian fibrations over an $\infty$-category $\cC$ and $\Cat$-valued contravariant functors from $\cC$ (Theorem~3.2.0.1 of~\cite{topos}), and likewise between right fibrations over $\cC$ and space-valued presheaves on $\cC$ (Theorem~2.2.1.2 of~\cite{topos}); the $\infty$-categorical version of the Yoneda functor $\cC\to \Psh(\cC)\simeq {\sf RFib}_\cC$ as it evaluates on objects as $c\mapsto \cC_{/c}$ (see~\S5.1 of~\cite{topos}).  

We will also make use of topological categories, such as $\mfld_n$ of $n$-manifolds and embeddings among them. By a functor $\cS \ra \cC$ from a topological category to an $\oo$-category $\cC$ we will always mean a functor $\sN\Sing \cS \ra \cC$ from the simplicial nerve of the ${\sf Kan}$-enriched category obtained by applying the product preserving functor $\Sing$ to the morphism topological spaces. 

The reader uncomfortable with this language can substitute the words ``topological category" for ``$\oo$-category" wherever they occur in this paper to obtain the correct sense of the results, but they should then bear in mind the proviso that technical difficulties may then abound in making the statements literally true. The reader only concerned with algebras in chain complexes, rather than spectra, can likewise substitute ``pre-triangulated differential graded category" for ``stable $\oo$-category" wherever those words appear, with the same proviso.

\section{Manifolds with tangential structure}
In this section, we first introduce a topological category of $n$-manifolds and embeddings among them.
We then define the \emph{tangent classifier} functor.
Using this tangent classifier, we then introduce tangential structures, and define an $\infty$-category of such structured $n$-manifolds.

\subsection{Manifolds and embeddings}

\begin{definition}\label{def.good}
A smooth manifold $M$ is \emph{finitary} if it admits a \emph{good} cover, which is to say it admits a finite open cover $\cU :=\{U\subset M\}$ with the property that, for each finite subset $S\subset \cU$, the intersection $\underset{U\in S}\bigcap U$ is either empty or diffeomorphic to Euclidean space.   
\end{definition}

\begin{observation}
A smooth manifold $M$ is finitary if and only if it is the interior of a compact smooth manifold with (possibly empty) boundary.
\end{observation}

\begin{terminology}\label{size}
In this article, by ``manifold'' we will mean ``finitary smooth manifold'', unless otherwise stated.  

\end{terminology}

\begin{remark}
The size restriction of Terminology~\ref{size} is not an essential requirement.
Indeed, each non-compact manifold is built as sequential colimit of finitary manifolds.
Correspondingly, this smallness condition could be removed and one could instead add to Definition~\ref{exc} the requirement that a homology theory preserves sequential colimits.
With this modification to that definition, the main results are still valid.  
\end{remark}

\begin{definition}\label{manifold} 
The symmetric monoidal topological category $\mfld_n$ has as objects smooth $n$-manifolds.
The space of morphisms in $\mfld_n$ from $M$ to $N$ is
\[
\Map_{\mfld_n}(M,N) :=  \Emb(M,N)
\]the set of embeddings of $M$ into $N$ equipped with the compact-open $C^\infty$ topology. 
The symmetric monoidal structure is disjoint union.
\end{definition}

Note that disjoint union is not the coproduct in $\mfld_n$; in fact, there are almost no nontrivial colimits in $\mfld_n$.

We will make ongoing use of the following result.
\begin{prop}\label{kister}
The continuous homomorphism of topological monoids $\sO(n) \to \Emb(\RR^n,\RR^n)$ is a homotopy equivalence.  
\end{prop}

\begin{proof}
This continuous homomorphism factors as a composite of continuous homomorphisms:
\[
\sO(n) \longrightarrow
{\sf GL}(n) \longrightarrow
\Emb_0(\RR^n,\RR^n) \longrightarrow
\Emb(\RR^n,\RR^n)~;
\]
where ${\sf GL}(n)$ is the topological group of linear automorphisms of $\RR^n$, and $\Emb_0(\RR^n,\RR^n)$ is the submonoid of $\Emb(\RR^n,\RR^n)$ consisting of those smooth embeddings that preserve the origin.  
The result is established upon showing each of these continuous homomoprhisms is a section of a deformation retraction.  
Well, the Gram--Schmidt process demonstrates a deformation retraction to the inclusion $\sO(n) \to {\sf GL}(n)$.
Translation $(\RR^n\xra{x\mapsto f(x)}\RR^n)\mapsto (\RR^n \xra{x\mapsto f(x)-tf(0)}\RR^n)$ demonstrates a deformation retraction to the inclusion $\Emb_0(\RR^n,\RR^n)\to \Emb(\RR^n,\RR^n)$.
The expression 
\[ f\mapsto \left\{ 
  \begin{array}{l l}
    \frac{f(tx)}{t} 
    & 
    \quad \text{for $t>0$}
    \\
    D_0f(x) 
    & 
    \quad \text{for $t=0$}
  \end{array} \right.
\] 
demonstrates a deformation retraction to the inclusion ${\sf GL}(n) \to \Emb_0(\RR^n,\RR^n)$.

\end{proof}

\noindent
Now, temporarily consider the full $\infty$-subcategory 
\[
\cE{\sf uc}_n  ~ \subset ~ \mfld_n
\]
consisting of the object $\RR^n$; this $\infty$-category is a delooping of the topological monoid $\Emb(\RR^n,\RR^n)$
We draw an immediate consequence of Proposition~\ref{kister}. First, note that there exists a natural functor
\[
\BO(n) \longrightarrow \cE{\sf uc}_n~,\qquad 
\ast\mapsto \RR^n~,
\]
from the classifying space of the orthogonal group, defined by the inclusion $\sO(n)\hookrightarrow \Emb(\RR^n,\RR^n)$.
\begin{cor}\label{euc}
The functor
\[
\BO(n) \longrightarrow \cE{\sf uc}_n~,\qquad 
\]
is an equivalence between $\infty$-categories.
In particular, there is a canonical equivalence between the $\infty$-category of presheaves
\[
\Psh(\cE{\sf uc}_n)~\simeq~\spaces_{/\BO(n)}
\]
and spaces over $\BO(n)$. 

\end{cor}

\subsection{Sheaves on $n$-manifolds}\label{sec.sheaf.classify}
Proposition~\ref{kister} offers a classification of sheaves on the $\infty$-category of $B$-framed $n$-manifolds, with respect to a standard Grothendieck topology.

\begin{definition}
The symmetric monoidal category $\dmfld_n$ has objects $n$-manifolds; its morphisms are smooth embeddings; the symmetric monoidal structure is disjoint union.  
The full subcategory ${\sf Euc}_n\subset \dmfld_n$ consists of the object $\RR^n$. In the \emph{standard} Grothendieck topology on $\dmfld_n$, a sieve $\cU \subset \dmfld_{n/M}$ is a covering sieve exactly if, for each element $x\in M$, there is an object $(U\xra{e} M)\in \cU$ for which $\{x\}\subset e(U)$.
The $\infty$-category of \emph{sheaves} on $\dmfld_n$ is the full $\infty$-subcategory
\[
\shv(\dmfld_n)~\subset~\Psh(\dmfld_n)
\]
consisting of those presheaves $\cF\colon \dmfld_n^{\op} \to \spaces$ for which, for each standard covering sieve $\cU\subset \dmfld_{n/M}$, the canonical functor
\[
(\cU^{\op})^{\triangleleft}
~\simeq~
(\cU^{\triangleright})^{\op}
\longrightarrow
(\dmfld_{n/M})^{\op}
\longrightarrow
\dmfld_n^{\op}
\xra{~\cF~}
\spaces
\]
is a limit diagram, where
\[
\cU^\triangleright:=\cU\times[1]\underset{\cU\times\{1\}}\amalg \ast
\]
is the right-cone of $\cU$ and likewise $\cU^\triangleleft$ is the left-cone.

\end{definition}

Note the evident functor between $\infty$-categories 
\[
\dmfld_n
\longrightarrow
\mfld_n
\]
defined by the natural identity mapping from sets of embeddings, with the discrete topology, to the same sets of embeddings endowed the compact-open $C^\infty$ topology. 
\begin{definition}\label{def.sheaves}
The $\infty$-category of \emph{sheaves} on $\mfld_n$ is the pullback among $\infty$-categories:
\[
\xymatrix{
\shv(\mfld_n)  \ar[rr]  \ar[d]
&&
\Psh(\mfld_n)  \ar[d]^-{\rm restriction}
\\
\shv(\dmfld_n) \ar[rr]
&&
\Psh(\dmfld_n) .
}
\]

\end{definition}

The next result serves to contextualize the recurring role of spaces over $\BO(n)$. 
This result is not essential for the overall logic of this paper, 
so we do not supply a proof.
The result and its proof is, however, analogous to (and simpler than) Proposition~\ref{weiss.classify} and its proof.
\begin{prop}[\cite{aft1}]\label{sheaf.classify}
Restriction along the functor $\BO(n) \to \mfld_n$ defines an equivalence between $\infty$-categories
\[
\shv(\mfld_n)
\xra{~\simeq~}
\Psh\bigl(\BO(n)\bigr)~\simeq~\spaces_{/\BO(n)}~.
\]

\end{prop}

\begin{remark}
Proposition~\ref{sheaf.classify} is notable in that the$\oo$-topos $\shv(\mfld_n)$ is \emph{free} on its completely compact objects, which is a connected $\infty$-groupoid: $\BO(n)$.
Consequently, a sheaf on the $\infty$-category $\mfld_n$ is uniquely determined by its value on $\RR^n$ as an $\sO(n)$-space, or equivalently a space over $\BO(n)$, without reference to a sheaf condition.  

\end{remark}

\begin{remark}
Proposition~\ref{sheaf.classify} asserts that a sheaf on $\mfld_n$ is precisely the datum of a space $B$ equipped with a map $B\to \BO(n)$, which is equivalent to a rank-$n$ vector bundle $E\to B$ over $B$.  
In this equivalence, the value of the sheaf corresponding to $B\to \BO(n)$ on an object $M\in \mfld_n$ is the space $\Map_{/\BO(n)}\bigl( M , B \bigr)$ of lifts:
\[
\xymatrix{
&&
B  \ar[d]
\\
\Emb(\RR^n,M)_{\sO(n)} \ar[rr]  \ar@{-->}[urr]^-{\varphi}
&&
\BO(n) ,
}
\]
where the bottom left space is the $\sO(n)$-coinvariants of the space of morphisms in $\mfld_n$ from $\RR^n$ to $M$.  
Corollary~\ref{tang-class} will identify the bottom horizontal arrow as the tangent classifier, $M \xra{\tau_M} \BO(n)$.
Through that identification, such a lift is equivalent data to a map $\varphi\colon M\to B$ together with an isomorphism $TM \cong \varphi^\ast E$ between vector bundles over $M$.

\end{remark}

\begin{remark}\label{4}
The canonical functor $\dmfld_n \to \mfld_n$ carries morphisms $U\hookrightarrow V$ that are isotopy equivalences to equivalences in the $\infty$-category $\mfld_n$.  
It follows that the restriction functor factors
\[
\shv(\mfld_n)\longrightarrow \shv^{\sf l.c.}(\dmfld_n)~\subset~\shv(\dmfld_n)
\]
through the $\infty$-subcategory of locally constant sheaves.
We warn the reader that this first functor is \emph{not} an equivalence.
Indeed, it follows from a result of Segal (\cite{segal.foliations}) that the shape of the$\oo$-topos $\shv(\dmfld_n)$ is the classifying space of the groupoid completion ${\sf BEmb}^\delta(\RR^n,\RR^n)$ of the discrete monoid of smooth self-embeddings of $\RR^n$.
As such, through Proposition~\ref{sheaf.classify}, this functor can be identified as
\[
\Psh\bigl( \BO(n)\bigr)
\longrightarrow
\Psh\bigl( {\sf BEmb}^\delta(\RR^n,\RR^n) \bigr)~,
\]
which is restriction along the canonical functor ${\sf BEmb}^\delta(\RR^n,\RR^n)\ra {\sf BEmb}(\RR^n,\RR^n)\simeq \BO(n)$. The work~\cite{haefliger} establishes that this map is not an equivalence.  

\end{remark}

The following result, proved in~\cite{aft1}, contrasts with Remark~\ref{4}.
Indeed, the shape of the$\oo$-topos $\shv(\dmfld_{n/M})$ is the underlying space of $M$, while $\shv(\mfld_{n/M})$ is identified, as a consequence of Corollary~\ref{tang-class}, as the $\infty$-category $\spaces_{/M}$ of local systems on $M$.
\begin{prop}\label{5}
For each $n$-manifold $M$, the restriction functor $\shv(\mfld_{n/M})\ra\shv(\dmfld_{n/M})$ factors through an equivalence between $\infty$-categories
\[
\shv(\mfld_{n/M})\xra{~\simeq~} \shv^{\sf l.c.}(\dmfld_{n/M})~.
\]

\end{prop}

\subsection{Tangent classifier}

We will be particularly interested in $n$-manifolds that are equipped with the additional structure of a section of a given sheaf on $\mfld_n$.
Through Proposition~\ref{sheaf.classify}, such a section is a continuous system of linear structures on each tangent space, such as an orientation, spin structure, or a framing.  
Tangential structure of this sort can be swiftly accommodated by way of the \emph{tangent classifier}: each $n$-manifold $M$ has a tangent bundle, and it is classified by a map $\tau_M \colon M \to \BO(n)$ to the classifying space of the topological group $\sO(n)$ of linear isometries of $\RR^n$.
For $B\to \BO(n)$ a map between spaces, a $B$-framing on $M$ is a homotopy commutative diagram among spaces
\[
\xymatrix{
&&
B  \ar[d]
\\
M \ar[rr]_-{\tau_M} \ar[urr]^-{\varphi}
&&
\BO(n) . 
}
\]

\begin{example}
Consider the surjective homomorphism
\[
\sO(n) \xra{~\rm determinant~}\sO(1)~.
\]
The kernel of this homomorphism is ${\sf SO}(n)\subset \sO(n)$, and a ${\sf BSO}(n)$-framing on a topological $n$-manifold is precisely an orientation.  

\end{example}

Toward defining an $\infty$-category $\mfld_n^B$ of $B$-framed $n$-manifolds, we next explain how to make the tangent classifier continuously functorial among open embeddings.  
Namely, through Corollary~\ref{euc}, we define the tangent classifier as the restricted Yoneda functor:
\begin{equation}\label{tau}
\tau\colon \mfld_n \xra{\rm Yoneda} \Psh(\mfld_n) \xra{\rm restriction} \Psh(\cE{\sf uc}_n)~\underset{\rm Cor~\ref{euc}}\simeq~\spaces_{/\BO(n)}~.
\end{equation}
We postpone to Corollary~\ref{tang-class} justification that the value $\tau_M$ is indeed the familiar tangent classifier, namely that
the functor $\tau$ sends a manifold $M$ to the map of spaces $M\ra \BO(n)$, homotopy-coherently among embeddings in the manifold variable.

\begin{observation}\label{3}
Because $\RR^n$ is connected, this tangent classifier $\tau\colon \mfld_n \to \spaces_{/\BO(n)}$ is symmetric monoidal with respect to coproducts in the codomain.
In other words, $\tau$ carries finite disjoint unions to finite coproducts over $\BO(n)$.

\end{observation}

\subsection{$B$-framed manifolds}

In this section, we fix a space $B$ as well as a map $B\to \BO(n)\simeq {\sf BGL}(n)$, which is equivalent to a rank-$n$ vector bundle $E\to B$ over $B$.  
Through Proposition~\ref{sheaf.classify}, such data defines a sheaf on $\mfld_n$.  
We now consider an $\infty$-category $\mfld_n^B$, of $n$-manifolds equipped with sections of this sheaf.

\begin{definition}\label{bmanifold}
The symmetric monoidal $\infty$-category $\mfld_n^B$ of $B$-framed $n$-manifolds is the limit in the following diagram:
\[
\xymatrix{
\mfld^B_n    \ar[rr]    \ar[d]
&&
\spaces_{/B}    \ar[d]
\\
\mfld_n    \ar[rr]^-\tau
&&
\spaces_{/\BO(n)}  .
}
\]
\end{definition}

With coproduct as the symmetric monoidal structure on $\spaces_{/B}$ and $\spaces_{/\BO(n)}$, the right vertical functor is canonically symmetric monoidal.  
Observation~\ref{3} grants that the bottom horizontal functor $\tau$ in the diagram in Definition~\ref{bmanifold} is canonically symmetric monoidal.
Formally, the forgetful functor ${\sf Cat}_{\infty}^{\ot} \to \Cat$ from symmetric monoidal $\infty$-categories to $\infty$-categories creates and preserves limits.
We conclude that Definition~\ref{bmanifold} indeed defines an $\infty$-category equipped with a symmetric monoidal structure, as claimed.

\begin{remark}
We describe the objects and spaces of morphism spaces in the $\infty$-category $\mfld_n^B$.
\begin{itemize}
\item
An object is a \emph{$B$-framed $n$-manifold}, which is the data of an $n$-manifold $M$ and a lift
\[
\xymatrix{
&&
B  \ar[d]
\\
M \ar[rr]_-{\tau_M} \ar[urr]^-{\varphi}
&&
\BO(n) 
}
\]
of its tangent classifier.  
Here, we understand that this is a commutative diagram in the $\infty$-category $\spaces$, i.e., a homotopy coherently commutative diagram of spaces.  
Equivalently, for $E\to B$ the rank-$n$ vector bundle over $B$ classified by the given map $B\to \BO(n)$, such a commutative diagram is the data of a map $M\xra{\varphi} B$ between spaces together with an isomorphism $TM \cong \varphi^\ast E$ between vector bundles over $M$.

\item
Let $(M,\varphi)$ and $(N,\psi)$ be $B$-framed $n$-manifolds, as above. 
The space of morphisms in $\mfld_n^B$ from $(M,\varphi)$ to $(N,\psi)$ is the space of \emph{$B$-framed embeddings}, which is the limit space:
\begin{equation}\label{1}
\xymatrix{
\Emb^B(M,N)\ar[r]\ar[d]&\Map_{/B}(M,N)\ar[d]\\
\Emb(M,N)\ar[r]&\Map_{/\BO(n)}(M,N).
}
\end{equation}
Here, for $M \xra{\varphi}X\xla{\psi}N$ a diagram in $\spaces$, then $\Map_{/X}(M,N)$ is the space of maps from $M$ to $N$ over $X$: it is the limit space:
\[
\xymatrix{
\Map_{/X}(M,N) \ar[rr] \ar[d]
&&
\Map(M,N) \ar[d]^-{-\circ \psi}
\\
\ast \ar[rr]^-{\{\varphi\}}
&&
\Map(M,X) .
}
\]
That is, a point in $\Map_{/X}(M,N)$ is represented by a map $M\ra N$ and a homotopy between the two resulting maps from $M$ to $X$

\end{itemize}

\end{remark}

Proposition~\ref{kister} yields the following.

\begin{observation} 
The space of $B$-framings on $\RR^n$ is the fiber of the given map $B\to \BO(n)$ over the point $\ast \xra{\{\RR^n\}} \BO(n)$ selecting the vector space $\RR^n$.  
For $\varphi$ a $B$-framing on $\RR^n$, there is a canonical equivalence between monoid-objects in $\spaces$:
\[
\Omega_{\varphi} B~\simeq~ \Emb^B(\RR^n, \RR^n)~;
\]
here, $\Omega_{\varphi} B$ is the monoid in $\spaces$, via concatenation, of loops in $B$ based at the point $\ast \xra{\varphi} {\sf fiber}\bigl(B\to \BO(n)\bigr) \to  B$.
\end{observation}

\subsection{Examples and discussion of $B$-framings}\label{sec.structures}
{~}

\noindent
{\bf Framings.}
In the case that $\bigl(B \to \BO(n)\bigr) := \bigl(\ast \xra{\{\RR^n\}}\BO(n)\bigr)$, then a $B$-framing on an $n$-manifold $M$ is a \emph{framing} of $M$, i.e., a trivialization of its tangent bundle.
Such a trivialization is a sequence of $n$ linearly independent vector fields $(X_1,\dots,X_n)$ on $M$.

Now, let $(M,\varphi)$ and $(N,\psi)$ be framed $n$-manifolds, with corresponding vector fields $(X_i)_{1\leq i\leq n}$ and $(Y_i)_{1\leq i\leq n}$.  
A naive definition of the space of \emph{framed embeddings} might be the subspace,
\[
\Bigl\{M\overset{e}\hookrightarrow N \mid De(X_i)=Y_i\text{ for each }1\leq i\leq n \Bigr\}
~\subset~
\Emb(M,N)~,
\]
of those smooth embeddings that strictly respect the framings.   
However, for the purposes of defining manifold invariants from algebraic input (e.g., factorization homology), this naive definition of framed embeddings is deficient:
\begin{enumerate}
\item
for $(N,\psi)$ a framed $n$-manifold with $N$ compact, it receives no framed embeddings from $\RR^n$ (with its standard framing) since it receives no isometric embeddings from $\RR^n$;

\item
for each $i\geq 0$, this space of framed embeddings from $(\RR^n)^{\sqcup i}$ to $\RR^n$ (each with its standard framing) is \emph{not} equivalent to the space $\cE_n(i)$ of $i$-ary operations of the little $n$-disk operad.  

\end{enumerate}
One might try to fix the above naive definition of framed embeddings by weakening the condition $De(X_i) = Y_i$ to the \emph{data} of a sequence of smooth functions $M\xra{\lambda_i} \RR_{>0}$ subject to the condition of an equality $De(X_i) = \lambda_i Y_i$.
But this, too, has similar shortcomings.

Practically, the essential feature of the `correct' space of framed embeddings is so that there is a homotopy equivalence
\[
\Emb^{\sf fr}\bigl( (\RR^n)^{\sqcup i} , (N,\psi)\bigr)
~\simeq~
\conf_i(N)
\]
with the \emph{configuration space} of injections from $\{1,\dots,i\}$ into $N$.
Lemma~\ref{EE-equivs}, to come, grants that this is the case for the definition~(\ref{1}) of framed embeddings (in this case that $B=\ast$) as a homotopy pullback.

\medskip

\noindent
{\bf Linear structures.}
More generally, for $G$ a Lie group and $\rho \colon G\to {\sf GL}(n)$ a smooth homomorphism, consider the composite map $\sB G \to {\sf BGL}(n) \simeq \BO(n)$.  
A $\sB G$-framing on an $n$-manifold $M$ is a compatible system of lifts $D(\alpha^{-1}\beta)\colon U\cap V\to G$ along $\rho$ of the derivatives of the transition maps of a smooth atlas for $M$.  
So a ${\sf BSO}(n)$-framing on $M$ is a smooth structure on $M$ together with an orientation on $M$; a ${\sf B\ast}$-framing on $M$ is a framing, as discussed above; a ${\sf BSpin}(n)$-framing on $M$ is a spin structure on $M$.

\medskip

\noindent
{\bf General structures.}
In general, for $B$ a space, the space of maps $B\to \BO(n)$ is a moduli space of rank-$n$ vector bundles on $B$.
So, given such a vector bundle $E\to B$, a $B$-framing on an $n$-manifold $M$ is a continuous map $f\colon M\to B$ together with an isomorphism $TM \cong f^\ast E$ between vector bundles over $M$.
In the case that $\bigl( B\to \BO(n)\bigr)\simeq \bigl( X \times \BO(n)\xra{\sf pr}\BO(n)\bigr)$, then a $B$-framing on an $n$-manifold $M$ is simply a map $M\to X$ to $X$ from the underlying space of $M$.

\subsection{Weiss sheaves on $n$-manifolds}\label{sec.weiss}

\begin{definition}[\cite{weiss}]
In the \emph{Weiss} Grothendieck topology on the category $\dmfld_n$, a sieve $\cU \subset \dmfld_{n/M}$ is a covering sieve if, for each finite subset $S\subset M$, there is an object $(U\hookrightarrow M)\in \cU$ for which $S\subset e(U)$.  
The $\infty$-category of \emph{Weiss} sheaves on $\dmfld_n$ is the full $\infty$-subcategory
\[
\shv^{\sf Weiss}(\dmfld_n)~\subset~\Psh(\dmfld_n)
\]
consisting of those presheaves $\cF\colon \dmfld_n^{\op} \to \spaces$ for which, for each Weiss covering sieve $\cU\subset \dmfld_{n/M}$, the canonical functor
\[
(\cU^{\op})^{\triangleleft}
~\simeq~
(\cU^{\triangleright})^{\op}
\longrightarrow
(\dmfld_{n/M})^{\op}
\longrightarrow
\dmfld_n^{\op}
\xra{~\cF~}
\spaces
\]
is a limit diagram.  
The $\infty$-category of \emph{Weiss} sheaves on $\mfld_n$ is the pullback among $\infty$-categories:
\[
\xymatrix{
\shv^{\sf Weiss}(\mfld_n)  \ar[rr]  \ar[d]
&&
\Psh(\mfld_n)  \ar[d]^-{\rm restriction}
\\
\shv^{\sf Weiss}(\dmfld_n) \ar[rr]
&&
\Psh(\dmfld_n) .
}
\]

\end{definition}

We next exhibit generators the for the $\oo$-category of Weiss sheaves.

\begin{definition}\label{def.discrete.versions}
The symmetric monoidal $\oo$-categories 
\[
\ddisk_n\subset \dmfld_n
\qquad
\text{ and }
\qquad
\disk_n\subset \mfld_n
\]
are the full $\oo$-subcategories consisting of disjoint unions of $n$-dimensional Euclidean spaces.

\end{definition}
Notice the evident symmetric monoidal functors
\[
\ddisk_n 
\longrightarrow 
\disk_n
\qquad
\text{ and }
\qquad
\dmfld_n
\longrightarrow
\mfld_n~.
\]

We insert the next result to contextualize the fundamental role of disks', though this result is not essential for the overall logic of this paper.
Its proof makes use of results established later on, which of course do not logically depend on it.  

\begin{prop}\label{weiss.classify}
Restriction along $\disk_n \hookrightarrow \mfld_n$ defines an equivalence between $\infty$-categories:
\[
\shv^{\sf Weiss}(\mfld_n)
\xra{~\simeq~}
\Psh(\disk_n)~.
\]

\end{prop}

\begin{proof}
Denote the fully-faithful inclusion $\iota \colon \disk_n \hookrightarrow \mfld_n$.
This functor determines an adjunction
\[
\iota^\ast \colon \Psh(\mfld_n)
\rightleftarrows
\Psh(\disk_n) \colon \iota_\ast
\]
in which the left adjoint is restriction along $\iota$ and the right adjoint is right Kan extension along $\iota$.
It is enough to show that this adjunction restricts as an equivalence:
\[
\iota^\ast \colon \shv^{\sf Weiss}(\mfld_n)
\rightleftarrows
\Psh(\disk_n)\colon \iota_\ast~.
\]
This amounts to verifying two assertions.
\begin{enumerate}
\item
For each $\cF\in \shv^{\sf Weiss}(\mfld_n)$, and each $M\in \mfld_n$, the unit map
\[
\cF(M) 
\xra{~\rm unit~}
\limit\Bigl( (\disk_{n/M})^{\op} \to \disk_n^{\op} \xra{\iota} \mfld_n^{\op}\xra{\cF} \spaces \Bigr)
\]
is an equivalence between spaces.

\item
For each $\cG\in \Psh(\disk_n)$, and for each Weiss covering sieve $\cU \subset \dmfld_{n/M}$, the canonical map
\[
\iota_\ast \cG(M)
\longrightarrow
\limit\Bigl( \cU^{\op} \to \dmfld_n^{\op} \to \mfld_n^{\op} \xra{\iota_\ast \cG} \spaces  \Bigr)
\]
is an equivalence between spaces.

\end{enumerate}
We first establish~(1).   
The canonical functor $\ddisk_{n/M} \to \disk_{n/M}$ determines the sequence of maps among spaces
\[
\xymatrix{
\cF(M)  \ar[r]^-{\rm unit}  \ar[dr]
&
\limit\bigl( (\disk_{n/M})^{\op} \to \disk_n^{\op} \xra{\iota} \mfld_n^{\op}\xra{\cF} \spaces \bigr)  \ar[d]
\\
&
\limit\bigl( (\ddisk_{n/M})^{\op} \to \ddisk_n^{\op} \to \disk_n^{\op} \xra{\iota} \mfld_n^{\op}\xra{\cF} \spaces \bigr)
.
}
\]
Because smooth open embeddings from disjoint unions of $n$-dimensional Euclidean spaces to each smooth $n$-manifold is a basis for the Weiss Grothendieck topology thereon, the full subcategory $\ddisk_n \subset \dmfld_n$ is a basis for the Weiss Grothendieck topology.  
It follows that the diagonal map in the above diagram is an equivalence between spaces.  
Furthermore, Proposition~\ref{EEd-vs-EE} implies the functor $(\ddisk_{n/M})^{\op} \to (\disk_{n/M})^{\op}$ is initial.
It follows that the downward map in the above diagram is an equivalence.
We conclude that the horizontal map in the above diagram is an equivalence, as desired.

We now establish~(2).  
The Weiss covering sieve $\cU\subset \dmfld_{n/M}$ determines a functor between $\infty$-categories
\begin{equation}\label{13}
\colim\Bigl( \cU \to \dmfld_n \xra{\disk_{n/-}} \Cat \Bigr)
\longrightarrow
\disk_{n/M}
\end{equation}
from the colimit indexed by $\cU$.  
Through the standard formula computing values of right Kan extension as limits indexed by $\infty$-undercategories, the map in~(2) is canonically identified as the map between spaces
\[
\limit\Bigl( (\disk_{n/M})^{\op} \to \disk_n^{\op} \xra{\cG} \spaces \Bigr)
\longrightarrow
\]
\[
\limit\Bigl( \colim\bigl( \cU \to \dmfld_n \xra{\disk_{n/-}} \Cat \bigr)^{\op}  \to (\disk_{n/M})^{\op} \to \disk_n^{\op} \xra{\cG} \spaces \Bigr)~.
\]
It is therefore enough to show that the functor~(\ref{13}) is final.
Observe that, for each $n$-manifold $W$, the $\infty$-category $\disk_{n/W}$ is canonically a right fibration over $\disk_n$.  
To assess finality of~(\ref{13}), it is enough to compute the relevant colimit in the $\infty$-category ${\sf RFib}_{\disk_n}$ of right fibrations over $\disk_n$, and show that the morphism between right fibrations over $\disk_n$
\[
\colim\Bigl( \cU \to \dmfld_n \xra{\disk_{n/-}} {\sf RFib}_{\disk_n} \Bigr)
\longrightarrow
\disk_{n/M}
\]
is an equivalence.  
A morphism between right fibrations over an $\infty$-category is an equivalence if and only if it is so when base changed over the maximal $\infty$-subgroupoid of the base.  
Through Lemma~\ref{EE-equivs}, the maximal $\infty$-subgroupoid of $\disk_n$ is the coproduct $\underset{i\geq 0}\coprod \sB \bigl(\Sigma_i\wr \sO(n)\bigr)$. 
For each $n$-manifold $W$, Lemma~\ref{EE-equivs} identifies the base change of $\disk_{n/M}$ over the $i$-cofactor as $\conf_i(W)_{\Sigma_i}$, the unordered configuration space. 
So let $i\geq 0$.
So we are to show that the canonical map between spaces
\begin{equation}\label{14}
\colim\Bigl( \cU \to \dmfld_n \xra{\conf_i(-)_{\Sigma_i}} \spaces_{/\sB \bigl(\Sigma_i\wr \sO(n)\bigr)} \Bigr)
\longrightarrow
\conf_i(M)_{\Sigma_i}
\end{equation}
is an equivalence between spaces over $\sB \bigl(\Sigma_i\wr \sO(n)\bigr)$.
This is simply to show that this map is an equivalence between spaces, ignorant to the structure maps to $\sB \bigl(\Sigma_i\wr \sO(n)\bigr)$.  
Consider the smallest sieve $\cU_i \subset \dmfld_{ni/\conf_i(M)_{\Sigma_i}}$ containing, for each $U\in \cU$, the object $\bigl(\conf_i(U)_{\Sigma_i} \hookrightarrow \conf_i(M)_{\Sigma_i}\bigr)\in \dmfld_{ni/\conf_i(M)_{\Sigma_i}}$.  
Precisely because $\cU$ is a Weiss covering sieve, the sieve $\cU_i$ is a standard covering sieve.
Using that the underlying topological space of $\conf_i(M)_{\Sigma_i}$ is paracompact and Hausdorff, Theorem~A.3.1~\cite{dag} can be applied to $\cU_i$ with the result that the above map~(\ref{14}) between $\infty$-groupoids is an equivalence.  
This completes the proof.

\end{proof}

\begin{definition}
Let $\cC$ be a presentable $\infty$-category.
The $\infty$-category of \emph{$\cC$-valued Weiss cosheaves} on $\dmfld_n$ is the full $\infty$-subcategory
\[
\csh^{\sf Weiss}_\cC(\dmfld_n)
~\subset~
\Fun(\dmfld_n , \cC)
\]
consisting of those functors $\cA\colon \dmfld_n \to \cC$ for which, for each Weiss covering sieve $\cU\subset \dmfld_{n/M}$, the canonical functor
\[
\cU^{\triangleright}
\longrightarrow
\dmfld_{n/M}
\longrightarrow
\dmfld_n
\xra{~\cA~}
\cC
\]
is a colimit diagram.  
The $\infty$-category of \emph{$\cC$-valued Weiss cosheaves} on $\mfld_n$ is the pullback among $\infty$-categories:
\[
\xymatrix{
\csh_\cC^{\sf Weiss}(\mfld_n)  \ar[rr]  \ar[d]
&&
\Fun(\mfld_n , \cC) \ar[d]^-{\rm restriction}
\\
\csh_\cC^{\sf Weiss}(\dmfld_n) \ar[rr]
&&
\Fun(\dmfld_n , \cC ).
}
\]

\end{definition}

The next result is routine consequence of Proposition~\ref{weiss.classify}.
\begin{cor}\label{weiss.cosheaf}
Let $\cC$ be a presentable $\infty$-category.
Restriction along $\disk_n \to \mfld_n$ defines an equivalence between $\infty$-categories:
\[
\csh_\cC^{\sf Weiss}(\mfld_n) 
\xra{~\simeq~}
\Fun(\disk_n , \cC)~.
\]

\end{cor}

\begin{remark}
Corollary~\ref{weiss.cosheaf} says that a $\cC$-valued Weiss cosheaf on $\mfld_n$ is simply a functor $\disk_n \to \cC$, without regard to a cosheaf condition.  

\end{remark}

\begin{remark}\label{60}
We follow up on Remark~\ref{4} and Remark~\ref{5}.
Namely, restriction along $\dmfld_n \to \mfld_n$ factors
\[
\csh_\cC^{\sf Weiss}(\mfld_n)\longrightarrow \csh_\cC^{\sf Weiss,l.c.}(\dmfld_n)~\subset~\csh_\cC^{\sf Weiss}(\dmfld_n)
\]
through the $\infty$-subcategory of locally constant Weiss cosheaves.
This first functor is not an equivalence.  
However, for each $n$-manifold $M$, the restriction functor factors as an equivalence between $\infty$-categories,
\[
\csh_\cC^{\sf Weiss}(\mfld_{n/M})
\xra{~\simeq~} 
\csh_\cC^{\sf Weiss,l.c}(\dmfld_{n/M})
\subset 
\csh_\cC^{\sf Weiss}(\dmfld_{n/M})~,
\]
involving the locally constant Weiss cosheaves on $M$.

\end{remark}

\subsection{Disks}

Prompted by the results in~\S\ref{sec.weiss}, 
we now turn to consider $B$-framed $n$-disks, as they organize as a symmetric monoidal $\infty$-category.  
We identify the maximal $\infty$-subgroupoid of the $\infty$-category $\disk^B_{n/M}$ in terms of configuration spaces in a $B$-framed $n$-manifold $M$.

\begin{definition}[\cite{oldfact}]
The symmetric monoidal $\oo$-category $\disk^B_n\subset \mfld_n^B$ is the full $\oo$-subcategory of $\mfld^B_n$ whose objects are disjoint unions of $B$-framed $n$-dimensional Euclidean spaces.
\end{definition}

\begin{remark}\label{2}
Consider the framing structure, which is the map $\ast\ra \BO(n)$ selecting the basepoint of $\BO(n)$. 
As discussed in~\S\ref{sec.structures}, a $\ast$-framing on an $n$-manifold $M$ is a trivialization of the tangent bundle of $M$.
We denote the associated $\oo$-category of framed $n$-disks as $\disk^{\fr}_n$. 
This symmetric monoidal $\infty$-category $\disk_n^{\fr}$ is equivalent to the PROP associated to the operad $\cE_n$, of Boardman-Vogt~\cite{bv}.
This is the case because the inclusion of rectilinear embeddings as framed embeddings determines an equivalence $\cE_n(i) \xra\sim \Emb^{\sf fr}\bigl( (\RR^n)^{\sqcup i} , \RR^n\bigr)$ from the space of $i$-ary operations of the operad $\cE_n$; see, e.g., \cite{oldfact} for a  presentation of this equivalence.  
\end{remark}

\begin{example}
Consider the structure $\bigl(B \to \BO(n)\bigr) := \bigl( \BO(n)\xra{\sf id} \BO(n)\bigr)$.
The symmetric monoidal $\infty$-category $\disk_n = \disk_n^{\BO(n)}$ is equivalent to the PROP associated to the unoriented version of the ribbon, or ``framed," $\cE_n$-operad; see~\cite{sw} for a treatment of this operad.\footnote{The historical use of ``framed" here is potentially misleading, since in the ``framed" $\cE_n$ operad the embeddings do not preserve the framing, while in the usual $\cE_n$ operad the embeddings do preserve the framing (up to scale). It might lead to less confusion to replace the term ``framed $\cE_n$ operad" with ``unoriented $\cE_n$ operad."
} 
This equivalence follows from Proposition~\ref{kister}.

\end{example}

Given a topological space $X$ and a finite cardinality $i$, the \emph{configuration space} of $i$ (ordered) points in $X$ is the subspace 
\[
\conf_i(X):=\Bigl\{\{1,\dots,i\} \xra{c} X \mid c\text{ is injective}  \Bigr\} ~\subset~ X^{\times i}~.
\]
This configuration space has an evident free action of the symmetric group $\Sigma_i$, as given by precomposition.
The \emph{unordered} configuration space is the $\Sigma_i$-coinvariants: $\conf_i(X)_{\Sigma_i}$.

In the next result, for $M$ a $B$-framed $n$-manifold, we consider the $\infty$-overcategory 
\[
\disk_{n/M}^B~:=~ \disk^B_n \underset{\mfld_n^B}\times \mfld^B_{n/M}~.
\]
An object in this $\infty$-category is given by a $B$-framed embedding $(\RR^n)^{\sqcup i} \hookrightarrow M$ for some $i$.  

\begin{lemma}[\cite{oldfact}]\label{EE-equivs}
The maximal $\infty$-subgroupoid of $\disk_n^B$ is canonically identified as the space
\[
\underset{i\geq 0} \coprod {B^{\times i}}_{\Sigma_i}~\simeq~\bigl(\disk_n^B\bigr)^\sim
\]
where the coproduct is indexed by finite cardinalities and each cofactor is the $\Sigma_i$-homotopy coinvariants of the $i$-fold product of the space $B$.  
In particular, the symmetric monoidal functor $[-]\colon \disk_n^B \to \fin$, given by taking sets of connected components of underlying manifolds, is conservative.  

For $M$ a $B$-framed $n$-manifold,
the maximal $\infty$-subgroupoid of $\disk^B_{n/M}$ is canonically identified as the space
\[
\underset{i\geq 0}\coprod \conf_i(M)_{\Sigma_i}~\simeq~\bigl(\disk^B_{n/M}\bigr)^\sim
\]
where the coproduct is indexed by finite cardinalities, and each cofactor is an unordered configuration space. 

\end{lemma}

We conclude this section by justifying the term \emph{tangent classifier} for the functor $\mfld_n \xra{\tau} \spaces_{/\BO(n)}$ from~(\ref{tau}).  
\begin{cor}\label{tang-class}
The value of the tangent classifier~(\ref{tau}) on an $n$-manifold $M$ is the map between spaces $M \xra{\tau_M} \BO(n)$ classifying its tangent bundle.  

\end{cor}

\begin{proof}
First, recognize $\cE{\sf uc}_n\subset \disk_n$ as the full $\infty$-subcategory consisting of the connected $n$-manifolds.  
Next, specialize the second statement of Lemma~\ref{EE-equivs} at $i=1$ to obtain an identification
\[
M  \underset{\rm Lem~\ref{EE-equivs}}{~\simeq~}  \cE{\sf uc}_{n/M}  \underset{\rm Prop~\ref{kister}}{~\simeq~} \Emb(\RR^n,M)_{\sO(n)} \longrightarrow \BO(n)
\]
involving the homotopy $\sO(n)$-coinvariants.
Unwinding the equivalences above recognizes this map in terms of the frame bundle for $M$, as
\[
M~\simeq~{\sf Fr}(M)_{\sO(n)} \longrightarrow \BO(n)~.
\]
Evidently, this map classifies the tangent bundle of $M$.

\end{proof}

\subsection{Manifolds with boundary}\label{sec.with.bdry}
We will also employ the category of $n$-manifolds with boundary.

\begin{definition}\label{def.bdry.good}
A smooth manifold $M$ with boundary is \emph{finitary} if it admits a \emph{good} cover, which is to say it admits a finite open cover $\cU :=\{U\subset M\}$ with the property that, for each finite subset $S\subset \cU$, the intersection $\underset{U\in S}\bigcap U$ is either empty, diffeomorphic to Euclidean space, or diffeomorphic to Euclidean half-space, $\RR_{\geq 0} \times \RR^{n-1}$.
\end{definition}

\begin{terminology}\label{size.boundary}
In this article, by ``manifold with boundary'' we mean ``finitary smooth manifold with boundary," unless otherwise stated.  

\end{terminology}

\begin{definition}  
$\mfld_n^\partial$ is the symmetric monoidal topological category of $n$-manifolds, possibly with boundary.
The topological space of morphisms between two is the set of smooth open embeddings equipped with the compact-open $C^\infty$ topology. 
The symmetric monoidal structure is disjoint union.
The full symmetric monoidal topological category
$\disk_n^\partial \subset \mfld_n^\partial$ is that consisting of finite disjoint unions of $\RR^n$ and $\RR_{\geq 0}\times \RR^{n-1}$.
\end{definition}

\begin{remark} 
The category $\disk_n^\partial$ is minimal with respect to the condition that any finite subset of an $n$-manifold with boundary has an open neighborhood diffeomorphic to an object of $\disk_n^\partial$. In particular, the closed $n$-disk $\DD^n$ is not an object of $\disk_n^\partial$.
\end{remark}

The following result is an application of the Alexander trick.

\begin{prop}[\cite{oldfact}]\label{product} 
The symmetric monoidal functor 
\[
\RR_{\geq 0}\times-:\mfld_{n-1} 
\longrightarrow 
\mfld_n^\partial
\] 
is fully-faithful. 
Namely, for each pair of $(n-1)$-manifolds $M$ and $N$, the map 
\[
\Emb(M,N) \longrightarrow \Emb\bigl(\RR_{\geq 0}\times M,\RR_{\geq 0}\times N\bigr)
\] 
is an equivalence between spaces.
\end{prop}

\begin{remark} 
Together with Proposition~\ref{kister}, the previous proposition implies that the continuous homomorphism between topological monoids
\[
\sO(n-1)
\hookrightarrow
\Emb(\RR_{\geq 0}\times \RR^{n-1}, \RR_{\geq 0}\times\RR^{n-1})
\] 
is a homotopy equivalence. 
Consequently, $\disk^\partial_n$ is an unoriented variant of the Swiss cheese operad of Voronov \cite{voronov}. 
Specifically, the framed variant $\disk_n^{\partial, \fr}$ is equivalent to the PROP associated to the Swiss cheese operad.
\end{remark}

\subsection{Localizing with respect to isotopy equivalences}\label{sec.localizing}

Here we show that the $\infty$-category $\disk^B_{n/M}$ is a localization of its more discrete version $\ddisk^B_{n/M}$ on the collection of those embeddings that are isotopic to isomorphisms.  
This comparison plays a fundamental role in recognizing certain colimit expressions in this theory, for instance those that support the pushforward formula of~\S\ref{sec.push}.

Recall Definition~\ref{def.discrete.versions}, and the symmetric monoidal functors:
\[
\ddisk_n \longrightarrow \disk_n\qquad \text{ and }\qquad  \dmfld_n \longrightarrow \mfld_n
\]
We denote the pullback symmetric monoidal $\infty$-categories:
\[
\xymatrix{
\ddisk_n^B  \ar[d]  \ar[r]  
&
\disk_n^B  \ar[d]
&&
\dmfld_n^B \ar[r]  \ar[d]
&
\mfld_n^B  \ar[d]
\\
\ddisk_n  \ar[r]
&
\disk_n
&\text{ and }&
\dmfld_n \ar[r]
&
\mfld_n.
}
\]
For each $n$-manifold $M$,
we denote the $\infty$-subcategory 
\begin{equation}\label{eqn.I_X}
\cI_M~\subset ~\ddisk^B_{n/M}:= \ddisk_n^B \underset{\dmfld_n^B}\times \dmfld^B_{n/M}
\end{equation}
that consists of the same objects but only those morphisms $(U\hookrightarrow M) \hookrightarrow (V\hookrightarrow M)$ whose image in $\disk^B_{n/M}$ is an equivalence.

\begin{prop}[\cite{oldfact}]\label{EEd-vs-EE}
The functor
$
\ddisk^B_{n/M} \longrightarrow \disk^B_{n/M}
$
witnesses a localization between $\infty$-categories:
\[
\bigl(\ddisk^B_{n/M}\bigr)[\cI_M^{-1}]~\simeq~\disk^B_{n/M}~.
\]

\end{prop}

Let $M$ be a $B$-framed $n$-manifold.
Each of the $\infty$-categories $\ddisk^B_{n/M}$ and $\disk^B_{n/M}$ is naturally the active $\infty$-subcategory of an $\infty$-operad,\footnote{Recall, \S2 of \cite{dag}, that a morphism of finite based sets $I_\ast \xra{f} J_\ast$ is {\it active} if $f^{-1}\{\ast\} = \{\ast\}$, and is {\it inert} if the restriction $f_|: f^{-1}J \ra J$ is injective. A morphism in an $\oo$-operad $\cO \ra {\sf Fin}_\ast$ is active or inert if its image in ${\sf Fin}_\ast$ is.} each of which we again denote as $\ddisk^B_{n/M}$ and $\disk^B_{n/M}$, respectively.
\begin{itemize}
\item
The $\infty$-operad structure on $\ddisk^B_{n/M}$ is such that the $\infty$-category of colors is the poset in which an object is an open subset $U\subset M$ that is abstractly diffeomorphic to Euclidean space.
There is a unique $i$-ary morphism from an $i$-fold collection $(U_k)_{1\leq k\leq i}$ of such to another $V$ precisely if the $U_k$ are pairwise disjoint and their union $\bigcup_{1\leq k \leq i} U_k \subset V$ is contained in $V$.  

\item
The $\infty$-operad structure on $\disk^B_{n/M}$ is such that the $\infty$-category of colors is the $\infty$-groupoid $\Emb(\RR^n,M)_{\sO(n)}\simeq M$ of an open disk in $M$.
The space of $i$-ary morphisms from an $i$-fold collection $(U_j)_{1\leq j \leq i}$ of such to another $V$ is the fiber of the composite map
\[
\Emb\Bigl( \bigsqcup_{j=1}^i U_j , V \Bigr)  
\longrightarrow
\prod_{j=1}^i \Emb(U_j , V)
\longrightarrow
\prod_{j=1}^i \Emb(U_j , M)
\]
over the point selecting the given sequence of embeddings $(U_j)_{1\leq j\leq i}$.  

\end{itemize}

The next result also appears in~\cite{dag} as Theorem~5.4.5.9.
\begin{cor}\label{6}
Let $\cV$ be a symmetric monoidal $\infty$-category.
For each $B$-framed $n$-manifold $M$, restriction along the morphism $\ddisk^B_{n/M} \to \disk^B_{n/M}$ defines a fully-faithful functor
\[
\Alg_{\disk^B_{n/M}}(\cV)
\xra{~\sf f.f.~}
\Alg_{\ddisk^B_{n/M}}(\cV)~.
\]

\end{cor}

In the case $M=\RR^n$, Corollary~\ref{6} can be used to prove that locally constant factorization algebras on $\RR^n$ are equivalent to $\cE_n$-algebras---see \S\ref{sec.factalg} for a discussion of factorization algebras.

The next construction is made possible using Proposition~\ref{EEd-vs-EE}.
\begin{construction}\label{def.f-inverse}
Let $M$ be a $B$-framed $n$-manifold, and let $N$ be a $B'$-framed $k$-manifold possibly with boundary.
Let $f\colon M \to N$ be a continuous map that satisfies the following regularity condition:
\begin{itemize}
\item[~]
Each of the restrictions
\[
f_|\colon f^{-1}(N\smallsetminus \partial N) \to N\smallsetminus \partial N\qquad \text{ and }\qquad f_|\colon f^{-1}(\partial N) \to \partial N
\]
is a smooth fiber bundle.  
\end{itemize}
We produce a composite morphism between $\infty$-operads
\[
f^{-1} \colon \ddisk^{\partial, B'}_{k/N} \longrightarrow \dmfld^B_{n/M} \longrightarrow \mfld^B_{n/M}~,
\]
in which the second morphism is the standard one.
We now describe the first functor.
For formal reasons, we can assume the maps $B\ra\BO(n)$ and $B'\ra\BO(k)$ are equivalences.  
For this case, the first functor is given by $(U\hookrightarrow N)\mapsto (U\underset{N}\times M \hookrightarrow M)$, which is evidently a multi-functor.
By inspection, this functor $f^{-1}$ carries isotopy equivalences to equivalences.
Through Proposition~\ref{EEd-vs-EE}, there results a multi-functor
\begin{equation}\label{f-inv}
f^{-1} \colon \disk^{B'}_{k/N} \longrightarrow \mfld^B_{n/M}~,
\end{equation}
as desired.

\end{construction}

\section{Homology theories for manifolds}\label{sec.homology}

The value of factorization homology over an $n$-manifold $M$ of an $n$-disk algebra $A$ is a sort of average, indexed by $n$-disks embedded in $M$, of the value of $A$ on such $n$-disks.
We make this definition precise, as well as observe a property of this definition for which it is universal, by defining factorization homology as left Kan extension of an $n$-disk algebra $A\colon \disk_n \to \cV$ along the inclusion $\disk_n \hookrightarrow \mfld_n$.  
\\

\noindent
Through out this entire section, we fix a space $B\to \BO(n)$ over $\BO(n)$ as well as a symmetric monoidal $\infty$-category $\cV$ that is {\it $\ot$-presentable} in the following sense.
\begin{definition}\label{ast} 
A symmetric monoidal $\oo$-category $\cV$ is \emph{$\ot$-presentable} if it satisfies both of the following conditions.
\begin{itemize}
\item 
The underlying $\infty$-category of $\cV$ is presentable: $\cV$ admits colimits and every object is a filtered colimit of compact objects.\footnote{This is with respect to an understood fixed uncountable cardinal $\kappa$, i.e., $\cV$ admits $\kappa$-small colimits and every object is a $\kappa$-filtered colimit of $\kappa$-compact objects.}

\item 
The symmetric monoidal structure of $\cV$ distributes over colimits: for each object $V\in \cV$, the functor $V\ot -\colon \cV\to \cV$ carries colimit diagrams to colimit diagrams.
\end{itemize}
\end{definition}

\begin{example}
Let $\cS$ be a presentable $\infty$-category.
Consider the Cartesian symmetric monoidal $\infty$-category $(\cS,\times)$, in which the symmetric monoidal structure is categorical product.  
Provided $\cS$ is Cartesian closed, then this symmetric monoidal $\infty$-category is $\ot$-presentable.  
In particular, $(\spaces,\times)$ is $\ot$-presentable; for $\cX$ any $\infty$-topos, then $(\cX,\times)$ is $\ot$-presentable; also, $(\Cat,\times)$ is $\ot$-presentable.
\end{example}

\begin{example}
Let $R$ be a commutative ring.
The symmetric monoidal $\infty$-category $\bigl({\sf Mod}_R, \underset{R}\ot\bigr)$, of $R$-modules with tensor product over $R$ is $\ot$-presentable.
Note, however, that its opposite $\bigl({\sf Mod}_R^{\op}, \underset{R}\ot\bigr)$ is not $\ot$-presentable.

\end{example}

\begin{remark} 
The results in \S\ref{sec.homology} (in particular, the Eilenberg--Steenrod axioms for factorization homology) only require that the symmetric monoidal structure of $\cV$ distributes over sifted colimits.
(This generality is established as a special case of~\S2 of~\cite{aft2}.)
However, the calculations of \S\ref{sec.calculations} onwards require the symmetric monoidal structure to distribute over all colimits, so for simplicity of exposition we enforce this stronger hypothesis throughout.
\end{remark}

\subsection{Disk algebras}

\begin{definition} 
The $\infty$-category of \emph{$\disk_n^B$-algebras} in $\cV$ is that of symmetric monoidal functors from $\disk_n^B$ to $\cV$:
\[
\Alg_{\disk_n^B}(\cV)~:=~\Fun^\ot(\disk^B_n, \cV)~.
\]

\end{definition}

\begin{remark}\label{framed.En}
Let $G\to \sO(n)$ be a morphism between group-objects in the $\infty$-category $\spaces$ (equivalently, a map of loop spaces).
Through this representation, change-of-framing defines an action of $G$ on the symmetric monoidal $\infty$-category $\disk_n^{\fr}$.  
As such, the symmetric monoidal forgetful functor
\[
\disk_n^{\fr}
\longrightarrow
\disk_n^{\sB G}
\]
witnesses the $G$-coinvariants: $\bigl(\disk_n^{\fr}\bigr)_{G}\xra{\simeq} \disk_n^{\sB G}$.
Consequently, for each symmetric monoidal $\infty$-category $\cV$, the restriction functor canonically factors through the $G$-invariants,
\[
\Alg_{\disk_n^{\sB G}}(\cV)
\xra{~\simeq~}
\bigl(\Alg_{\cE_n}(\cV)\bigr)^{G}
\longrightarrow
\Alg_{\cE_n}(\cV)
\]
as an equivalence.
In particular, there is a canonical equivalence from the $\infty$-category of $\disk_n$-algebras 
\[
\Alg_{\disk_n}(\cV)
\xra{~\simeq~}
\bigl(\Alg_{\cE_n}(\cV)\bigr)^{\sO(n)}
\]
and that of $\sO(n)$-invariant $\cE_n$-algebras.

\end{remark}

We denote the restricted Yoneda functor
\[
\EE\colon \mfld_n^B \xra{~\rm Yoneda~}  \Psh(\mfld_n^B) \xra{~\rm restriction~} \Psh(\disk_n^B)~.
\]
\begin{definition}\label{coend}
Let $M$ be a $B$-framed $n$-manifold.
Let $A$ be a $\disk_n^B$-algebra in $\cV$.  
\emph{Factorization homology (of $M$ with coefficients in $A$)} is the object in $\cV$ given either as the colimit (provided it exists), or the coend (provided it exists):
\begin{eqnarray}
\nonumber
\int_M A
&
:=
&
\colim\bigl(\disk_{n/M}^B \to \disk_n^B \xra{A} \cV\bigr)
\\
\nonumber
&
\simeq
&
\EE_M\underset{\disk_M^B} \bigotimes A~.
\end{eqnarray}

\end{definition}

\begin{remark} 
We comment on the two equivalent expressions defining factorization homology: as a colimit indexed by an overcategory, and as a coend.  
This is analogous to the familiar fact that, for $X_\bullet\colon \bdelta^{\op}\to \spaces$ a simplicial space, its geometric realization can be expressed 
\[
|X_\bullet|~:=~\colim\bigl(\bdelta_{/X_\bullet} \to \bdelta \xra{\Delta^\bullet} \spaces\bigr)~\simeq~ X_\bullet \underset{\bdelta}\bigotimes \Delta^\bullet
\]
as a colimit of topological simplices indexed by the category of simplices in $X_\bullet$, or as the standard coend expression which is a quotient of $\underset{p\geq 0} \coprod X_p\times \Delta^p$.
\end{remark}

The fully-faithful symmetric monoidal functor $\iota\colon \disk_n^B \hookrightarrow \mfld_n^B$ gives the restriction functor
\[
\Alg_{\disk_n^B}(\cV){:=} \Fun^{\ot}(\disk_n^B,\cV)
~\longleftarrow~  
\Fun^\ot\bigl(\mfld_n^B,\cV\bigr)\colon \iota^\ast~.
\]
The next result, proved in~\cite{oldfact}, identifies factorization homology as the values of a left adjoint to this restriction functor, provided $\cV$ is $\ot$-presentable.  

\begin{prop}[\cite{oldfact}]\label{as-LKE}
Let $\cV$ be $\ot$-presentable $\infty$-category.
The restriction functor $\iota^\ast$ admits a left adjoint,
\[
\iota_!\colon \Alg_{\disk_n^B}(\cV)~ \rightleftarrows~ \Fun^\ot\bigl(\mfld_n^B,\cV\bigr)\colon \iota^\ast~,
\]
over a left adjoint $\iota_!\colon \Fun(\disk_n^B,\cV) \rightleftarrows \Fun(\mfld_n^B,\cV)\colon \iota^\ast$.  
Furthermore, this left adjoint evaluates on a $\disk_n^B$-algebra $A$ as factorization homology:
\[
\iota_!(A)\colon M\mapsto \int_M A~.
\]
\end{prop}

\begin{remark}
Proposition~\ref{as-LKE} implies factorization homology can be expressed as a symmetric monoidal left Kan extension, at least when $\cV$ is $\ot$-presentable. 
Consequently, the definition of factorization homology given above is equivalent to operadic left Kan extension (after parsing Definitions 3.1.1.2 and 3.1.2.2 of \cite{dag}), which is the definition of factorization homology, or topological chiral homology, given by Lurie~\cite{dag} (Definition~5.5.2.6).

\end{remark}

The next result explains how factorization homology transforms under change of tangential structure.  
\begin{prop}\label{without-B}
Let $B\xra{\alpha} B'$ be a map between spaces over $\BO(n)$.
Let $M = (M,\varphi)$ be a $B$-framed $n$-manifold.
Consider the $B'$-framed $n$-manifold $\alpha M:= (M,\alpha\varphi)$.
Let $A$ be a $\disk_n^{B'}$-algebra $A\colon \disk_n^{B'}\to \cV$.
Consider the $\disk_n^B$-algebra $\alpha A \colon \disk_n^B \to \disk_n^{B'} \xra{A}\cV$.  
The canonical morphism in $\cV$
\[
\int_{M} \alpha A
\xra{~\simeq~}
\int_{\alpha M} A
\] 
is an equivalence.  
\end{prop}
\begin{proof} 
Note the canonical commutative diagram among $\infty$-categories:
\[
\xymatrix{
\disk_{n/M}^B \ar[rr] \ar[d]
&&
\disk_{n/\alpha M}^{B'} \ar[d]
\\
\disk_n^B \ar[rr] \ar[dr]_-{\alpha A}
&&
\disk_n^{B'} \ar[dl]^-{A}
\\
&
\cV
&
.
}
\]
For formal reasons, the top horizontal functor is an equivalence between $\infty$-categories.
It follows that the canonical morphism in $\cV$
\[
\int_{M} \alpha A
~:=~
\colim\Bigl( \disk_{n/M}^B \to \disk_n^B \xra{\alpha A} \cV\Bigr)
\longrightarrow
\colim\Bigl( \disk_{n/\alpha M}^{B'}\to \disk_n^{B'} \xra{A} \cV\Bigr) 
~=:~
\int_{\alpha M} A
\]
is an equivalence, as desired.

\end{proof}

\subsection{Factorization algebras}\label{sec.factalg}
We now use factorization homology to construct {\it factorization algebras} over each $B$-framed $n$-manifold, as in the sense used by Costello--Gwilliam (see Definition~\ref{def.fact.alg}). 
Namely, Proposition~\ref{compare.fact.alg} shows that $\disk_n^B$-algebra $A$ in $\cV$ determines a factorization algebra $\cF_A$ on each $B$-framed $n$-manifold $M$ whose value $\cF_A(U)$ on an open subset $U\subset M$ is factorization homology of $A$,
\[
\cF_A(U)~\simeq~ \int_U A~,
\]
over $U$ (with its $B$-framing inherited from that of $M$).  

For the next definition, for each $n$-manifold $M$, we observe the following multi-category structure on the poset ${\sf Opens}(M):=\dmfld_{n/M}$ of open subsets in $M$ ordered by inclusion.
Namely, an object is an open subset of $M$, and there is a multi-morphism from $(U_i)_{i\in I}$ to $V$, which is unique, provided the collection $\{U_i\}_{i\in I}$ is pairwise disjoint and provided $\underset{i\in I}\bigcup U_i\subset V$.  
Recall from~\S\ref{sec.weiss} the definition of (locally constant) Weiss cosheaves.  
\begin{definition}[\cite{kevinowen}]\label{def.fact.alg}
Let $\cV$ be a symmetric monoidal $\infty$-category.
Let $M$ be an $n$-manifold.
The $\infty$-category of ($\cV$-valued) \emph{factorization algebras} (on $M$) is the full $\infty$-subcategory of algebras in $\cV$ over the multi-category ${\sf Opens}(M)$ as in the pullback among $\infty$-categories:
\[
\xymatrix{
{\sf Alg}_M(\cV)   \ar[rr] \ar[d]
&&
\Alg_{{\sf Opens}(M)}(\cV)  \ar[d]
\\
\csh_\cV^{\sf Weiss}(M)  \ar[rr]
&&
\Fun\bigl({\sf Opens}(M) , \cV\bigr) .
}
\]
The $\infty$-category of \emph{locally constant} factorization algebras is the full $\infty$-subcategory of factorization algebras as in the pullback diagram among $\infty$-categories:
\[
\xymatrix{
{\sf Alg}^{\sf l.c.}_M(\cV)   \ar[rr] \ar[d]
&&
{\sf Alg}_M(\cV)  \ar[d]
\\
\csh_\cV^{\sf Weiss,l.c.}(M)  \ar[rr]
&&
\csh_\cV^{\sf Weiss}(M) .
}
\]
\end{definition}

\begin{remark}
In other words, a factorization algebra is a multi-functor $\cF\colon {\sf Opens}(M) \to \cV$ whose restriction to the poset ${\sf Opens}(M)$ is a Weiss cosheaf on $M$.
Informally, a factorization algebra is likewise a functor $\cF\colon {\sf Opens}(M) \to \cV$ from the poset of open subsets of $M$ to the underlying $\infty$-category of $\cV$ that satisfies codescent with respect to Weiss covers, 
together with a system of compatible equivalences in $\cV$: for each finite sequence $(U_i)_{i\in I}$ of pairwise disjoint open subsets of $M$, an equivalence $\cF\bigl(\underset{i\in I} \bigcup U_i\bigr) \simeq \underset{i\in I} \bigotimes \cF(U_i)$.  

\end{remark}

\begin{prop}\label{compare.fact.alg}
Let $\cV$ be a $\ot$-presentable $\infty$-category, and let $B$ be a space over $\BO(n)$.
For each $B$-framed $n$-manifold $M$, factorization homology defines a functor
\[
\Alg_{\disk_n^B}(\cV)
\longrightarrow
{\sf Alg}^{\sf l.c.}_M(\cV)
~,\qquad
A\mapsto \Bigl(~ U\mapsto \int_U A~\Bigr)~
\]
from $\disk_n^B$-algebras to locally constant factorization algebras on $M$.

\end{prop}

\begin{proof}
Notice the evident diagram among $\infty$-operads
\[
\disk_n^B
\xra{~\iota~}
\mfld_n^B
\longleftarrow
\dmfld_n^B
\longleftarrow
{\sf Opens}(M) .
\]
This diagram determines the diagram among $\infty$-categories:
\[
\xymatrix{
\Alg_{\disk_n^B}(\cV) \ar[d]_-{\rm forget}
&&
\Fun^{\ot}(\mfld_n^B,\cV) \ar[ll]_-{\iota^\ast} \ar[r] \ar[d]_-{\rm forget}
&
\Fun^{\ot}(\dmfld_n^B,\cV) \ar[r] \ar[d]^-{\rm forget}
&
\Alg_{{\sf Opens}(M)}(\cV) \ar[d]^-{\rm forget}
\\
\Fun(\disk_n^B,\cV) 
&&
\Fun(\mfld_n^B,\cV) \ar[r] \ar[ll]_-{\iota^\ast}
&
\Fun(\dmfld_n^B,\cV) \ar[r]
&
\Fun\bigl({\sf Opens}(M), \cV\bigr) .
}
\]
Proposition~\ref{as-LKE} gives the commutative diagram involving left adjoints to each instance of $\iota^\ast$:
\[
\xymatrix{
\Alg_{\disk_n^B}(\cV) \ar[d]_-{\rm forget} \ar[rr]^-{\iota_!}
&&
\Fun^{\ot}(\mfld_n^B,\cV) \ar[r] \ar[d]_-{\rm forget}
&
\Fun^{\ot}(\dmfld_n^B,\cV) \ar[r] \ar[d]^-{\rm forget}
&
\Alg_{{\sf Opens}(M)}(\cV) \ar[d]^-{\rm forget}
\\
\Fun(\disk_n^B,\cV) \ar[rr]^-{\iota_!}
&&
\Fun(\mfld_n^B,\cV) \ar[r] 
&
\Fun(\dmfld_n^B,\cV) \ar[r]
&
\Fun\bigl({\sf Opens}(M), \cV\bigr) .
}
\]
Corollary~\ref{weiss.cosheaf} gives the factorization of the bottom horizontal sequence of functors:
\[
\xymatrix{
\Alg_{\disk_n^B}(\cV) \ar[d]_-{\rm forget} \ar[rrrr]^-{{\rm restrict}~\circ~ \iota_!}
&&
&&
\Alg_{{\sf Opens}(M)}(\cV) \ar[d]^-{\rm forget}
\\
\Fun(\disk_n^B,\cV) \ar[rr]^-{\iota_!}
&&
\csh^{\sf Weiss}_\cV(\mfld_n^B) \ar[r]
&
\csh^{\sf Weiss,l.c.}_\cV(\dmfld_n^B) \ar[r]
&
\csh^{\sf Weiss,l.c.}_\cV(M) .
}
\]
The result follows.

\end{proof}

\begin{remark}
Proposition~\ref{compare.fact.alg} grants the commutative diagram among $\infty$-categories:
\[
\xymatrix{
\Alg_{\disk_n}(\cV) \ar[rr]_-{\rm Prop~\ref{compare.fact.alg}}^-{A\mapsto \cF_A} \ar[dr]_-{\int_M}
&&
\Alg_M(\cV) \ar[dl]^-{\rm global~cosections}
\\
&
\cV
&
.
}
\]
In other words, for $\cF_A$ the factorization algebra determined by the $\disk_n$-algebra $A$, its global cosections is factorization homology:
\[
\int_M A~\simeq~ \cF_A(M)~.
\]
In the case that $M$ is equipped with a framing, $A$ need only be a $\cE_n$-algebra (see Remark~\ref{framed.En}).

\end{remark}

\begin{terminology}\label{d1}
For $X$ a stratified space, a factorization algebra $\cF$ on $X$ is \emph{constructible} if, for each stratum $X_p\subset X$ of $X$, the restricted factorization algebra $\cF_{|X_p}$ on $X_p$ is locally constant.  
\end{terminology}

\begin{remark}\label{r1}
Let $X$ be a stratified space.  
Consider the $\infty$-category $\cB_X$ of singularity-types in $X$, and stratified open embeddings among them.  
Consider the symmetric monoidal $\infty$-category $\disk(\cB_X)$ in which an object is a finite-fold disjoint union of objects in $\cB_X$, and a morphism between them is a stratified open embedding.  
Established in~\cite{aft2} is a similar construction to that of Proposition~\ref{compare.fact.alg} as it concerns symmetric monoidal functors from $\disk(\cB_X)$ to constructible factorization algebras on $X$.
See that reference for a thorough development, and~\S\ref{sec.singular} of this article for a synopsis.  

\end{remark}

\subsection{Factorization homology over oriented $1$-manifolds with boundary}\label{sec.interval}
We show that factorization homology over a closed interval is a two-sided bar construction.

Recall from
~\S\ref{sec.with.bdry} the symmetric monoidal $\infty$-category $\mfld_1^\partial$ and its symmetric monoidal full $\infty$-subcategory $\disk_1^\partial$ consisting of finite disjoint unions of Euclidean and half-Euclidean spaces.  
In this subsection we consider, similarly, the symmetric monoidal $\infty$-category and its symmetric monoidal full $\infty$-subcategory,
\[
\disk_1^{\partial,\sf or}
~\subset~
\mfld_1^{\partial, \sf or}~.
\]
An object in the latter is a 1-manifold with boundary equipped with an orientation, while such an object belongs to the smaller if each connected component of its underlying 1-manifold with boundary is diffeomorphic to Euclidean space or half-Euclidean space.
The space of morphisms between two such objects therein is the space of smooth open embeddings that preserve orientations, equipped with the compact-open $C^\infty$ topology.
The symmetric monoidal structure is disjoint union.

\begin{remark}
The 1-manifold with boundary $[-1,1]$, equipped with its standard orientation determined by the non-vanishing vector field $\partial_t$, is an object in $\mfld_1^{\partial,\sf or}$ that does \emph{not} belong to $\disk_1^{\partial,\sf or}$.  
\end{remark}

We next articulate a sense in which $\disk_1^{\partial, \sf or}$ is entirely combinatorial.

\begin{definition}[\cite{aft2}]\label{Assoc}
${\sf Assoc}^{\sf RL}$ is the $\infty$-operad corepresenting triples $(A;P,Q)$ consisting of an associative algebra together with a unital right and a unital left module.  
Specifically, it is a unital multi-category whose space of colors is the three-element set $\{M, R, L\}$, and with spaces of multi-morphisms given as follows.  Let $I\xra{\sigma}\{M,R,L\}$ be a map from a finite set.
\begin{itemize}
\item ${\sf Assoc}^{\sf RL}(\sigma, M)$ is the set of linear orders on $I$ for which no element is related to an element of $\sigma^{-1}(\{R,L\})$.
In other words, should $\sigma^{-1}(\{R,L\})$ be empty, then there is one multi-morphism from $\sigma$ to $M$ for each linear order on $\sigma^{-1}(M)$; should $\sigma^{-1}(\{R,L\})$ not be empty, then there are no multi-morphisms from $\sigma$ to $M$.

\item ${\sf Assoc}^{\sf RL}(\sigma, L)$ is the set of linear orders on $I$ for which each element of $\sigma^{-1}(L)$ is a minimum, and no element is related to an element in $\sigma^{-1}(R)$.  
In other words, should $\sigma^{-1}(\{R\})$ be empty and $\sigma^{-1}(\{L\})$ have cardinality at most $1$, then there is one multi-morphism from $\sigma$ to $M$ for each linear order on $\sigma^{-1}(M)$;  should $\sigma^{-1}(\{R\})$ not be empty or $\sigma^{-1}(\{L\})$ have cardinality greater than $1$, then there are no multi-morphisms from $\sigma$ to $M$.

\item ${\sf Assoc}^{\sf RL}(\sigma, R)$ is the set of linear orders on $I$ for which each element of $\sigma^{-1}(R)$ is a maximum, and no element is related to an element in $\sigma^{-1}(L)$.
In other words, should $\sigma^{-1}(\{L\})$ be empty and $\sigma^{-1}(\{R\})$ have cardinality at most $1$, then there is one multi-morphism from $\sigma$ to $M$ for each linear order on $\sigma^{-1}(M)$;  should $\sigma^{-1}(\{L\})$ not be empty or $\sigma^{-1}(\{R\})$ have cardinality greater than $1$, then there are no multi-morphisms from $\sigma$ to $M$.

\end{itemize}
Composition of multi-morphisms is given by concatenating linearly ordered sets.  

\end{definition}

The next result references the \emph{symmetric monoidal envelope} of the colored operad ${\sf Assoc}^{\sf RL}$.  
It is initial among symmetric monodal $\infty$-categories equipped with an ${\sf Assoc}^{\sf RL}$-algebra.  In other words, it is a symmetric monoidal $\infty$-category ${\sf Env}\bigl( {\sf Assoc}^{\sf RL}\bigr)$ corepresenting the copresheaf on the $\infty$-category of symmetric monoidal $\infty$-categories
\[
\Map^{\ot}\Bigl({\sf Env}({\sf Assoc}^{\sf RL}) , -\Bigr) 
\colon 
\Cat^{\ot}
\longrightarrow
\Spaces
~,\qquad
\cV\mapsto 
\bigl(\Alg_{{\sf Assoc}^{\sf RL}}(\cV)\bigr)^{\sim}
~,
\]
whose value on a symmetric monoidal $\infty$-category is the moduli space of ${\sf Assoc}^{\sf RL}$-algebras in it.

\begin{lemma}\label{disk.env}
Taking connected components defines an equivalence between symmetric monoidal $\infty$-categories,
\begin{equation}\label{disk-comb}
[-]\colon \disk^{\partial, \sf or}_1  \xra{~\simeq~}  {\sf Env}\bigl({\sf Assoc}^{\sf RL}\bigr)
\end{equation}
to the symmetric monoidal envelope,
where the values on the symmetric monoidal generators are: $[\RR] = M$, $[\RR_{\geq 0}] = R$, and $[\RR_{\leq 0}] = L$.  
\end{lemma}

\begin{proof}
Evidently, this defines a symmetric monoidal functor.
We now show that it is an equivalence.
Because it is so on symmetric monoidal generators,~(\ref{disk-comb}) is essentially surjective on spaces of objects.
It remains to show that~(\ref{disk-comb}) is fully-faithful.
So let $U$ and $V$ be objects in $\disk^{\partial,\sf or}_1$.
We must show that the map between spaces
\begin{equation}\label{33}
\Map_{\disk^{\partial,\sf or}_1}(U,V)
\longrightarrow
\Map_{{\sf Env}({\sf Assoc}^{\sf RL})}([U],[V])
\end{equation}
is an equivalence.  
For $V = \underset{\alpha\in [V]} \bigsqcup V_\alpha$ the partition as connected components, direct inspection of the definition of these two symmetric monoidal $\infty$-categories yields an identification of this map is as the $[V]$-indexed product of such maps
\begin{eqnarray}
\nonumber
\Map_{\disk^{\partial,\sf or}_1}(U,V)
&
\xra{\simeq}
&
\underset{[U]\xra{f} [V]}\coprod \underset{\alpha\in [V]}\prod\Map_{\disk^{\partial,\sf or}_1}(U_{|f^{-1}\alpha},V_\alpha)
\\
\nonumber
&
\longrightarrow
&
\underset{[U]\xra{f} [V]} \coprod \underset{\alpha\in [V]}\prod \Map_{{\sf Env}({\sf Assoc}^{\sf RL})}(f^{-1}\alpha,[V_\alpha])
\\
\nonumber
&
\xla{\simeq}
&
\Map_{{\sf Env}({\sf Assoc}^{\sf RL})}([U],[V])~.
\end{eqnarray}
We are therefore reduced to the case that $V$ is non-empty and connected.
By direct inspection, the space of morphisms in $\disk_1^{\partial,\sf or}$ from $U$ to $V$ is a $0$-type.
So we are left to show that the map~(\ref{33}) is a bijection (on connected components).  
There are three cases to consider.
\begin{itemize}
\item 
In the case that $V\cong (-1,1)$ is oriented-diffeomorphic to an open interval, this $0$-type is empty if $U$ has non-empty boundary, and otherwise it is the set of linear orders on the set $[U]$ of connected components. 

\item
In the case that $V \cong [-1,1)$ is oriented-diffeomorphic to a left-cosed/right-open interval, this $0$-type is empty if $U$ has non-empty outward-pointing boundary, and otherwise it is the set of linear orders on $[U]$ for which each connected component with inward-pointing boundary is a minimum.  

\item
In the case that $V \cong (-1,1]$ is oriented-diffeomorphic to a left-open/right-closed interval, this $0$-type is empty if $U$ has non-empty inward-pointing boundary, and otherwise it is the set of linear orders on $[U]$ for which each connected component with outward-pointing boundary is a maximum.  

\end{itemize}
Inspecting Definition~\ref{Assoc}, and the description of the symmetric monoidal functor $[-]$ under examination, reveals that the map~(\ref{33}) between 0-types is an equivalence, as desired.

\end{proof}

\begin{example}\label{ex.assoc}
Note the functor $\Alg^{\sf aug}_{\sf Assoc}(\cV) \to \Alg_{\sf Assoc^{RL}}(\cV)$ from augmented associative algebras, given by $(A\to \uno)\mapsto (A;\uno,\uno)$.  
Concatenating with the equivalence of Lemma~\ref{disk.env} results in a functor $\Alg^{\sf aug}_{\sf Assoc}(\cV) \to\Fun^\ot\bigl(\disk^{\partial, \sf or}_1, \cV\bigr)$.

\end{example}

Recall from~\S\ref{sec.with.bdry} the $\infty$-operad structure on $\disk_{n/M}$.
\begin{observation}\label{interval-assoc}
Consider the ordinary category $\sO^{\sf RL}$ in which an object is a finite linearly ordered set $(I,\leq)$ together with a pair of disjoint subsets $R\subset I \supset L$ for which each element in $R$ is a minimum and each element in $L$ is a maximum; a morphism $(I,\leq, R,L)\to (I',\leq',R',L')$ is an order preserving map $I\xra{f} I'$ for which $f(R\sqcup L)\subset R'\sqcup L'$.  
Concatenating linear orders endows $\sO^{\sf RL}$ with the structure of a multi-category.
Note the evident forgetful morphism between multi-categories $\sO^{\sf RL} \to {\sf Env}\bigl({\sf Assoc}^{\sf RL}\bigr)$.  
By direct inspection, the equivalence~(\ref{disk-comb}) lifts to an equivalence between $\infty$-operads:
\begin{equation}\label{interval-comb}
\xymatrix{
\disk^{\partial,\sf or}_{1/[-1,1]}  \ar@{-->}[rr]^-{\simeq}_-{[-]}  \ar[d]_-{\rm forget}
&&
\sO^{\sf RL}  \ar[d]
\\
\disk^{\partial,\sf or}_1  \ar[rr]^-{\simeq}_-{[-]}
&&
{\sf Env}({\sf Assoc}^{\sf RL} ) .}
\end{equation}

\end{observation}

\begin{cor}\label{tensor-prod}
Let $(A;P,Q)$ be an ${\sf Assoc}^{\sf RL}$-algebra in $\cV$; which is to say an associative algebra $A$ together with a unital left and a unital right $A$-module.
Through Lemma~\ref{disk.env}, regard $(A;P,Q)$ as a symmetric monoidal functor $(A;P,Q)\colon\disk^{\partial, \sf or}_1\to \cV$. 
There is a canonical equivalence in $\cV$ from the balanced tensor product to factorization homology over the closed interval:
\[
Q\underset{A}\ot P\xra{~\simeq~} \int_{[-1,1]} (A;P,Q)~.  
\]

\end{cor}

\begin{proof}
Recognize the opposite of the simplex category as the full subcategory $\bdelta^{\op}\subset \sO^{\sf RL}$ consisting of those objects $(I,\leq , R, L)$ for which $R\neq \emptyset \neq L$.  
Adjoining minima and maxima gives a left adjoint $\sO^{\sf RL} \to \bdelta^{\op}$ to the inclusion.
Therefore, the inclusion $\bdelta^{\op} \to \sO^{\sf RL}$ is final.
Concatenating this final functor with the equivalence of Observation~\ref{interval-assoc} results in a final functor $\bDelta^{\op} \to \disk^{\partial, {\sf or}}_{1/[-1,1]}$.
By inspection, the resulting simplicial object 
\[
{\sf Bar}_\bullet\bigl(Q,A,P\bigr)\colon \bDelta^{\op} \to \disk^{\partial, {\sf or}}_{1/[-1,1]} \to \disk^{\partial, \sf or}_1 \xra{[-]}{\sf Env}\bigl({\sf Assoc}^{\sf RL}\bigr) \xra{(A;P,Q)} \cV
\]
is identified as the two-sided bar construction, as indicated. 
We conclude the equivalence in $\cV$:
\[
Q\underset{A}\ot P~ \simeq~ \bigl| {\sf Bar}_{\bullet}(Q,A,P)\bigr|
\xra{~\simeq~} 
\colim\bigl(\bDelta^{\op} \xra{{\sf Bar}_\bullet\bigl(Q,A,P\bigr)} \cV\bigr)\xra{~\simeq~} \int_{[-1,1]}(A;P,Q)~.
\]

\end{proof}

Note that the given multi-functor $\sO^{\sf RL} \to {\sf Env}({\sf Assoc^{RL}})$ witness an equivalence ${\sf Env}(\sO^{\sf RL}) \xra{\simeq} {\sf Env}({\sf Assoc^{RL}})$ between symmetric monoidal $\infty$-categories.  
After Observation~\ref{interval-assoc}, this offers the following consequence, which refers to Terminology~\ref{d1}.
\begin{cor}\label{assoc-interval}
Factorization homology defines an equivalence of $\infty$-categories,
\[
\int\colon \Alg_{{\sf Assoc}^{\sf RL}}(\cV) \xra{~\simeq~} \Alg_{\disk^{\partial, \sf or}_{1/[-1,1]}}(\cV)~,
\] 
between associative algebra equipped with a unital left and right module and constructible factorization algebras over the closed interval.  
\end{cor}

\subsection{Homology theories: definition}
We now define homology theories for $B$-framed $n$-manifolds.

\begin{definition}[Collar-gluing]\label{def.collar-gluing}
Let $M$ be a $B$-framed $n$-manifold.
A \emph{collar-gluing} of $M$ is a continuous map
\[
f\colon M \to [-1,1]
\]
to the closed interval for which the restriction $f_|\colon M_{|(-1,1)}\to (-1,1)$ is a smooth fiber bundle.  
We will often denote a collar-gluing $M\xra{f}[-1,1]$ simply as the open cover
\[
M_- \underset{M_0\times \RR}\bigcup M_+~\cong~M~,
\]
where $M_-= f^{-1}[-1,1)$ and $M_+=f^{-1}(-1,1]$ and $M_0 = f^{-1}\{0\}$.  

\end{definition}

\begin{remark}
We think of a collar-gluing of $M$ as a codimension-1 properly embedded submanifold $M_0\subset M$ whose complement is partitioned by connected components: $M\smallsetminus M_0 = M_- \sqcup M_+$. 
Such data is afforded by gluing two manifolds with boundary along a common boundary.  
The actual data of a collar-gluing specifies that named just above, in addition to a bi-collaring of the common boundary.
\end{remark}

Construction~\ref{def.f-inverse} and the results of~\S\ref{sec.interval} give the following.
\begin{cor}\label{excision-arrow}
Let $\cF\colon \mfld_n^B \longrightarrow \cV$ be a symmetric monoidal functor.  
Let $M$ be a $B$-framed $n$-manifold.
A collar-gluing $M_- \underset{M_0\times \RR}\bigcup M_+ \cong M$ determines an associative algebra $\cF(M_0)$ together with a unital left module structure on the object $\cF(M_+)$ and a unital right module structure on $\cF(M_-)$, as well as a morphism in $\cV$:
\begin{equation}\label{exc-compare}
\cF(M_-)\underset{\cF(M_0)} \bigotimes \cF(M_+) \longrightarrow \cF(M)~.
\end{equation}

\end{cor}

\begin{proof}
The collar-gluing is a continuous map $M\xra{f}[-1,1]$.
Construction~\ref{def.f-inverse} gives the the first morphism in the composible sequence of morphisms among of $\infty$-operads:
\[
f_\ast \cF\colon \disk^{\partial, {\sf or}}_{1/[-1,1]} \xra{f^{-1}} \mfld^B_{n/M} \to \mfld_n^B \xra{\cF} \cV~.
\]
Through Corollary~\ref{assoc-interval}, this morphism between $\infty$-operads is equivalent to an algebra in $\cV$ over the operad ${\sf Assoc}^{\sf RL}$.
Unwinding that equivalence reveals that its underlying associative algebra is $\cF(M_0)$, whose underlying object is the value $\cF(M_0\times \RR)$ and whose associative algebra structure is given by oriented embeddings among the $\RR$-coordinate, and its underlying unital modules are the values $\cF(M_\pm)$.  
Furthermore, the definition of factorization homology as a colimit supplies the canonical morphism in $\cV$:
\[
\cF(M_-) \underset{\cF(M_0)}\bigotimes \cF(M_+)~{}~ \underset{\rm Cor~\ref{tensor-prod}}\simeq ~{}~\int_{[-1,1]} f_\ast \cF~{}~ \longrightarrow ~{}~ \cF\bigl(f^{-1}([-1,1])\bigr) =  \cF(M)~.
\]

\end{proof}

\begin{definition}\label{exc} 
A symmetric monoidal functor $\cF:\mfld_n^{B}\ra \cV$ satisfies \emph{$\ot$-excision} if, for each collar-gluing $M_- \underset{M_0\times \RR}\bigcup M_+\cong M$ of $B$-framed $n$-manifolds, the canonical morphism in $\cV$,
\[
\cF(M_-)\underset{\cF(M_0\times\RR)}\bigotimes\cF(M_+)\underset{(\ref{exc-compare})}{\xra{~\simeq~}}\cF(M)~,
\] 
is an equivalence.
The $\oo$-category of $\cV$-valued \emph{homology theories} for $B$-framed $n$-manifolds is the full $\oo$-subcategory 
\[
\bH(\mfld_n^{B},\cV)~\subset~\Fun^\ot(\mfld^B_n, \cV)
\]
consisting of those symmetric monoidal functors that satisfy $\ot$-excision.
\end{definition}

\begin{remark} 
We emphasize that a $\cV$-valued homology theory depends on the symmetric monoidal structure $\ot$ of $\cV$.  
For instance, let $\Bbbk$ a field and consider the $\infty$-category $\m_{\Bbbk}$ of $\Bbbk$-modules.
There are two natural symmetric monoidal structures on $\m_{\Bbbk}$: direct sum $\oplus$; tensor product $\underset{\Bbbk}\ot$.  
As established in~\S\ref{sec.direct.sum}, a homology theory valued in $(\m_{\Bbbk},\oplus)$ evaluates on an $n$-manifold $M$ as the $\Bbbk$-chains, $\sC_\ast(M;\Bbbk)$.
On the other hand, in~\S\ref{sec.commutative}, specifically
~Remark~\ref{config.spaces}, shows that a homology theory valued in $(\m_{\Bbbk},\underset{\Bbbk}\ot)$ typically does \emph{not} factor through the forgetful functor $\mfld_n \to \spaces$ to the underlying homotopy type of manifolds.  
\end{remark}

\subsection{Pushforward}\label{sec.push}
We prove that factorization homology satisfies $\ot$-excision in the sense of Definition~\ref{exc}.
We realize this as an instance of a general construction of a pushforward.

The following result is the technical crux of the later results of this article; through this result one can access the values of factorization homology.  
We state the the result now, and prove it at the end of this section.

\begin{lemma}[\cite{cotangent}, \cite{oldfact}]\label{excision} 
Let $M$ be a $B$-framed $n$-manifold, and let $A$ be a $\disk_n^B$-algebra in $\cV$, for $\cV$ a $\ot$-presentable $\infty$-category.
For each collar-gluing $M_-\underset{M_0\times \RR}\bigcup M_+ \cong M$, the canonical morphism in $\cV$,
\[
\int_{M_-}A\bigotimes_{\displaystyle\int_{M_0\times\RR}A}\int_{M_+}A\underset{(\ref{exc-compare})}{\xra{~\simeq~}} \int_MA
\] 
is an equivalence.

\end{lemma}

Lemma~\ref{excision} specializes to to the following special case of $n=1$, $B\simeq \ast$, and $M=S^1$, which demonstrates the utility of $\ot$-excision.

\begin{cor}[\cite{oldfact}, \cite{dag}]\label{HH} 
Let $A$ be an associative algebra in a $\ot$-presentable $\infty$-category $\cV$. 
There is a canonical equivalence between objects in $\cV$:
\[
\hh_\bullet(A)
~\simeq~ 
\int_{S^1}A
\] 
between the Hochschild homology of $A$ and factorization homology of $A$ over the circle.

\end{cor}
\begin{proof} 
Consider the collar-gluing $\RR^{\op}\underset{\RR^{\op} \sqcup \RR}\bigcup \RR \cong S^1$, where the `op' superscripts denote the opposite orientation from the standard orientation on Euclidean space, which is the map ${\sf prom}\colon S^1 \to [-1,1]$ which projects unit vectors in $\RR^2$ to the first coordinate.  
Lemma~\ref{excision}, applied to this collar-gluing, states the last of the equivalences in the sequence of equivalences in $\cV$:
\[
\hh_\bullet(A)
~\simeq~
A\underset{A\otimes A^{\op}}\bigotimes A
~\simeq~
\int_{\RR}A\underset{{\displaystyle\int_{{S^0\times\RR}}\! A}}\bigotimes\int_{\RR}A
\xra{~\simeq~}
\int_{S^1}A  ~.
\]
The middle equivalence is by inspecting values, while the first is definitional. 

\end{proof}

The next definition makes use of the multi-functor $f^{-1} \colon \disk^{\partial}_{k/N} \to \mfld^B_{n/M}$ of Construction~\ref{def.f-inverse} associated to each continuous map $M \ra N$ for which each of the restrictions, $M_{|N\smallsetminus \partial N} \to N\smallsetminus \partial N$ and $M_{|\partial N} \to \partial N$, are smooth manifold bundles.  
\begin{definition}[\cite{oldfact}]\label{disk-f} Let $M$ be an $B$-framed $n$-manifold, and let $N$ be an oriented $k$-manifold, possibly with boundary. For $f:M\ra N$ a map such that the restrictions of $f$ over both the interior of $N$ and the boundary of $N$ are fiber bundles, then the $\oo$-category $\disk_f$ is the limit of the diagram among $\infty$-categories
\[
\xymatrix{
\disk^B_{n/M}\ar[dr]&&\ar[ld]^{{\sf ev}_0}{\sf Ar}(\mfld^B_{n/M})\ar[rd]_{{\sf ev}_1}&&\ar[dl]^{f^{-1}}\disk_{k/N}^{\partial}\\
&\mfld^B_{n/M}&&\mfld^B_{n/M}
&
,
}
\] 
where ${\sf Ar}(\mfld^B_{n/M})$ is the $\oo$-category of functors $[1]\to \mfld^B_{n/M}$. 
\end{definition}

\begin{remark}
Informally, an object in the $\infty$-category $\disk_f$ is a triple $(V,U, \beta )$ for which: $U$ is an open subset of $N$ that is diffeomorphic to a disjoint union of Euclidean spaces; $V$ is an open submanifold of $M$ that is diffeomorphic to a disjoint union of Euclidean spaces; and $\beta$ is an isotopy among embeddings of $V$ into $M$ from the given inclusion to one that factors through $f^{-1}U\hookrightarrow M$.

\end{remark}

The following result is stated and proved in~\S3 of~\cite{oldfact}.
Its proof amounts to a partition-of-unity argument (specifically Theorem 1.1 \cite{dugger-isaksen} or Theorem~A.3.1~\cite{dag}), applied to unordered configuration spaces.  
\begin{lemma}[\cite{oldfact}]\label{final} In the situation of Definition~\ref{disk-f}, the functor 
\[
{\sf ev}_0:\disk_f \ra \disk_{n/M}^B
~,\qquad
(U,V,\beta)\mapsto V~,
\]
is final.
\end{lemma}

Lemma~\ref{final} applied to the fold map $M\sqcup M \to M$ gives the following important technical property to the $\infty$-category of disks in $M$.
\begin{cor}[Corollary~3.22 in~\cite{oldfact}, Proposition~5.5.2.16 in~\cite{dag}]\label{disk-sifted}
For $M$ a $B$-framed $n$-manifold, the $\oo$-overcategory $\disk_{n/M}^B$ is sifted.
\end{cor}

We now apply Lemma~3.32 to factorization homology in the next result. In the statement of this result we use the notation
\[
f_\ast A:\xymatrix{\disk_{k/N}^{\partial}\ar[r]^{f^{-1}}& \mfld_{n/M}^B\ar[r]^-{\int A}& \cV}
\] 
for the composite functor, where $f^{-1}$ is as in Construction~\ref{def.f-inverse}.

\begin{prop}[Proposition~3.23 in \cite{oldfact})]\label{pushforward} 
Fix a map of spaces $B\to \BO(n)$. Let $M$ be a $B$-framed manifold, and let $A$ be a $\disk_n^B$-algebra in a $\ot$-presentable $\infty$-category $\cV$.
Let $f\colon M\to N$ be a continuous map to a smooth $k$-manifold with boundary whose restriction over the boundary and interior of $N$ is a smooth fiber bundle.  
There is a canonical map in $\cV$,
\[
\int_{N}f_\ast A
:=
\colim\Bigl(
\disk_{k/N}^{\partial}
\xra{~f_\ast A~}
\cV\Bigr)
\xra{~\simeq~}
\int_MA~,
\] 
which is an equivalence.
\end{prop}

\begin{proof}[Proof of Lemma~\ref{excision}] 
Let $f:M\ra [-1,1]$ be a collar-gluing.
There are canonical morphisms in $\cV$ 
\[
\int_{M_-}A\bigotimes_{\displaystyle\int_{M_0\times\RR}A}\int_{M_+}A
\underset{\rm Cor~\ref{tensor-prod}}{\xra{~\simeq~}}
\int_{[-1,1]} f_\ast A
\underset{\rm Prop~\ref{pushforward}}{\xra{~\simeq~}} \int_M A
~. 
\]
Corollary~\ref{tensor-prod} states that the first arrow is an equivalence.
Proposition~\ref{pushforward} states that the second morphism is an equivalence.

\end{proof}

\subsection{Homology theories: characterization}\label{sec.hmlgy}
We characterize factorization homology, analogously to Eilenberg--Steenrod's characterization of homology theories for spaces.

\begin{theorem}[\cite{oldfact}]\label{homology} 
Let $\cV$ be $\ot$-presentable $\infty$-category.
Let $B\to \BO(n)$ be a map from a space.
There is a canonical equivalence between $\infty$-categories,
\[
\displaystyle\int: \Alg_{\disk_n^B}(\cV)
\overset{\simeq}{~\rightleftarrows~}
\bH(\mfld_n^B, \cV):{\sf ev}_{\RR^n}~,
\]
between $\disk^B_n$-algebras in $\cV$ and homology theories for $B$-framed $n$-manifolds with coefficients in $\cV$. 
This equivalence is implemented by the factorization homology functor $\int$ and the functor of evaluation on $\RR^n$.

\end{theorem}

\begin{proof}
Proposition~\ref{as-LKE} identifies factorization homology as the values of symmetric monoidal left Kan extension, thereby implementing the left adjoint in an adjunction 
\[
i_!\colon \Alg_{\disk_n^B}(\cV)
~ \rightleftarrows ~  
\Fun^\ot\bigl( \mfld_n^B, \cV\bigr)\colon i^\ast
~.
\]
The unit of this adjunction is an equivalence because $\disk_n^B\ra \mfld_n^B$ is fully-faithful, and Kan extension along a fully-faithful functor restricts as the original functor.  

Now let $\cF\colon \Mfld_n^B \to \cV$ be a symmetric monoidal functor.
The counit of this adjunction evaluates on $\cF$ as a morphism $\int A \ra \cF$, where $A=\cF_{|\RR^n}$ is the $\disk^B_n$-algebra defined by the values of $\cF$ on disjoint unions of $B$-framed Euclidean $n$-spaces. 
It remains to verify that this counit is an equivalence.
To that end, consider the full $\infty$-subcategory $\cM\subset \mfld_n^B$ consisting of those objects $M$ for which this counit $\int_MA \xra{\simeq} \cF(M)$ is an equivalence.  
We wish to show the inclusion $\cM \hookrightarrow \mfld_n^B$ is an equivalence.

Both $\int A$ and $\cF$ are symmetric monoidal functors.
So the full $\infty$-subcategory $\cM\subset \mfld_n^B$ is symmetric monoidal.
By definition, $\cM$ contains $B$-framed Euclidean spaces. 
We conclude the containment $\disk_n^B\subset \cM$.

Now, for each $0\leq k\leq n$, consider the base change $B_k$ of $B\to \BO(n)$ along the map $\BO(k)\times \BO(n-k) \xra{\oplus} \BO(n)$ then projected $B_k \to \BO(k)\times \BO(n-k)\xra{\sf proj}\BO(k)$.  
We will show, by induction on $k$, that $\cM$ contains the image of the evident functor $\mfld_k^{B_k}\to \mfld_n^B$.  
The previous paragraph gives that this is the case for $k=0$.

By assumption, $\cF$ satisfies $\ot$-excision, while Lemma~\ref{excision} states that $\int A$ satisfies $\ot$-excision as well.  
We conclude that $\cM$ is closed under collar-gluings. 
Therefore $\cM$ contains the image of any $B_k$-framed $k$-manifold that can be witnessed by a finite iteration of collar-gluings from disjoint unions of $B_{k-1}$-framed $(k-1)$-manifolds.  
By induction on $k$, we are therefore reduced to showing each $B_k$-framed $k$-manifold can be witnessed as a finite iteration of collar-gluings involving $B_k$-framed $k$-manifolds in the image of $\mfld_{k-1}^{B_{k-1}} \to \mfld_k^{B_k} \to \mfld_n^B$.  

Well, the reigning assumption that manifolds be finitary guarantees, for each $B_k$-framed $k$-manifold $M$, the existence of a compactification $\overline{M}$ of the underlying manifold of $M$ as a smooth $k$-manifold with boundary.
Choose a proper Morse function $f\colon \overline{M} \to [-R,R]$ for which $f^{-1}(\{\pm R\}) = \partial \overline{M}$.  
Such a Morse function witnesses $M$ as a finite collar-gluing involving $B_k$-framed $k$-manifolds of whose underlying manifold is diffeomorphic to $M_0\times \RR$ for some $(k-1)$-manifold $M_0$.  
Such a $(k-1)$-manifold is therefore isomorphic to one that is equipped with a $B_{k-1}$-framing.

\end{proof}

\begin{remark} 
In~\cite{aft2} a version of Theorem~\ref{homology} is established for structured singular manifolds equipped with a stratified tangential structure -- 
this will be surveyed in~\S\ref{sec.singular}.  
The result, as it appears there, is slightly more general than Theorem~\ref{homology} above: the assumption that $\cV$ be $\ot$-presentable can be weakened to the assumption that the underlying $\infty$-category of $\cV$ is presentable yet the symmetric monoidal structure distributes only over geometric realizations and filtered colimits (i.e., over sifted colimits).  

\end{remark}

\begin{remark}
In~\cite{oldfact} a version of Theorem~\ref{homology} is established for topological manifolds in place of smooth manifolds.  
In essence, the only feature of the smooth category used in the proof of Theorem~\ref{homology} was the existence of open handlebody decompositions, granted by Morse functions.
But such open handlebody decompositions are guaranteed in the topological setting, through the results of~\cite{moise} for $n=3$, of~\cite{quinn} as it concerns the complement of a point for $n=4$ together with smoothing theory~\cite{kirbysieb}, of~\cite{quinn} for $n=5$, and of~\cite{kirbysieb} for $n>5$.

\end{remark}

\medskip

For much of this article, we will primarily concern ourselves with homology theories in the sense of Definition~\ref{exc}. 
We point out, however that there are very interesting functors symmetric monoidal functors $\mfld_n\to \cV$ that are not $\ot$-excisive. 
For instance, as detailed in~\cite{qcloops}, for $X$ a derived (commutative) stack over a fixed field $\Bbbk$, cotensoring to $X$ followed by taking ring of functions defines a symmetric monoidal functor:
\[
\mfld_n \xra{~M\mapsto X^M~} {\sf Stacks}(\Bbbk)\xra{~\cO~} \m_{\Bbbk}~
\]
As the case $n=1$, the value of this functor on the circle is the Hochschild homology of $X$: $\cO(X^{S^1})\simeq \hh_\bullet(X)$ (compare with Corollary~\ref{HH}).
If $X$ is not affine, the above functor will generically fail to be $\ot$-excisive. 
Now, the cotensor $X^M$ only depends on the underlying homotopy type of $M$, as we shall see in Proposition \ref{tensor}.
This is a feature of $X$ being a derived \emph{commutative} stack.
Should $X$ be a derived stack over $\disk_n$-algebras as in~\cite{thez}, there is a refined version of the cotensor $X^M$ that is sensitive to the manifold structure of $M$.   
Derived stacks over $\disk_n$-algebras arise as the outcome of deformation quantization of a shifted symplectic derived (commutative) stacks, such as in Rozansky--Witten theory.

\section{Nonabelian Poincar\'e duality}\label{sec.nonabel}

Applying Theorem \ref{homology}, we offer a slightly different perspective, and proof, of the nonabelian Poincar\'e duality of Salvatore \cite{salvatore}, Segal~\cite{segallocal}, and Lurie \cite{dag}, which calculates factorization homology of iterated loop spaces as compactly-supported mapping spaces.

\begin{definition} For a space $B$, $\Space_B:=(\spaces_{/B})^{B/}$ is the $\infty$-category of {\it retractive} spaces over $B$, i.e., spaces over $B$ equipped with a section. The $\oo$-category $\Space^{\geq n}_B$ is the full $\oo$-subcategory of $\Space_B$ consisting of those $X\rightleftarrows B$ for which the retraction is $n$-connective, that is, the homomorphism $\pi_q X \ra \pi_q B$ is an isomorphism for $q< n$ with any choice of base-point of $B$.\end{definition}

\begin{remark}
Equivalently, an object $X\in \Space_B^{\geq n}$ may be thought of as a fibration $X\ra B$ equipped with a section for which, for each $b\in B$, the fiber $X_b$ is $n$-connective. 
\end{remark}

\begin{definition} 
Fix $X\in \Space_B$.
Let $M$ be a manifold, equipped with a map $M\to B$.
The space of \emph{compactly-supported sections} of $X$ over $M$ is the colimit in the $\infty$-category $\spaces$:
\[
\Gammac(M,X)
~:=~
\colim_{K\underset{\rm compact}\subset M}
\Map_{/B}\bigl( (M\smallsetminus K \to M ) , (B\to X) \bigr)~,
\]
where the colimit is indexed by the filtered poset of compact codimension-0 submanifold with boundary in $M$, and where the space indexed by a compact subset $K\subset M$ is the pullback:
\[
\xymatrix{
\Map_{/B}\bigl( (M\smallsetminus K \to M ) , (B\to X) \bigr)  \ar[rr]  \ar[d]
&&
\Map_{/B}( M , X)  \ar[d]
\\
\ast \simeq \Map_{/B} ( M \smallsetminus K , B)  \ar[rr]
&&
\Map_{/B}( M\smallsetminus K , X)
}
\]

\end{definition}

For $B\ra \BO(n)$ a map between spaces, note that $\Gammac(-,X)$ defines a covariant functor $\mfld_n^B \ra \Space$.
By inspection, this functor carries finite disjoint unions to finite products of spaces, which is to say that $\Gammac(-,X)$ is symmetric monoidal with respect to the Cartesian monoidal structure on the $\infty$-category $\spaces$ of spaces.
In particular, the restriction
\begin{equation}\label{def.Omega.B}
\Omega_B^n X \colon \disk_n^B 
\hookrightarrow
\mfld_n^B
\xra{~\Gammac(-,X)~}
\spaces
\end{equation}
is a symmetric monoidal functor, i.e., a $\disk_n^B$-algebra in $\spaces$.  
Each point $\ast\xra{\{b\}}B$ determines a $B$-framed $n$-dimensional vector space $(V,b)$, which can be regarded as an object $(V,b)\in \disk_n^B$.
The value of the above functor $\Omega_B^n X$ on this object is equivalent to the $n$-fold loop space $\Omega^n X_b$ of the fiber $X_b$ of the map $X\to B$ over $b\in B$.

We introduce the following terminology in order to state Theorem~\ref{spaces}.
Consider the base point $\ast \xra{\{\RR^n\}} \BO(n)$.
Each lift of this point $\ast \xra{\{b\}}B$, which is a point in the fiber ${\sf fiber}\bigl(B\to \BO(n)\bigr)$, canonically determines a symmetric monoidal functor $\disk_n^{\fr} \to \disk_n^B$.
Thereafter, each symmetric monoidal functor $A\colon \disk_n^B \to \spaces$ restricts as an associative monoid
\[
\pi_0(A,b)\colon \disk_1^{\fr} \xra{\RR^{n-1}\times -} \disk_n^{\fr} \longrightarrow \disk_n^B \xra{~A~} \spaces \xra{~\pi_0~} {\sf Set} ~.
\]
We say $A$ is \emph{group-like} if, for each point $b\in {\sf fiber}\bigl(B\to \BO(n)\bigr)$, this monoid $\pi_0(A,b)$ is a group.

\begin{example}
In the case $B\simeq \ast$, a $B$-framing is a framing in the standard sense and $\disk_n^B = \disk_n^{\sf fr}$. A symmetric monoidal functor $A\colon \disk_n^{\sf fr} \to \spaces$ is then the data of an $\cE_n$-algebra, and this $\cE_n$-algebra is group-like in the sense of \cite{may} if and only if $A$ is group-like in the the above sense.

\end{example}

\begin{theorem}[\cite{oldfact}]\label{spaces} 
Restricting the functor $\Gammac: \Space_B \ra \Fun^\ot(\mfld_n^B, \Space)$ defines a fully-faithful inclusion 
\[
\Space_B^{\geq n} ~\hookrightarrow~ \bH(\mfld_n^B,\Space)
\] 
of $n$-connective retractive spaces over $B$ into homology theories. The essential image consists of those $\cF$ for which the restriction $\cF_{|\disk_n^B}$ is group-like.
\end{theorem}

\begin{remark}
The fully-faithfulness of the above functor is given by Theorem 5.1.3.6 of~\cite{dag}. This is a parametrized form of May's theorem from \cite{may}, identifying $n$-connective spaces as $\disk^{\fr}_n$-algebras.
\end{remark}

The following result is the technical crux in the proof of Theorem~\ref{spaces}; it asserts that, for each $n$-connective retractive space $X\rightleftarrows B$, the assignment of compactly-supported sections $\Gammac(-,X)$ is $\ot$-excisive.

\begin{lemma}[\cite{oldfact}]\label{gamma-c} 
Let $X\in \Space^{\geq n}_B$ be an $n$-connective retractive space over $B$. 
Let $M$ be $B$-framed $n$-manifold.
For $M'\underset{M_0\times\RR}\bigcup M''\cong M$ a collar-gluing, the canonical map between spaces,
\[
\Gammac(M',X)\underset{\Gammac(M_0\times\RR,X)}\times\Gammac(M'',X)
\xra{~\simeq~}
\Gammac(M, X)~,
\] 
from the quotient of the product $\Gammac(M',X)\times\Gammac(M'',X)$ by the diagonal action of $\Gammac(M_0\times\RR,X)$, is an equivalence.

\end{lemma}

\begin{proof}[Proof given in~\cite{oldfact}]

Since $M_0\hookrightarrow M$ is a proper embedding, a compactly-supported section over $M$ can be restricted to obtain a compactly-supported section over $M_0$, as well as over $M\smallsetminus M'$ and over $M\smallsetminus M''$.
Namely, there is a diagram among spaces of compactly-supported sections
\[
\xymatrix{
\Gammac(M',X) \times \Gammac(M'',X)  \ar[dr]  \ar[rr]  \ar[dd]
&
&
\Gammac(M'',X)  \ar[d]
\\
&
\Gammac(M,X)  \ar[r]  \ar[d]  \ar[dr]
&
\Gammac(M\smallsetminus M',X)  \ar[d]
\\
\Gammac(M',X)  \ar[r]
&
\Gammac(M\smallsetminus M'',X)  \ar[r]
&
\Gammac(M_0,X).
}
\]
By inspection, the bottom horizontal sequence is a fiber sequence, as is the right vertical sequence, as is the diagonal sequence.
Also, the inner square is pullback because $M'\underset{M_0\times \RR}\bigcup M''\cong M$ is a pushout.  
Because $M_0\subset M$ is equipped with a regular neighborhood, these fiber sequences are in fact Serre fibration sequences, and so the inner square is a weak homotopy pullback square.  
In particular, there is a right homotopy coherent action of $\Omega \Gammac(M_0,X)$ on $\Gammac(M',X)$, a left homotopy coherent action of $\Omega \Gammac(M_0,X)$ on $\Gammac(M'',X)$, and a continuous map of topological spaces
\begin{equation}\label{gammas}
\Gammac(M',X)\underset{\Omega \Gammac(M_0,X)} \times \Gammac(M'',X)\longrightarrow \Gammac(M,X)
\end{equation}
from the balanced homotopy coinvariants.  
Because $X\to B$ is $n$-connective and $M_0$ is $(n-1)$-dimensional, the base $\Gammac(M_0,X)$ is connected.
It follows that the map~(\ref{gammas}) is in fact a weak homotopy equivalence.  
The assertion follows after the canonical identification $\Omega \Gammac(M_0,X) \cong \Gammac(M_0\times \RR,X)$ as group-like $\cE_1$-spaces.

\end{proof}

Theorem~\ref{homology} implies the nonabelian Poincar\'e duality theorem of Salvatore \cite{salvatore}, Segal~\cite{segallocal}, and Lurie \cite{dag}. First, recall from~(\ref{def.Omega.B}) the $\disk_n^B$-algebra $\Omega_B^n X$ determined by a retractive space $X\rightleftarrows B$.

\begin{cor}[Nonabelian Poincar\'e duality \cite{oldfact}]\label{non-abel} 
Let $B\to \BO(n)$ be a map from a space.
Let $X\rightleftarrows B$ be an $n$-connective retractive space over $B$.
For each $B$-framed $n$-manifold $M$, there is a canonical equivalence 
\[
\int_M \Omega_B^n X 
\xra{~ \simeq ~}
\Gammac(M,X)
\] 
from the factorization homology over $M$ of $\Omega^n_{B}X$ to the space of compactly-supported sections of $X$ over $M$.
\end{cor}

We now explain how this result specializes to familiar Poincar\'e duality between homology and compactly-supported cohomology.
Let $A$ be an abelian group.
Consider the product space $X\simeq \BO(n)\times K(A,n)$.
Observe that the values of the $n$-disk algebra
\[
\Omega^n_{\BO(n)} X
\colon 
\disk_n
\longrightarrow
\spaces
~,\qquad
\underset{0\leq i\leq k} \bigsqcup V_i
\mapsto 
\prod_{1\leq i\leq k}\Map_\ast \bigl( V_i^+ , K(A,n) \bigr)
\underset{\rm noncanonical}{~\simeq~}
A^{\times k}~,
\]
are non-canonically identified as products of $A$.
Furthermore, the restriction of this $n$-disk algebra along the standard symmetric monoidal functor $\disk_n^{\sf fr} \to \disk_n$, which selects Euclidean spaces, is the $\cE_n$-algebra forgotten from the commutative algebra structure on $A$. 
The $\sO(n)$-action on $A$ is that inherited from the infinite-loop structure of the Eilenberg--MacLane spectrum $\sH A$ through the $J$-homomorphism.
Let $M$ be an $n$-manifold.
An $A$-orientation of $M$ defines a natural transformation making the diagram of $\infty$-categories commute:
\begin{equation}\label{11}
\xymatrix{
\disk_{n/M} \ar[rr]  \ar[d]_-{\pi_0}
&&
\disk_n \ar[d]^-{\Omega_{\BO(n)}^n X} 
\\
{\sf Fin}  \ar[rr]^-{I\mapsto A^I}
&&
\spaces ;
}
\end{equation}
here, the bottom horizontal functor is the symmetric monoidal functor from finite sets with disjoint union, which is the free symmetric monoidal $\infty$-category generated by a point, determined by the commutative algebra $A$ in the Cartesian symmetric monoidal $\infty$-category $\spaces$.
So choose an $A$-orientation of $M$ (supposing one exists).  
Corollary~\ref{non-abel} then yields an equivalence between spaces
\begin{equation}\label{10}
\int_M A
\underset{(\ref{11})}{~\simeq~}
\int_M \Omega^n_{\BO(n)} X
\underset{\rm Cor~\ref{non-abel}}{~\simeq~} 
\Mapc\bigl(M,K(A,r)\bigr)~.
\end{equation}
Note that each space in this display is equipped with a base point, and the maps in this display are canonically based maps.  
As the $\cE_n$-algebra $A$ is forgotten from a commutative algebra, we recognize $\int_M A$ as the free $A$-module in the $\infty$-category $\spaces$ generated by the underlying space of $M$.
Therefore, applying homotopy groups to~(\ref{10}) gives an isomorphism between graded abelian groups,
\[
\overline{\sH}_{\ast}(M;A)
~\cong~
\pi_\ast \int_M A
~\cong~
\pi_\ast \Mapc\bigl(M,K(A,n)\bigr)
~\cong~
\overline{\sH}^{n-\ast}_{\sf c}(M;A)~,
\]
from reduced homology to reduced cohomology.

\begin{remark} The factorization homology $\int_M\Omega^n_BX$ is built from configuration spaces of disks in $M$ with labels defined by $X\to B$, and the preceding result thereby has roots in the configuration space models of mapping spaces dating to the work of Segal, May, McDuff and others in the 1970s; see \cite{segal}, \cite{may}, \cite{mcduff}, and \cite{bodig}. However, the classical configuration-space-with-labels, as described in~\cite{bodig}, models a mapping space with target the $n$-fold suspension of $X$, rather than into $X$ itself. Factorization homology thus more closely generalizes the configuration spaces with {\it summable} or {\it amalgamated} labels of Salvatore \cite{salvatore} and Segal \cite{segallocal}.

\end{remark}

\begin{proof}[Proof of Theorem~\ref{spaces}]
Corollary~\ref{non-abel} identifies the functor $\Gamma_{\sf c} \colon \spaces_B^{\geq n} \to  \bH(\mfld_n^B,\Space)$ as the composition $\Gamma_{\sf c} \colon \spaces_B^{\geq n} \xra{\Omega^n_B} \disk_n^B \xra{\int}  \bH(\mfld_n^B,\Space)$.
Theorem~\ref{homology} gives that $\int$ is fully-faithful, so it remains to argue that $\Omega_B^n$ is fully-faithful with essential image the group-like $\disk_n^B$-algebras in spaces.  
This is immediate because, for instance, $\spaces_{/B}$ is an $\infty$-topos (Theorem 5.1.3.6 of \cite{dag}).  

\end{proof}

\section{Calculations}\label{sec.calculations}

In \S\ref{sec.nonabel}, we identified values of factorization homology of $\disk_n$-algebras in spaces, with its Cartesian monoidal structure, as twisted compactly-supported mapping spaces.
Here we identify more values of factorization homology of $\disk_n$-algebras in chain complexes, and spectra, notably with monoidal structure given by tensor product -- these cases are closest to the physical motivation given in the introduction.

In this section, we fix a map $B\to \BO(n)$ from a space.
Two cases of notable interest are $\ast \xra{\{\RR^n\}}\BO(n)$ and $\BO(n)\xra{\sf id}\BO(n)$.

\subsection{Factorization homology for direct sum}\label{sec.direct.sum}
We now show that in the simple case of $\disk_n$-algebras in chain complexes with direct sum, factorization homology recovers ordinary homology.

Fix a presentable $\infty$-category $\cV$, such as the $\infty$-category $\spectra$ of spectra, or the $\infty$-category $\m_{\Bbbk}$ of chain complexes over a fixed commutative ring $\Bbbk$.
Coproduct, which in the case that $\cV$ is stable is direct sum, endows $\cV$ with the coCartesian symmetric monoidal structure: $(\cV,\amalg)$.
Consider the commutative diagram among $\infty$-categories, in which each functor is implemented by the evident restriction:
\[
\xymatrix{
\Alg_{\disk_n^B}(\cV,\amalg) \ar[d]_-{\simeq}
&&
\Fun^{\ot}\bigl(\mfld_n^B,(\cV,\amalg)\bigr) \ar[ll] \ar[dll]
\\
\Fun(B,\cV)
&&
.
}
\]
In this case of a coCartesian symmetric monoidal structure, the downward functor is an equivalence, as indicated.  In other words, a $\disk_n^B$-algebra in $(\cV,\amalg)$ is precisely a functor $B\to \cV$.  
\footnote{Here and always, the space $B$ is taken as an $\infty$-groupoid and thereafter as an $\infty$-category.
In this way, we make sense of a functor $B\to \cV$, which is simply a $B$-indexed local system of objects in $\cV$.}

\begin{example}
Let $G$ be a topological group equipped with a continuous representation $G\xra{\rho} {\sf GL}(n)\simeq \sO(n)$.  
Then a $\disk_n^{\sB G}$-algebra in $(\cV,\amalg)$ is a $G$-module $\sB G \to \cV$.  
In the case that $G = e	$ is the trivial group, such a $G$-module is simply an object in $\cV$.
\end{example}

Through Proposition~\ref{as-LKE}, which identifies factorization homology as a left adjoint to the top horizontal functor in the preceding diagram, we conclude that factorization is identified as a left adjoint to the downleftward functor, which is restriction along the fully-faithful inclusion $B\to \mfld_n^B$ of $B$-framed Euclidean spaces.  
Left adjoints to restriction functors evaluate as left Kan extensions.
The standard formula for left Kan extensions therefore identifies, for each $B$-framed $n$-manifold $M= (M,\varphi)$ and each $\disk_n^B$-algebra $B\xra{V}\cV$, the value of factorization homology as the colimit in $\cV$:
\[
\int_M V
~\simeq~
\colim\bigl( B_{/M} \to B \xra{V} \cV\bigr)
\]
Corollary~\ref{tang-class} identifies the functor $B_{/M}\to B$ appearing in this expression as the given functor $M\xra{\varphi} B$ from the underlying $\infty$-groupoid, or homotopy type, of $M$:
\[
\int_M V
~\simeq~
\colim\bigl( M\xra{\varphi} B \xra{V} \cV\bigr)~.
\]

\begin{remark}
Note that the value of factorization homology in $(\cV,\amalg)$ over a $B$-framed manifold $M\xra{\varphi}B$ only depends on the underlying space of $M$ equipped with its given map to $B$.
	
\end{remark}

\begin{example}
Let $G$ be a topological group equipped with a continuous representation $G\xra{\rho} {\sf GL}(n)\simeq \sO(n)$.  
Let $V$ be a $G$-module in $\cV$.
Let $M$ be an $n$-manifold.
A $\sB G$-framing on $M$ is a principal $G$-bundle $P \to M$ equipped with a $\rho$-equivariant map $P\to {\sf Fr}(M)$ over $M$ to the frame-bundle of $M$.
Given such a $\sB G$-framing on $M$, we identify the factorization homology
\[
\int_M V~\simeq~
P \underset{G}\ot V
\]
as the diagonal $G$-quotient of the tensor of $V\in \cV$ with the space $P$.
In particular, in the case that $G\xra{=}\sO(n)$, the factorization homology is the diagonal $\sO(n)$-quotient with the frame-bundle of $M$ tensored with the $\sO(n)$-module $V$ in $\cV$:
\[
\int_M V~\simeq~ {\sf Fr}(M)\underset{\sO(n)}\ot V~.
\]
In the case that $G=e$ is a trivial group, then $V\in \cV$ is simply an object and the principal bundle $P \xra{=}M$ is identical to $M$, so that factorization homology
\[
\int_M V~\simeq~ M\ot V
\]
is simply the tensor of the object $V\in \cV$ with the underlying space of $M$.

\end{example}

\begin{example}
Suppose $B =\ast$, so that a $B$-framing on an $n$-manifold is exactly a framing. 
Suppose $\cV = \m_{\Bbbk}$ is the $\infty$-category of chain complexes over a fixed commutative ring $\Bbbk$.
As discussed above, a $\cE_n$-algebra in $(\m_{\Bbbk},\oplus)$ is simply a chain complex, $V\in \m_{\Bbbk}$.
Furthermore, the factorization homology
\[
\int_M V~\simeq~
M\ot V~\simeq~ \sC_\ast(M;V)~\in \m_{\Bbbk}
\]
is the singular chain complex of $M$ with coefficients in $V$.
	
\end{example}

\begin{example}
Again suppose $B =\ast$, so that a $B$-framing on an $n$-manifold is exactly a framing. 
Suppose $\cV = \spectra$ is the $\infty$-category of spectra.
As discussed above, an $\cE_n$-algebra in $(\spectra,\oplus)$ is simply a spectrum $V\in \spectra$.
Furthermore, the factorization homology
\[
\int_M V~\simeq~
M\ot V~\simeq~ \Sigma^{\infty}_+ M \wedge V~\in \spectra
\]
is the smash product of $V$ with the suspension spectrum of $M$.

\end{example}

\begin{remark}
Taking products with Euclidean spaces defines a sequence of symmetric monoidal $\infty$-categories:
\[
\mfld_0^{\fr}
\longrightarrow
\mfld_1^{\fr}
\longrightarrow
\dots
\longrightarrow
\mfld_n^{\fr}
\longrightarrow
\dots
~.
\]
Forgetting to underlying spaces defines, for each $k$, a symmetric monoidal functor $\mfld_k^{\fr}\to \spaces^{\sf fin}$ to the $\infty$-category of finite spaces\footnote{i.e., finite CW complexes} and coproduct among them.  
Via the canonical equivalence between underlying spaces $M\simeq M\times \RR$ for each framed manifold $M$, 
these forgetful functors assemble as a symmetric monoidal functor from the colimit
\begin{equation}\label{40}
\underset{k\geq 0}\colim \mfld_k^{\fr}
\longrightarrow
\spaces^{\sf fin}~.
\end{equation}
We argue that this symmetric monoidal functor is an equivalence.
\begin{itemize}
\item
{\bf Essentially surjective.}	
Each finite space is equivalent to the underlying space of the geometric realization of a finite simplicial complex.
Such a geometric realization can be piecewise-linearly embedded into Euclidean space.  
A regular open neighborhood of such an embedding is a finitary smooth manifold that is homotopy equivalent to this geometric realization.
In this way we conclude that the functor~(\ref{40}) is essentially surjective.

\item
{\bf Fully-faithful.}
Let $M$ and $N$ be $n$-manifolds.
It follows from Whitney's embedding theorem that the map between spaces
\[
\underset{k\geq 0}\colim \Emb^{\fr}(M\times \RR^{k-n} , N\times \RR^{k-n})
\longrightarrow
\Map_{\spaces}(M,N)
\]
is a weak homotopy equivalence.

\end{itemize}
Through the equivalence~(\ref{40}), Theorem~\ref{homology} applied to $\cV = (\m_{\Bbbk},\oplus)$ recovers to the formulation of the Eilenberg--Steenrod axioms given in the introduction. 
If one sets $\cV$ to be the opposite $(\m_{\Bbbk}^{\op},\oplus)$, then one likewise recovers the Eilenberg--Steenrod axioms for cohomology.

\end{remark}

\subsection{Factorization homology with coefficients in commutative algebras}\label{sec.commutative}
Here, we examine factorization homology of commutative algebras, otherwise known as $\cE_\infty$-algebras.

Fix a symmetric monoidal $\infty$-category $\cV$ that is 
$\ot$-presentable.

There is a canonical identification of the symmetric monoidal envelope of the commutative $\infty$-operad
\[
{\sf Env}({\sf Comm})
~\simeq~
\fin= (\fin,\amalg)
\]
with the coCartesian symmetric monoidal category of finite sets.
\footnote{
See the discussion just above Lemma~\ref{disk.env} for a quick description of a \emph{symmetric monoidal envelope}.  
}
This is to say that restriction along the commutative algebra in $\fin$ whose underlying object is $\ast\in \fin$ defines an equivalence between $\infty$-categories:
\[
\Fun^{\ot}(\fin,\cV)
\xra{~\simeq~}
\Alg_{\com}(\cV)
\]
So restriction along the symmetric monoidal functor given by taking connected components
\[
[-]: \disk^B_n\ra\disk_n \xra{[-]} {\sf Fin}
\]
defines a forgetful functor 
\[
{\sf fgt}\colon \Alg_{\com}(\cV)
\longrightarrow 
\Alg_{\disk^B_n}(\cV)~.
\] 
Via Corollary~3.2.3.3 of~\cite{dag}, the assumed $\ot$-presentability of $\cV$ implies the $\infty$-category $\Alg_{\com}(\cV)$ is presentable.
In particular, it is tensored over the $\infty$-category of spaces:
\[
\spaces\times \Alg_{\com}(\cV) \xra{~\ot~} \Alg_{\com}(\cV)~,\qquad (X,A)\mapsto \colim\bigl(X\xra{!} \ast \xra{\{A\}} \Alg_{\com}(\cV)\bigr)~.
\]
\begin{prop}\label{tensor} 
The following diagram among $\infty$-categories canonically commutes:
\[
\xymatrix{
\mfld^B_n \times \Alg_{\com}(\cV) \ar[rr]^-{{\rm forget}\times {\sf id}}\ar[d]_-{{\sf id}\times {\rm fgt}} 
&&
\Space \times \Alg_{\com}(\cV)\ar[rr]^-\ot 
&&
\Alg_{\com}(\cV)\ar[d]^-{\rm forget}
\\
\mfld^B_n \times\Alg_{\disk^B_n}(\cV)\ar[rrrr]^{\displaystyle \int} 
&&
&&
\cV .
}
\]
In particular, for each $B$-framed $n$-manifold $M$, and each commutative algebra $A$, there is a canonical equivalence 
\[
\int_M A~\simeq~ M\ot A
\] 
between the factorization homology of $A$ over $M$ and the tensor of the commutative algebra $A$ with the underlying space of $M$.
\end{prop}

\begin{proof} 
Observe the commutative diagram among symmetric monoidal $\infty$-categories:
\[
\xymatrix{
\disk_n^B \ar[rr] \ar[d]_-{[-]}
&&
\mfld_n^B \ar[d]^-{\rm forget}
\\
\fin \ar[rr]
&&
\spaces
}
\]
in which the bottom two are endowed with their coCartesian symmetric monoidal structures.  
Restriction defines the commutative diagram among $\infty$-categories of symmetric monoidal functors
\begin{equation}\label{41}
\xymatrix{
\Alg_{\disk_n^B}(\cV)
&&
\Fun^{\ot}(\mfld_n^B,\cV) \ar[ll]
\\
\Alg_{\com}(\cV) \ar[u]^-{\sf fgt}
&&
\Fun^{\ot}(\spaces,\cV) . \ar[ll] \ar[u]
}
\end{equation}
It follows from Proposition~3.2.4.7 of~\cite{dag} that the forgetful functor $\Alg_{\com}(\cV) \to\cV$ canonically lifts as a symmetric monoidal functor $\bigl( \Alg_{\com}(\cV) , \amalg \bigr) \to \bigl( \cV , \ot \bigr)$ from the coCartesian symmetric monoidal structure on $\Alg_{\com}(\cV)$.
Postcomposing by this symmetric monoidal functor determines the commutative diagram among $\infty$-categories:
\begin{equation}\label{42}
\xymatrix{
\Alg_{\com}\bigl( \Alg_{\com}(\cV),\amalg \bigr) \ar[d]
&&
\Fun^{\ot}\bigl(\spaces,(\Alg_{\com}(\cV) , \amalg) \bigr)  \ar[d]  \ar[ll]
\\
\Alg_{\com}(\cV) 
&&
\Fun^{\ot}(\spaces,\cV) . \ar[ll] 
}
\end{equation}
It follows from~\S\ref{sec.direct.sum} that the right downward functor in this diagram is an equivalence. 
From the same proposition cited in~\cite{dag} above, the right downward functor preserves colimits.  
From~\S\ref{sec.direct.sum}, we conclude that the left adjoint to the bottom horizontal functor in~(\ref{41}) is the composite functor
\[
\Alg_{\com}(\cV)
\xra{~A\mapsto (X\mapsto X\ot A)~}
\Fun^{\ot}\bigl(\spaces , \Alg_{\com}(\cV)\bigr)
\xra{~\rm forget~}
\Fun^{\ot}(\spaces , \cV)
\]
of first tensoring followed by the forgetful functor.
Now, taking left adjoints of the horizontal functors results in a lax-commutative diagram among $\infty$-categories:
\[
\xymatrix{
\Alg_{\disk_n^B}(\cV) \ar@(u,u)[rr]^-{\int}
&
\Downarrow
&
\Fun^{\ot}(\mfld_n^B,\cV) 
\\
\Alg_{\com}(\cV) \ar[u]^-{\sf fgt} \ar[r]^-{\ot}
&
\Fun^{\ot}\bigl(\spaces , (\Alg_{\com}(\cV) , \amalg) \bigr)  \ar[r]^-{\rm forget}
&
\Fun^{\ot}(\spaces,\cV) .  \ar[u]
}
\]
The proof is complete once we show the 2-cell in this diagram is invertible.
This 2-cell is invertible if and only if, for each commutative algebra $A\in \Alg_{\com}(\cV)$ and each $B$-framed $n$-manifold $M$, the resulting morphism in $\cV$,
\begin{equation}\label{43}
\int_M A
\longrightarrow
M\ot A~,
\end{equation}
is an equivalence. (Here, and in what follows, we are not giving notation to the functor ${\sf fgt}$.)
Because the underlying space of each $B$-framed Euclidean $n$-dimensional space is contractible, this morphism~(\ref{43}) is an equivalence in such cases.
By construction, the tensoring functor $\spaces \xra{X\mapsto X\ot A}\Alg_{\com}(\cV)$ preserves colimits.   
Because collar-gluings among $B$-framed $n$-manifolds forget as pushouts among their underlying spaces, it follows that the functor $\mfld_n^B \xra{M\mapsto M\ot A} \Alg_{\com}(\cV) \xra{\rm forget} \cV$ is $\ot$-excisive, and in particular it is symmetric monoidal.  
It follows from Theorem~\ref{homology} that the morphism~(\ref{43}) is an equivalence, as desired.

\end{proof}

Proposition~\ref{tensor}, together with the universal property of tensoring over spaces, yields the following.
\begin{cor}\label{tensor.universal}
Let $A\in \Alg_{\com}(\cV)$ be a commutative algebra in a symmetric monoidal $\infty$-category that is $\ot$-presentable.
Let $M$ be a $B$-framed $n$-manifold.
There is a canonical commutative algebra structure on the factorization homology $\int_M A$.
Furthermore, as a commutative algebra, it corepresents the copresheaf
\[
\Map_{\com}\Bigl( \int_M A , - \Bigr)
\colon
\Alg_{\com}(\cV)
\longrightarrow
\Spaces
~,\qquad
C\mapsto 
\Map_{\com}(A,C)^M~,
\]
whose value on a commutative algebra is the the cotensor, or mapping space, of $M$ to the space of morphisms from $A$ to $C$.

\end{cor}

Proposition~\ref{tensor} and Corollary~\ref{tensor.universal}, together with the fact that left adjoints compose, yields the following.
\begin{cor}\label{on.sym}
In the situation of Corollary~\ref{tensor.universal}, for $V\in \cV$ an object with $\sym(V)\in \Alg_{\com}(\cV)$ the free commutative algebra on $V$, there is a canonical equivalence between commutative algebras in $\cV$:
\[
\int_M \sym(V) 
~\simeq ~
\sym(M\ot V)~.
\]

\end{cor}

\begin{example}
In the case that $\cV = (\m_{\Bbbk},\underset{\Bbbk}\ot)$, Corollary~\ref{on.sym} specializes as an equivalence
\[
\int_M \sym(V)
~\simeq~
\sym\bigl( \sC_\ast(M;V) \bigr)~,
\]
the graded-commutative algebra on the chain complex of singular chains on the underlying space of $M$ with coefficients in $V$.
\end{example}

The next result, proved as Proposition~5.3 in~\cite{oldfact}, further identifies factorization homology of commutative algebras, in a suitably infinitesimal case: the cohomology of suitably nilpotent spaces.  
Its proof amounts to using Proposition~\ref{tensor}, then arguing via an Eilenberg--Moore spectral sequence.  

\begin{prop}[\cite{oldfact}]
Let $M$ be an $n$-manifold.
Let $X$ be a nilpotent space whose first $n$ homotopy groups are finite,  $|\pi_i X|< \oo$ for $i\leq n$, and let $\Bbbk$ be a commutative ring. 
There is a natural equivalence of $\Bbbk$-modules 
\[
\int_M \sC^\ast(X;\Bbbk)~ \simeq~ \sC^\ast(X^M;\Bbbk)
\] 
between the factorization homology of $M$ with coefficient in the $\Bbbk$-cohomology of $X$ and the $\Bbbk$-cohomology of the space of maps to $X$ from the underlying space of $M$.
\end{prop}

\begin{remark} 
In the work~\cite{gtz1}, the authors lay out a similar approach, by way of Hochschild homology-type invariants, to the study of the cohomology of mapping spaces.
\end{remark}

\subsection{Factorization homology from Lie algebras}
We now identify factorization homology of $\disk_n$-algebras coming from Lie algebras.  
The results that follow are closely analogous to those above concerning the factorization homology of $\disk_n$-algebras coming from based topological spaces. 
For simplicity, we assume our Lie algebras are defined over a fixed field $\Bbbk$ of characteristic zero.

As we proceed, we use that the $\infty$-category $\Alg_{\Lie}(\m_{\Bbbk})$ of Lie algebras over $\Bbbk$ are cotensored over spaces:
\[
\spaces^{\op}\times \Alg_{\Lie}(\m_{\Bbbk})
\xra{~ (X,\mathfrak{g})\mapsto \mathfrak{g}^X~}
\Alg_{\Lie}(\m_{\Bbbk})
~,\qquad
(X,\mathfrak{g})\mapsto \limit\Bigl(~ X \xra{!} \ast \xra{~\{\mathfrak{g}\}~} \Alg_{\Lie}(\m_{\Bbbk})~\Bigr)~.
\]
This limit indeed exists because, for instance, the $\infty$-category $\Alg_{\Lie}(\m_{\Bbbk})$ admits products, cofiltered limits, and totalizations.  
More explicitly, $\mathfrak{g}^X \simeq \sC^\ast(X;\mathfrak{g}) = \Hom_{\Bbbk}\bigl(\sC_\ast(X;\Bbbk) , \mathfrak{g}\bigr)$, with a canonically inherited Lie algebra structure.  
For $M$ an $n$-manifold, we define the kernel Lie algebra,
\[
\xymatrix{
\Map_{\sf c}(M , \mathfrak{g})   \ar[rr]  \ar[d]
&&
{\mathfrak{g}}^{M^+}   \ar[d]^-{{\sf ev}_+}
\\
0 \ar[rr]
&&
\mathfrak{g}  ,
}
\]
where $M^+$ is the 1-point compactification of $M$.
More explicitly, $\Mapc(M,\frak g)\simeq \sC_{\sf c}^\ast(M;\frak g)$, the compactly-supported cochains on $M$ with coefficients in $\frak g$. 
In the case that $M=\RR^n$, we notate $\Omega^n \mathfrak{g}:=\Mapc(\RR^n,\frak g)$, and observe that this Lie algebra assembles as a $\disk_n$-algebra,
\[
\Omega^n \frak g \colon 
\disk_n
\longrightarrow
\bigl( \Alg_{\Lie}(\m_{\Bbbk}) , \times \bigr)~,
\]
to the Cartesian symmetric monoidal structure on Lie algebras.
Postcomposing this symmetric monoidal functor by that of Lie algebra chains, which indeed carries finite products of Lie algebras to finite tensor products (over $\Bbbk$) of $\Bbbk$-modules, results in a composite symmetric monoidal functor 
\[
\sC_\ast^{\Lie}\bigl( \Omega^n \frak g \bigr)
\colon 
\disk_n
\xra{~\Omega^n \frak g~}
\bigl( \Alg_{\Lie}(\m_{\Bbbk}) , \times \bigr)
\xra{~\sC_\ast^{\Lie}~}
\bigl( \m_{\Bbbk} , \underset{\Bbbk}\ot \bigr) ~.
\]

The next identification is also discussed in~\cite{owen} and~\cite{kevinowen}, and appears in the present form in~\cite{oldfact}.
Its proof amounts to applying Lie algebra chains, which is a geometric realization-preserving functor, to the Poincar\'e/Koszul duality equivalence $\int_M \Omega^n {\frak g} \simeq \Mapc(M,{\frak g})$ of~\cite{zp}.
\begin{prop}\label{env} 
Let $M$ be an $n$-manifold.
Let $\frak g$ be a Lie algebra over $\Bbbk$.
There is a canonical identification between $\Bbbk$-modules,
\[
\int_M \sC_\ast^{\Lie}\bigl( \Omega^n {\frak g}\bigr)
~ \simeq~ 
\sC^{\Lie}_\ast\bigl(\Mapc(M, \mathfrak{g})\bigr)~.
\]

\end{prop}

\begin{remark} 
Let $\frak g$ be a Lie algebra over $\Bbbk$.
The $\disk_n$-algebra $\sC_\ast^{\Lie}\bigl( \Omega^n {\frak g}\bigr)$ appearing above has an interesting interpretation, established by Knudsen~\cite{ben.thesis}. There is a forgetful functor from $\cE_n$-algebras to Lie algebras,
\[
\Alg_{\Lie}(\m_{\Bbbk})
\longleftarrow
\Alg_{\cE_n}(\m_{\Bbbk})~\colon{\sf fgt}~.
\]
For $n=1$, there is the standard forgetful functor;
For $n>1$, Kontsevich's formality result (see~\cite{kontsevich} for a sketched outline, and~\cite{pascal} for a more detailed account) identifies the operad $\sC_\ast(\cE_n ; \RR)$ in $\m_\Bbbk$ is equivalent to the Poisson operad $\sP_n$, to which the Lie operad in $\m_{\RR}$ maps.\footnote{The operad $\sP_n$ in $\m_{\Bbbk}$ represents: a $\m_\Bbbk$-module with a commutative and associative multiplication $\cdot$; an $(n-1)$-shifted binary operation $[-,-]$ that satisfies the Jacobi identity; the structure of this Lie bracket $[-,-]$ being a derivation over $\cdot$ in each variable.}
(See~\cite{cohen} for an account at the level of homology.)
The adjoint functor theorem (Corollary~5.5.2.9 of~\cite{topos}) applies to this forgetful functor, thereby offering a left adjoint:
\[
\sU_n \colon \Alg_{\Lie}(\m_{\Bbbk})~ \rightleftarrows~ \Alg_{\cE_n}(\m_{\Bbbk})~\colon{\sf fgt}~.
\]
In the case $n=1$, this left adjoint $\sU_1$ agrees with the familiar universal enveloping algebra functor.  
In general, there is an identification between $\cE_n$-algebras,
\[
\sU_n \mathfrak{g}~\simeq~  \sC_\ast^{\Lie}\bigl( \Omega^n {\frak g} \bigr)
\]
proved by Knudsen~\cite{ben.thesis}.
Through this identification, Proposition~\ref{env} can be reformulated as an equivalence of chain complexes over $\Bbbk$,
\[
\int_M\sU_n\mathfrak{g}~\simeq~ \sC^{\Lie}_\ast\bigl(\Mapc(M,\mathfrak{g})\bigr)~,
\]
in the case that $M$ is equipped with a framing.  
\end{remark}

\subsection{Factorization homology of free $\disk^B_n$-algebras}\label{sec.free}
We now identify the factorization homology of free $\disk^B_n$-algebras.
The results here are extracted from~\S5.2 of~\cite{oldfact},~\S4.3 of~\cite{aft2}, and~\S2.4 of~\cite{pkd}.

For this section, we fix a symmetric monoidal $\infty$-category $\cV$ that is $\ot$-presentable.  
Also, for simplicity restrict our attention to the case the space $B$ is connected, and equipped with a base point.
This condition, and structure, on $B$ identify the tangential structure 
\[
(B\to \BO(n)) ~= ~(\sB G \xra{\sB \rho} \BO(n))
\]
as the classifying space of a morphism $\rho\colon G\to \sO(n)$ between group-objects in $\spaces$.

For each $B$-framed $n$-manifold $M$, and for each finite cardinality $i\geq 0$, consider the configuration space 
\[
\conf_i(M)~:=~ \Bigl\{\{1,\dots,i\}\xra{c} M\mid c\text{ is injective}  \Bigr\}~\subset ~M^{\times i}~,
\]
which is an open subspace of the $i$-fold product of $M$.
As an open subspace, it inherits the structure of a smooth $ni$-manifold.
The $\sB G$-framing on $M$ then determines a $\sB(\Sigma_i\wr G)$-framing on $\conf_i(M)$:
\[
\xymatrix{
&&
\sB(\Sigma_i\wr G) \ar[d]
\\
\conf_i(M)  \ar[rr]^-{\tau_{\conf_i(M)}}  \ar[urr]
&&
\BO(ni) .
}
\]
This lift of the tangent classifier of $\conf_i(M)$ selects a principal $\Sigma_i\wr G$-bundle
\[
\conf_i^G(M)
\longrightarrow
\conf_i(M)~,
\]
which we refer to as a \emph{$G$-framed} configuration space.
In the extremal case of $G\xra{=}\sO(n)$, then we have $\conf_i^G(M) = \conf_i(M)$. In the extremal case that $G=e$ is the trivial group then $\conf_i^G(M) = \conf_i^{\fr}(M)$ is the \emph{framed} configuration space, a point in which is an injection $c\colon \{1,\dots,i\}\hookrightarrow M$ together with a choice of the basis for each tangent space $T_{c(j)}M$.

The \emph{free $\disk_n^B$-algebra} functor is the left adjoint 
\[
\free_n^B\colon \m_G(\cV)~:=~ \Fun(B,\cV)
~\rightleftarrows~
\Alg_{\disk_n^B}(\cV)
\]
to the restriction along the fully-faithful inclusion $B \to \disk_n^B$ of the $B$-framed Euclidean spaces.

The next result identifies the factorization homology of free $\disk_n^{\sB G}$-algebras in terms of $G$-framed configuration spaces.
\begin{prop}\label{free} 
Let $\rho\colon G\to \sO(n)$ be a morphism between group-objects in $\spaces$.  
Let $M$ be a $\sB G$-framed $n$-manifold. 
Let $V\in \m_G(\cV)$ be a $G$-module in $\cV$.  
There is a canonical equivalence in $\cV$,
\[
\int_M \free_n^{\sB G}(V) ~\simeq~ \coprod_{i\geq 0} \conf_i^G(M)\underset{\Sigma_i\wr G}\ot V^{\ot i}
\] 
between the factorization homology of the free $\disk_n^{\sB G}$-algebra and the coproduct of $G$-framed configuration spaces of $M$ labeled by $V$.
\end{prop}

The proof of Proposition~\ref{free}, which appears in the above-mentioned citations, amalgamates the following key observations.
First, the result is true for $M$ a Euclidean space.
Second, a hypercover-type argument (as in Theorem~A.3.1 of~\cite{dag}), applied to $V$-labeled configuration spaces, gives that the righthand term satisfies $\ot$-excision.

Using that left adjoints compose, Proposition~\ref{free} applied to a free $G$-module on an object $W\in \cV$ gives the following result.
\begin{cor}
In the case that $V\simeq G\ot W$ is the free $G$-module on an object $W\in \cV$, there is a canonical equivalence in $\cV$:
\[
\int_M \free_n^{\sB G}(G\ot W)
~\simeq~
\coprod_{i\geq 0} \conf_i(M)\underset{\Sigma_i}\ot W^{\ot i}
~.
\]

\end{cor}

\begin{remark}[On homotopy invariance]\label{config.spaces}
The work of Longoni--Salvatore~\cite{simple} shows that configuration spaces of a manifold are sensitive to simple-homotopy type of the manifold. Based on Proposition~\ref{free}, this has the following consequence, that factorization homology is not a homotopy invariant of manifolds, even in the weakest sense.
The strongest form of homotopy invariance for the functor $\int A$, for a fixed $\disk_n$-algebra $A$, would be the structure of a factorization
\[
\xymatrix{
\Mfld_n   \ar[rr]^-{\int A}   \ar[d]_-{\rm forget}
&&
\cV
\\
\spaces\ar@{-->}[urr]
}
\]
of $\int A$ through the forgetful functor $\Mfld_n\ra\spaces$. 
This structure of homotopy invariance is equivalent to the $\disk_n$-algebra $A$ being commutative. 
Likewise, a strong form of proper-homotopy invariance would be a factorization
\[
\xymatrix{
\Mfld_n\ar[rr]^-{\int A}\ar[d]_-{(-)^+}&&\cV\\
\spaces_\ast^{\op}\ar@{-->}[urr]}
\]
where $(-)^+\colon \Mfld_n\ra \spaces_\ast^{\op}$ is 1-point compactification: sending a manifold $M$ to the 1-point compactification $M^+$, and an open embedding $M\hookrightarrow N$ to the collapse map $N^+\ra M^+$. The existence of this second factorization is roughly equivalent to $A$ being the $n$-disk enveloping algebra of a Lie algebra.
\end{remark}

\section{Filtrations}

In this section, we establish two filtrations of factorization homology: in the manifold variable, one can filter by bounding the cardinality of the number of embedded disks; in the algebra variable, one can filter by bounding the number of multiplications allowed by elements of the algebra. Filtrations of the first kind generalize the Goodwillie--Weiss embedding calculus \cite{weiss}; filtrations of the second kind are an instance of the Goodwillie functor calculus \cite{goodwillie}. These filtrations are exchanged under Koszul duality, as is described in the following section. For simplicity, we have written these sections for the case of framed $n$-manifolds---see \cite{pkd} and \cite{zp} for the case of unoriented $n$-manifolds, and the case of $B$-framed $n$-manifolds is analogous. 

For the rest of this section, we fix a symmetric monoidal $\oo$-category $\cV$ which is stable and $\ot$-presentable.

\subsection{Cardinality filtrations}
\begin{definition}
The $\oo$-category $\mfld_n^{\fr,\sf surj}$ is the $\oo$-subcategory of $\mfld_n^{\fr}$ whose objects are nonempty framed $n$-manifolds and whose morphisms consist of those embeddings which induce surjections on the set of connected components. $\mfld_n^{\fr,\leq k}\subset\mfld_n^{\fr,\sf surj}$ is the full $\oo$-subcategory whose objects have at most $k$ components. The $\oo$-categories $\disk_n^{\fr,\sf surj}$ and $\disk_n^{\fr,\leq k}$ are the further $\oo$-subcategories of $\mfld_n^{\fr,\sf surj}$ whose objects are nonempty finite disjoint unions of Euclidean spaces.
\end{definition}

That is, an embedding $f\colon M\hookrightarrow N$ belongs to $\mfld_n^{\fr,\sf surj}$ if and only if $\pi_0 f\colon \pi_0M \ra \pi_0 N$ is surjective.

\begin{definition}
For a functor $A: \disk_n^{\fr} \ra \cV$, the $k$th truncation of factorization homology is the colimit of the composite functor
\[
\tau_{\leq k}\int_M A := \colim\Bigl(\disk_{n/M}^{\fr,\leq k}\overset{A}\longrightarrow \cV\Bigr)~.
\]
\end{definition}

The difference between the subsequent truncations has a concrete description in terms of configuration spaces:

\begin{lemma}[\cite{pkd}]\label{lemma.layers}
Let $A$ be an $\cE_n$-algebra in $\cV$.
For each closed framed $n$-manifold $M$, there is a canonical cofiber sequence in $\cV$:
\[
\xymatrix{
\tau_{\leq k-1}\displaystyle\int_M A
\ar[r]&\tau_{\leq k}\displaystyle\int_M A\ar[r]&\conf_k(M)^+ \underset{\Sigma_k}
\bigotimes A^{\ot k}}
~.
\]
Here, $\conf_k(M)^+\in\spaces_\ast$ is the 1-point compactification of $\conf_k(M)$, regarded as a based $\Sigma_k$-space; the cofiber is the reduced tensor with this based $\Sigma_k$-space.
\end{lemma}

\begin{example}
Consider the case that $\cV=\mod_{\Bbbk}$ and $n=1$, so that $\int_{S^1} A \simeq {\sf HC}_\ast(A)$ is the Hochschild chain complex of the associative algebra $A$.  
After Lemma~\ref{lemma.layers}, Poincar\'e duality applied to the the framed $k$-manifold $\conf_k(S^1)$ with coefficients in the local system $A^{\ot k}$, identifies, for $k>0$, the cofibers of the cardinality filtration as
\[
\conf_k(S^1)^+\underset{\Sigma_k}\bigotimes A^{\ot k}
~\simeq~
{\sf cInd}_{\sC_k}^{\TT}(A[1]^{\ot k})~.
\]
Here, $A[1]$ is chain complex over $\Bbbk$ which is the suspension of $A$; $\sC_k$ is the cyclic group of order $k$, regarded as the closed subgroup of $k$th roots of unity in the Lie group $\TT$ of unit complex numbers; ${\sf cInd}_{\sC_k}^{\TT}$ is the coinduction from the standard action of $\sC_k$ on the $k$-fold tensor product $A[1]^{\ot k}$.

\end{example}

The next series of results lead to the proof of Proposition~\ref{prop.trunc}, that factorization homology can be computed as a sequential colimit of its truncated values. To do this, we will use a comparison of our disk categories with the Ran space $\ran(M)$. This is the space of finite subsets of $M$, topologized so that points can collide.

\begin{definition}
${\sf Fin}^{\sf surj}$ is the category of finite nonempty sets and surjections among them. The Ran space of a topological space $M$ is the topological space
\[
\ran(M) := \colim\Bigl({\sf Fin}^{\sf surj, \op}\xra{M}\Top\Bigr)
\]
which is the colimit of the diagram sending a set $J$ to the space $M^J$, and a surjection $J\ra I$ to the diagonal map $M^I \hookrightarrow M^J$. For finite subset $S\subset M$, the topological space
\[
\ran(M)_S \subset \ran(M)
\]
is the subspace consisting of those finite subsets of $M$ which contain $S$. The cardinality stratification of $\ran(M)_S\ra \NN$ sends a finite subset $T$ to its cardinality $|T|$.
\end{definition}

We will make use of the exit-path $\oo$-category of a stratified space: see \cite{dag} or \cite{aft1}.

\begin{lemma}[\cite{pkd}]
There is contravariant equivalence
\[
\disk_{n/M}^{\fr,\sf surj} \simeq \exit\bigl(\ran(M)\bigr)^{\op}
\]
with the opposite of the exit-path $\oo$-category of the Ran space of $M$. More generally, fix an embedding $S\subset \coprod_S\RR^n\hookrightarrow M$ for $S\subset M$. Then there is an equivalence
\[
\disk_{n/M}^{\fr,\sf surj}\underset{{\disk^{\fr}_{n/M}}}\times \disk_{n/M}^{\fr,\coprod_S\RR^n/^{\sf inj}}\simeq \exit\bigl(\ran(M)_S\bigr)^{\op}~.
\]
where $\disk_{n/M}^{\fr,\coprod_S\RR^n/^{\sf inj}}$ is the $\oo$-subcategory of $(\disk_{n/M})^{\coprod_S\RR^n/}$ whose objects are those $\coprod_S\RR^n\hookrightarrow U\hookrightarrow M$ such that $S\ra \pi_0U$ is injective.
\end{lemma}

\begin{lemma}[Lemma 2.3.2 \cite{pkd}]\label{lemma.surj.final}
The functor
\[
\disk^{\fr,\sf surj}_{n/M}\longrightarrow \disk^{\fr}_{n/M}
\]
is final.
\end{lemma}
\begin{proof}
By Quillen's Theorem A, it suffices to show that for each $g:\coprod_k \RR^n \hookrightarrow M$, an object of $\Disk_{n/M}$, the $\oo$-undercategory
\[
\disk^{\fr,\sf surj}_{n/M}\underset{\disk^{\fr}_{n/M}}\times\Bigl(\disk^{\fr}_{n/M}\Bigr)^{\coprod_k\RR^n/}
\]
has a contractible classifying space. There exists an adjunction
\[
\disk_{n/M}^{\fr,\sf surj}\underset{{\disk^{\fr}_{n/M}}}\times \disk_{n/M}^{\fr,\coprod_S\RR^n/^{\sf inj}}
\leftrightarrows
\disk^{\fr,\sf surj}_{n/M}\underset{\disk^{\fr}_{n/M}}\times\Bigl(\disk^{\fr}_{n/M}\Bigr)^{\coprod_k\RR^n/}
\]
hence an equivalence between their classifying spaces. By Lemma~\ref{lemma.surj.final}, we obtain
\[
\sB\Bigl(\disk_{n/M}^{\fr,\sf surj}\underset{{\disk^{\fr}_{n/M}}}\times \disk_{n/M}^{\fr,\coprod_S\RR^n/^{\sf inj}}
\Bigr)
\simeq
\sB\exit\bigl(\ran(M)_S\bigr)
\simeq
\ran(M)_S~.
\]
The result now follows from the contractibility of this form of the Ran space---to see this, we use a now standard argument from \cite{bd}: the Ran space carries a natural H-space structure, given by taking unions of subsets, for which the composition of the diagonal and the H-space multiplication is the identity. Consequently, its homotopy groups must be zero.
\end{proof}

\begin{remark}
Combining the equivalence
\[
\disk^{\fr,\sf surj}_{n/M} \simeq \exit\bigl(\ran(M)\bigr)^{\op}
\]
together with the finality of Lemma~\ref{lemma.surj.final}, we obtain that factorization homology (or any colimit indexed by $\disk^{\fr}_{n/M}$) is equivalent to the global sections of an associated cosheaf on the Ran space of $M$.
\end{remark}

Together with Lemma~\ref{lemma.layers}, the following is the main conclusion of this section, describing a complete filtration of factorization homology with layers given by twisted homologies of configuration spaces.

\begin{prop}\label{prop.trunc}
The natural map from the sequential colimit of truncations
\[
\varinjlim_k \tau_{\leq k} \int_M A \longrightarrow \int_M A
\]
is an equivalence.
\end{prop}
\begin{proof}
The finality of Lemma~\ref{lemma.surj.final} implies the equivalence 
\[
\int_MA:=\colim\Bigl(\Disk^{\fr}_{n/M}\xra{\int A} \cV\Bigr)\simeq \colim\Bigl(\Disk^{\fr,\sf surj}_{n/M}\xra{A} \cV\Bigr)~.
\]
There is then a sequence of equivalences
\[
\colim\Bigl(\Disk^{\fr,\sf surj}_{n/M}\xra{A} \cV\Bigr)\simeq \colim\Bigl(\Bigl(\varinjlim_k\Disk^{\fr,\leq k}_{n/M}\Bigr)\xra{A} \cV\Bigr) \simeq  \varinjlim_k \colim\Bigl(\Disk^{\fr,\leq k}_{n/M}\xra{A} \cV\Bigr) =\varinjlim_k\tau_{\leq k}\int_M A 
\]
from which the result follows.
\end{proof}

A dual result holds for factorization cohomology with coefficients in an $\cE_n$-coalgebra.

\begin{definition}\label{def.fact.cohomology}
For $C$ an $\cE_n$-coalgebra in $\cV$ (i.e., an $\cE_n$-algebra in $\cV^{\op}$), the factorization cohomology of $M$ with coefficients in $C$ is the limit
\[
\int^MC := \limit\Bigl(\Disk_{n/M}^{\op}\overset{C}\longrightarrow \cV\Bigr)~.
\]
\end{definition}

\begin{prop}
For $M$ a closed framed $n$-manifold and $C$ an $\cE_n$-coalgebra in $\cV$, there exists a cofiltration of $\int^MC$,
\[
\int^MC \longrightarrow \tau^{\leq \bullet}\int^MC
\]
whose associated graded is
\[
\coprod_k \Bigl(\bigl(C^{\ot k}\bigr)^{\conf_k(M)^+}\Bigr)^{\Sigma_k}~.
\]
\end{prop}

\subsection{Goodwillie filtrations}

Fixing a manifold $M$, one can apply the Goodwillie functor calculus, as developed in \cite{goodwillie}, \cite{kuhn1}, and \cite{dag}, to the functor
\[
\int_M: \Alg_{\cE_n}^{\sf aug}(\cV) \longrightarrow \cV~.
\]
The output is a cofiltration of factorization homology whose $k$th term is the polynomial approximation $P_k\int_M$.

\begin{prop}[\cite{pkd}]\label{pkd} 
Let $A$ an augmented $\cE_n$-algebra in $\cV$ with cotangent space $LA$.
For each closed framed $n$-manifold $M$, the following is a fiber sequence
\[\xymatrix{
\conf_k(M) \underset{\Sigma_k}\ot LA^{\ot k} \ar[d]\ar[r]&P_k \displaystyle\int_M A \ar[d]\\
\uno\ar[r]& P_{k-1} \displaystyle\int_MA\\}\] where $P_k$ is Goodwillie's $k$th polynomial approximation \cite{goodwillie} applied to the functor $\int_M$, and $LA$ is the cotangent space of $A$ at the augmentation (see \cite{thez}). In particular, the Goodwillie derivative $\partial_k\int_M$ is canonically equivalent with $\conf_k(M)\otimes \uno$, the tensor of the space $\conf_k(M)$ with the unit object in the symmetric monoidal $\infty$-category $\cV$.
\end{prop}

\section{Poincar\'e/Koszul duality}

Koszul duality, in its most basic form, interchanges two forms of algebra, determined by the form of algebraic structure that the cotangent space at an augmentation naturally inherits. There are two classical forms of Koszul duality: commutative algebra is Koszul dual to Lie algebra; associative algebra is dual to associative algebra. In particular, associative algebra is Koszul self-dual. This second assertion generalizes to the situation of $\cE_n$-algebra, where the Koszul self-duality dates to Getzler--Jones~\cite{getzlerjones}. This Koszul duality for $\cE_n$-algebra can be expressed in terms of factorization homology for manifolds with boundary. See \cite{lurie.formal} for another treatment of Koszul self-duality of $\cE_n$-algebra, in terms of twisted arrow categories.

\begin{theorem}[\cite{pkd}, \cite{zp}]\label{thm.firstpkd}
Let $\cV$ be a stable $\ot$-presentable $\oo$-category. There is a functor
\[
\int_{\DD^n/\partial\DD^n}:\Alg_{\cE_n}^{\sf aug}(\cV) \longrightarrow \cAlg_{\cE_n}^{\sf aug}(\cV)
\]
sending an augmented $\cE_n$-algebra $A$ to an $\cE_n$-coalgebra structure on
\[
\int_{\DD^n/\partial\DD^n} A
~ \simeq ~
A \underset{\displaystyle\int_{S^{n-1}\times \RR}A}\bigotimes \uno~.
\]
There is a natural transformation, the Poincar\'e/Koszul duality map, in $\Fun(\NN^{\triangleright},\cV)$:
\[
\xymatrix{
P_\bullet \displaystyle\int_M A \ar[d]\ar[r]&\ar[d] \tau^{\leq \bullet}\displaystyle\int^M\Bigl( \displaystyle\int_{\DD^n/\partial\DD^n}A\Bigr)\\
\displaystyle\int_MA \ar[r]&\displaystyle\int^M\Bigl( \int_{\DD^n/\partial\DD^n}A\Bigr)~.}
\]
This map is an equivalence when restricted to $\Fun(\NN,\cV)$.
\end{theorem}
\begin{remark}
The coalgebra structure of $\int_{\DD^n/\partial\DD^n}A$ is given in \cite{zp} by showing that factorization homology with coefficients in an augmented $\cE_n$-algebra is functorial with respect to the fold map
\[
\DD^n/\partial\DD^n\longrightarrow \DD^n/\partial\DD^n\vee \DD^n/\partial\DD^n~.
\]
More generally, the theory of zero-pointed manifolds of \cite{zp} exhibits an additional functoriality for factorization homology of augmented algebras, of extension-by-zero maps.
\end{remark}

In particular, the Goodwillie completion $\varprojlim P_\bullet \int_MA$ is equivalent to factorization cohomology with the coefficients in the coalgebra which is Koszul-dual to $A$. The natural map to the completion
\[
\int_MA\longrightarrow \varprojlim P_\bullet \int_MA
\]
need not be an equivalence, however.\footnote{One should only expect this map to be an equivalence if $\cV$ carries a t-structure, suitably compatible with the monoidal structure, for which the augmentation ideal of $A$ is either connected or $n$-coconnected.} A generalization of factorization homology, taking coefficients in formal moduli problems, gives a way to correct the failure of Poincar\'e/Koszul duality to be an equivalence in general.

\begin{definition}\label{stack} 
Let $\cV$ be a $\ot$-presentable $\infty$-category.
Let $B\to \BO(n)$ be a map from a space.
Let $X\colon \Alg_{\disk_n^B}(\cV) \to \spaces$ be a functor.
Let $M$ be a $B$-framed $n$-manifold $M$.
The factorization homology of $M$ with coefficients in $X$ is the object in $\cV$
\[
\int_MX~ :=~ \limit_{A\in{\sf Aff}^{\op}_{/X}}\int_M A~,
\] 
where ${\sf Aff} := \Alg_{\disk^B_n}(\cV)^{\op}\subset \Fun\bigl(\Alg_{\disk_n^B}(\cV),\spaces\bigr)$ is the image of the Yoneda embedding.
\end{definition}

\begin{remark} Intuitively, the object $\int_M X$ is $\Gamma(X,\int_M\cO)$, the global sections of the presheaf on $X$ obtained by applying factorization homology of $M$ to the structure sheaf of $X$. 
We interpret this form of factorization homology as the observables in a topological quantum field theory: a sigma-model with algebraic target $X$.  
\end{remark}

There is a modification of the above definition for a {\it formal} moduli problem, i.e., a space-valued functor on an $\oo$-category of Artin $\cE_n$-algebras.

\begin{definition}[${\sf Artin}_{\cE_n}$]
The full $\oo$-subcategory ${\sf Triv}_{\cE_n}\subset \Alg_{\cE_n}^{\sf aug,\geq 0}(\m_\Bbbk)$ is the essential image of connective perfect $\Bbbk$-modules ${\sf Perf}_\Bbbk^{\geq 0}$ under the functor that assigns a complex $V$ the associated square-zero extension $\Bbbk\oplus V$.
The $\oo$-category ${\sf Artin}_{\cE_n}$ of local Artin $\cE_n$-algebras is the smallest full $\oo$-subcategory of $\Alg_{\cE_n}^{\sf aug,\geq 0}(\m_\Bbbk)$ that contains ${\sf Triv}_{\cE_n}$ and is closed under small extensions. That is:
\begin{itemize}
\item[~]
If $B$ is in ${\sf Artin}_{\cE_n}$, $\Bbbk\oplus V$ is in ${\sf Triv}_{\cE_n}$, and the following diagram
\[
\xymatrix{
A   \ar[d]     \ar[r]
&
B   \ar[d]
\\
\Bbbk   \ar[r] 
&
\Bbbk\oplus V
}
\] 
forms a pullback square in $\Alg_{\cE_n}^{\sf aug, \geq 0}(\m_\Bbbk)$, then $A$ is in ${\sf Artin}_{\cE_n}$.
\end{itemize}
\end{definition}

\begin{definition}\label{def.facthom.formal}
Fix a field $\Bbbk$, and let $X\colon {\sf Artin}_{\cE_n}\to \spaces$ be a functor of local Artin $\cE_n$-algebras over $\Bbbk$.
The factorization homology of $X$ over $M$ is the $\Bbbk$-module
\[
\int_MX~ :=~ \limit_{R\in{\sf Artin}_{\cE_n/X}^{\op}}\int_M R~.
\] 
\end{definition}

Such formal moduli problems $X$ are defined by the generalization of the Maurer--Cartan functor from classical Lie theory.  In the following definition, we use the Koszul duality functor $\DD^n$ given as the composite of linear duality with factorization homology of $\DD^n/\partial\DD^n$.
\[
\xymatrix{
\DD^n: \Alg_{\cE_n}^{\sf aug}(\m_\Bbbk)\ar[r]&\cAlg_{\cE_n}^{\sf aug}(\m_\Bbbk)\ar[r]&\Alg_{\cE_n}^{\sf aug}(\m_\Bbbk)^{\op}}~.
\]
That is, the Koszul-dual $\cE_n$-algebra is
\[
\DD^n A = \Bigl(\int_{\DD^n/\partial\DD^n}A\Bigr)^\vee~.
\]
\begin{definition}\label{def.MC}
For an Artin $\cE_n$-algebra $R$ and an augmented $\cE_n$-algebra $A$ over a field $\Bbbk$, the Maurer--Cartan space \[\MC_{A}(R) :=\Alg_{\cE_n}^{\sf aug}(\DD^nR, A)\] is the space of maps from the Koszul dual of $R$ to $A$. The Maurer--Cartan functor $\MC \colon \Alg_{\cE_n}^{\sf aug} \to \Fun({\sf Artin}_{\cE_n},\Spaces)$ is the adjoint of the pairing
\[
{\sf Artin}_{\cE_n} \times\Alg_{\cE_n}^{\sf aug}(\m_\Bbbk) \xra{\DD^n \times {\sf id}} \Alg_{\cE_n}^{\sf aug}(\m_\Bbbk)^{\op} \times \Alg_{\cE_n}^{\sf aug}(\m_\Bbbk) \xra{\Map(-,-)} \Spaces~.
\]
$\MC$ sends $A$ to the functor $\MC_A:{\sf Artin}_{\cE_n}\ra \spaces$.
\end{definition}

The following is the framed version of the main theorem of \cite{pkd}.

\begin{theorem}[Poincar\'e/Koszul duality \cite{pkd}]\label{1main} 
Let ${M}$ be a closed framed $n$-manifold; let $A$ be an augmented $\cE_n$-algebra over a field $\Bbbk$ with $\MC_A$ the associated formal moduli functor of $\cE_n$-algebras. There is a natural equivalence 
\[
\Bigl(\int_{M}A\Bigl)^\vee~ \simeq~ \int_{M}\MC_A
\] 
between the $\Bbbk$-linear dual of the factorization homology of ${M}$ with coefficients in $A$ and the factorization homology of ${M}$ with coefficients in the moduli functor $\MC_A$.
\end{theorem}
\begin{remark}
One can regard the failure of the duality map in Theorem~\ref{thm.firstpkd} to be an equivalence to be related to the non-affineness of the formal moduli problem ${\sf MC}_A$. Were ${\sf MC}_A$ to be affine, it would be equivalent to the formal spectrum of the Koszul dual, ${\sf Spf}(\DD^nA)$, in which the case the duality map of Theorem~\ref{thm.firstpkd} would be an equivalence. In terms of topological quantum field theory, we regard ${\sf MC}_A$ as determining a sigma-model whose target $X={\sf MC}_A$ is formal, but not necessarily affine. Theorem~\ref{1main} then describes the observables $\int_{M}X$ in this sigma-model in terms of the factorization homology with coefficients in a shift of the cotangent space of the point ${\sf Spec}(\Bbbk)\ra X$.
\end{remark}

\section{Factorization homology for singular manifolds}\label{sec.singular}
We briefly outline an extension of factorization homology to singular manifolds, such as manifolds equipped with \emph{defects}.  
The notion of singular manifolds, and structured versions thereof, is developed in~\cite{aft1}.
The development of factorization homology of such can be found in~\cite{aft2}.

\subsection{Singular manifolds}

Heuristically, a singular manifold is a stratified space in which each stratum is endowed with the structure of a smooth manifold, and in which each stratum has a canonically associated link which is itself a singular manifold.   
Informally, the $\infty$-category $\snglr$ of singular manifolds and open (stratified) embeddings among such is minimal with respect to the following features.
\begin{enumerate}
\item
The empty set $\emptyset$ is a singular manifold.

\item
For $L$ a compact singular manifold, the open cone 
\[
\sC(L)~:=~ \ast \underset{L\times \{0\}} \coprod L\times [0,1)
\]
is a singular manifold, with the cone-point $\ast$ a stratum.  

\item
For $X$ and $Y$ singular manifolds, their product $X\times Y$ is a singular manifold.

\item
For $U\subset X$ an open subset of a singular manifold, then $U$ is canonically a singular manifold.

\item
For $\cU$ an open cover by singular manifolds of a paracompact Hausdorff space $X$, then $X$ is a singular manifold.

\item
For $L$ a compact singular manifold, the continuous homomorphism
\[
\Aut(L)
\xra{~\simeq~}
\Aut\bigl(\sC(L)\bigr)
~,\qquad
f\mapsto \sC(f)~,
\]
is equipped with a deformation retraction whose retraction $D_\ast \colon \Aut\bigl(\sC(L)\bigr) \to \Aut(L)$ is a continuous homomorphism.  

\end{enumerate}
A rigorous definition of $\snglr$ is given in~\S1 of~\cite{aft1}.
That definition is distractingly inductive, and so we elect to not spell it out here.

\begin{terminology}\label{size.singularity}
In this article, by ``singular manifold'' we will mean ``finitary singular manifold with boundary," in the sense of~\cite{aft1}.

\end{terminology}

\begin{remark}
The first two axioms give that $\ast$ is a singular manifold, and thereafter that $\sC(\ast) = [0,1)$ is a singular manifold with strata $\{0\}$ and $(0,1)$.  
The third and fifth axiom gives that $(0,1)^{\times n}\cong \RR^n$ is a singular manifold.  
The fifth axiom thereafter gives that each manifold is a singular manifold.
The second and third axiom then give that $\RR^k\times \sC(L)$ is a singular manifold for each compact manifold $L$.
In general, an arbitrary singular manifold is a paracompact Hausdorff topological space that admits a basis for its topology consisting of open embeddings from singular manifolds of the form $\RR^k\times \sC(L)$, where $k$ is an integer and $L$ is a compact singular manifold.

\end{remark}

A \emph{basic} is a singular manifold of the form $\RR^k\times \sC(L)$ for $k$ an integer and $L$ a compact singular manifold.

\begin{remark}
We follow up on the previous remark.
Namely, the definition of a singular manifold detailed in~\cite{aft1} is in analogy with that of a smooth manifold: it is a paracompact Hausdorff topological space equipped with a maximal atlas by \emph{basics}. 

\end{remark}

\begin{remark}\label{blow.up}
The sixth axiom above reflects the regularity of singular manifolds.  
The retraction map can be interpreted as taking the `derivative at the cone-point'.
A consequence of this structure is that one can `resolve singularities' of a singular manifold.
Specifically, for $X$ a singular manifold, and for $X_d\subset X$ a stratum that is closed as a subspace, there is a well-defined \emph{blow-up} of $X$ along $X_d$, as well as \emph{link} of $X_d$ in $X$, which assemble as the pullback and pushout diagram consisting of singular manifolds and continuous stratified maps among them:
\[
\xymatrix{
{\sf Link}_{X_d}(X)  \ar[rr]  \ar[d]
&&
{\sf Bl}_{X_d}(X)  \ar[d]
\\
X_d \ar[rr]
&&
X .
}
\]
This diagram has the feature that the left vertical map is a bundle of singular manifolds, while the right vertical map restricts as a bundle of singular manifolds over each stratum of $X$, and over the complement $X\smallsetminus X_d$ the right vertical map is an isomorphism.  

\end{remark}

\begin{example}\label{ex.boundary}
Let $\overline{M}$ be an $n$-manifold with boundary.
Regard $\overline{M}$ as a stratified space, whose strata are the connected components of the boundary $\partial \overline{M}$ and the connected components of of the interior $M:= \overline{M}\smallsetminus \partial \overline{M}$.  
As a stratified space, $\overline{M}$ is a singular manifold.
Indeed, it admits an atlas by basics of the form $\RR^n$ and $\RR^{n-1}\times \sC(\ast)$.

\end{example}

\begin{example}\label{ex.defects}
Let $M$ be an $n$-manifold.
Let $W\subset M$ be a properly embedded $d$-submanifold.
Consider the stratified space $(W\subset M)$ whose strata are the connected components of $W$ and the connected components of $M\smallsetminus W$.  
This stratified space $(W\subset M)$ is a singular manifold.
Indeed, it admits an atlas by basics of the form $\RR^n$ and $\RR^d\times \sC(S^{n-d-1})$.  
A singular manifold of this form is a \emph{defect}, or \emph{refinement}, of the manifold $M$.  
In particular, a knot in a 3-manifold determines a stratified space that refines onto the given 3-manifold.  

\end{example}

\begin{example}
Following up on Example~\ref{ex.defects}, consider a sequence $W_0\subset W_1 \subset M$ of properly embedded $d_0$-submanifold in a properly embedded $d_1$-submanifold in an $n$-manifold.  
Consider the stratified space $(W_0\subset W_1 \subset M)$ whose strata are the connected components of $W_0$, the connected components of $W_1\smallsetminus W_0$, and the connected components of $M\smallsetminus W_1$.  
This stratified space $(W_0 \subset W_1 \subset M)$ is a singular manifold.
Indeed, it admits an atlas by basics of the form $\RR^n$, $\RR^{d_1}\times \sC(S^{n-{d_1}-1})$, and $\RR^{d_0}\times \sC(\SS^{n-d_0-1})$ where $\SS^{n-d_0-1} = (S^{d_1-d_0-1} \subset S^{n-d_0-1})$ is the singular manifold of Example~\ref{ex.defects} applied to an equatorially embedded sphere. 

\end{example}

\begin{example}\label{ex.simplicial}
Let $p\geq 0$, and consider the topological simplex $\Delta^p$.
Consider the stratification of $\Delta^p$ in which the $i$-dimensional strata comprise the connected components of the complement of skeleta ${\sf sk}_i(\Delta^n) \smallsetminus {\sf sk}_{i-1}(\Delta^n)$.
This stratified space is a singular manifold.
Indeed, it admits an atlas by basics of the form $\RR^{n-p}\times \sC(\Delta^{p-1})$ for $0\leq p\leq n$. 
As a consequence, the geometric realization $|Z|$ of a finite simplicial complex is canonically endowed with the structure of a singular manifold.  

\end{example}

\begin{example}
The Grassmannian ${\sf Gr}_k(n)$ of $k$-planes in $\RR^n$ admits the \emph{Schubert} stratification, for which, for each cardinality $k$-subset $S\subset\{1,\dots,n\}$, the $S$-stratum consists of those $V$ for which $S$ is minimal (in the Schubert poset) for which the linear map $V \hookrightarrow \RR^n \xra{\sf proj} \RR^S$ is an isomorphism.  
This Schubert stratification of ${\sf Gr}_k(n)$ is a singular manifold.  

\end{example}

\begin{example}
The orthogonal group $\sO(n)$ admits the \emph{Bruhat} stratification, for which, for each element $\varphi \in \Sigma_n\wr (\ZZ/2\ZZ) = \Aut(n{\sf Cube})\subset \sO(n)$ of the group of symmetries of the $n$-cube, the $\varphi$-stratum consists of those orthogonal matrices that, via row elimination and scaling by positive scalars, have $\varphi$ as its matrix of pivots.  
This Bruhat stratification of $\sO(n)$ is a singular manifold. 

\end{example}

Recall from~\S\ref{sec.sheaf.classify} that we classified sheaves on the $\infty$-category $\mfld_n$ as space over $\BO(n)$; for $B\to \BO(n)$ such a thing, the space of sections of the associated sheaf on an $n$-manifold $M$ is the space of lifts of the tangent classifier:
\[
\xymatrix{
&&
B \ar[d]
\\
M \ar[rr]^-{\tau_M}  \ar@{-->}[urr]^-{\varphi}
&&
\BO(n)  .
}
\]
This was the context supporting the constant consideration of spaces over $\BO(n)$ in the developments above.

In \cite{aft1} there appears a similar classification.
Namely, consider the full $\infty$-subcategory $\bsc\subset \snglr$ consisting of the basic singular manifolds.  
\begin{theorem}[\cite{aft1}]\label{singular.sheaves}
Restriction along the inclusion $\bsc\hookrightarrow \snglr$ defines an equivalence between $\infty$-categories
\[
\shv(\snglr)
\xra{~\simeq~}
\Psh(\bsc)~.
\]

\end{theorem}

Through the straightening-unstraightening equivalence of~\S2 of~\cite{topos}, Theorem~\ref{singular.sheaves} gives that a sheaf on $\snglr$ is equivalent data to a right fibration among $\infty$-categories:
\[
\cB
\longrightarrow
\bsc~;
\]
a section of the space of sections of the associated sheaf on a singular manifold $X$ is the space of lifts
\[
\xymatrix{
&&
\cB \ar[d]
\\
\exit(X)^{\op} \ar[rr]^-{\tau_X}  \ar@{-->}[urr]^-{\varphi}
&&
\bsc  ;
}
\]
we call such a lift a \emph{$\cB$-structure}.  
Here, $\exit(X)$ is the exit-path $\infty$-category of $X$, and the functor $\tau_X$ carries a point in $X$ to a basic neighborhood of it.  
In~\cite{aft1} it is shown that the functor $\exit(X)^{\op} \xra{\tau_X} \bsc$ is the unstraightening of the presheaf $\bsc^{\op} \xra{U\mapsto \Emb(U,X)} \spaces$.  
In particular, $\exit(X)^{\op}\to \bsc$ is a right fibration.

\begin{example}
Let $B\to \BO(n)$ be a map from a space.
The composite map $B\to \BO(n)\to \bsc$ is a right fibration.
A singular manifold $X$ admits a $B$-structure if and only if $X$ is an $n$-manifold.
Furthermore, should $X$ be an $n$-manifold, the space $B$-structures is the space of $B$-framings on that $n$-manifold.  

\end{example}

\begin{example}
Let $\cD_{d\subset n} \subset \bsc$ be the full $\infty$-subcategory consisting of the two objects $\RR^d\times \sC(S^{n-d-1})$ and $\RR^n$.
The inclusion $\cD_{d\subset n} \hookrightarrow \bsc$ is a right fibration.
A singular manifold $X$ admits a $\cD_{d\subset n}$-structure if and only if the singular manifold is of the form $(W^d\subset M^n)$, a properly embedded $d$-submanifold in an $n$-manifold (see Example~\ref{ex.defects}).  
Furthermore, if $X$ admits a $\cD_{d\subset n}$-structure, then it is unique.  

\end{example}

\begin{example}
Choose, once and for all, an orientation-preserving diffeomorphism $\RR\cong (0,1)$.
Consider the right fibration $\cD_{n-1,n}^{\fr} := \{n-1<n\} \to \bsc$ that selects out the morphism $\RR^n = \RR^{n-1}\times \RR \cong \RR^{n-1}\times (0,1)\hookrightarrow \RR^{n-1}\times \sC(\ast)$, which is the standard open embedding of the complement of the cone-locus.  
A $\cD_{n-1,n}^{\fr}$-manifold is an $n$-manifold with boundary, equipped with a framing of its interior and a splitting of its framing along the boundary for which the last vector field points inward.

\end{example}

\begin{example}\label{ex.fr.defects}
Let $\cD_{d\subset n}^{\fr} \subset \bsc$ be the right fibration $\exit\bigl(\RR^d\times \sC(S^{n-d-1})\bigr)^{\op} \xra{\tau} \bsc$.  
A singular manifold $X$ admits a $\cD_{d\subset n}^{\fr}$-structure if and only if the singular manifold is of the form $(W^d\subset M^n)$, a properly embedded $d$-submanifold in an $n$-manifold (see Example~\ref{ex.defects}).  
Furthermore, for $X = (W \subset M)$ such a singular manifold, a $\cD_{d\subset n}^{\fr}$-structure on it is a framing of $M$ and a splitting of the framing along $W$ via the first $d$-coordinates.  

\end{example}

\begin{example}\label{ex.fr.general}
The representable right fibrations over $\bsc$ are particularly tractable.  
Let $U =\RR^d\times \sC(L)$ be a basic.
Notate the right fibration $\exit(U)^{\op} \xra{\tau_U} \bsc$ as $\cD_U \to \bsc$.
A singular manifold $X$ admits a $\cD_{U}$-structure if and only if $X$ admits an atlas by open subsets of $U$.  
Furthermore, for $X$ such a singular manifold, a $\cD_{U}$-structure on $X$ determines the structure of a framed $d$-manifold on the lowest-dimensional strata $X_d$ of $X$, as well as a trivialization of the link, 
\[
{\sf Link}_{X_d}(X)\underset{\rm over ~X_d}{~\cong~} X_d\times L~,
\]
and thereafter, upon the choice of a collar-neighborhood of ${\sf Link}_{X_d}(X)$ in the blow-up ${\sf Bl}_{X_d}(X)$, an open stratified embedding 
\[
X_d\times L \times \RR\hookrightarrow X\smallsetminus X_d~.
\]

\end{example}

\begin{observation}\label{DU.restrict}
Let $U = \RR^d\times \sC(L)$ be a basic.
Consider the right fibration $\cD_U\to \bsc$ of Example~\ref{ex.fr.general}.
Consider the full $\infty$-subcategory $\exit(U\smallsetminus \RR^d)^{\op} := \cD_{>U} \to \cD_U$ consisting of those objects, which are points in $U$, that do not belong to the cone-locus $\RR^d$.  
The composite functor $\cD_{>U}\to \cD_U \to \bsc$ is again a right fibration.
For $X$ a singular manifold equipped with a $\cD_U$-structure, the complement $X\smallsetminus X_d$ of its lowest-dimensional strata is a singular manifold, and it canonically inherits a $\cD_{>U}$-structure from the given $\cD_U$-structure on $X$.
Through the conclusion of Example~\ref{ex.fr.general}, the product $\RR^d\times L\times \RR\cong \RR^{d+1}\times L$ is canonically endowed with a $\cD_{>U}$-structure.
Thereafter, taking products with Euclidean-$(d+1)$-spaces defines a symmetric monoidal functor 
\begin{equation}\label{50}
\disk_{d+1}^{\fr} \xra{~-\times L~} \mfld(\cD_{>U})
~,\qquad
(\RR^{d+1})^{\sqcup I}
~ \mapsto  ~
(\RR^{d+1})^{\sqcup I} \times L \cong (\RR^d \times L \times \RR)^{\sqcup I}~.
\end{equation}

\end{observation}

\begin{definition}
Let $\cB\to \bsc$ be a right fibration.
The $\infty$-category $\mfld(\cB)$ of \emph{$\cB$-manifolds} is the pullback among $\infty$-categories:
\[
\xymatrix{
\mfld(\cB)  \ar[rrrr]  \ar[d]
&&
&&
\Psh(\bsc)_{/\cB}  \ar[d]
\\
\snglr  \ar[rrrr]^-{\tau\colon X\mapsto \bigl(\exit(X)^{\op}\xra{\tau_X}\bsc\bigr)}
&&
&&
\Psh(\bsc) .
}
\]
The $\infty$-category of \emph{$\cB$-disks} is the full $\infty$-subcategory
\[
\disk(\cB)
~\subset~
\mfld(\cB)
\]
consisting of those $\cB$-structured singular manifolds $X$ for which $X$ is a finite disjoint union of basics.  
Disjoint union of underlying singular manifolds makes both $\disk(\cB)$ and $\mfld(\cB)$ into symmetric monoidal $\infty$-categories.

\end{definition}

So a $\cB$-manifold is a pair $(X,\varphi)$ consisting of a singular manifold together with a $\cB$-structure on it.
The space of morphisms between two such are open (stratified) embeddings that respect $\cB$-structures.

\begin{example}\label{ex.U.boundary}
Consider the basic $U = \sC(\Delta^{n-1})$, where $\Delta^{n-1}$ is the compact singular manifold of Example~\ref{ex.simplicial}.  
A $\cD_U$-manifold is an $n$-manifold $\overline{M}$ with corners equipped with a framing as well as compatible splittings of the framing along each face.  
We call an $n$-manifold with corners equipped with such framing data a \emph{framed $\langle n\rangle$-manifold}.
(This notion of $\langle n\rangle$-manifolds, which are $n$-manifolds with corners with certain corner structure, is thoroughly developed in~\cite{laures}.)
For this case of $U$, we use the simplified notation:
\[
\cD_{\langle n\rangle}^{\fr} := \cD_U
\qquad\text{ and }\qquad
\disk_{\langle n\rangle}^{\fr}~:=~\disk(\cD_{\langle n\rangle}^{\fr})~.
\]

\end{example}

\subsection{Homology theories for structured singular manifolds}\label{sec.snglr.hmlgy}
In analogy with~\S\ref{sec.hmlgy}, one can define factorization homology of a $\disk(\cB)$-algebra over a $\cB$-manifold, and factorization homology can be characterized as homology theories for $\cB$-manifolds.

In this subsection, we fix a right fibration $\cB\to \bsc$, as well as a symmetric monoidal $\infty$-category $\cV$ that is $\ot$-presentable.

\begin{definition}
The $\infty$-category of \emph{$\disk(\cB)$-algebras} in $\cV$ is 
\[
\Alg_{\disk(\cB)}(\cV)~:=~\Fun^{\ot}\bigl(\disk(\cB) , \cV\bigr)~.
\]

\end{definition}

\begin{remark}
Recall the identification with the symmetric monoidal envelope ${\sf Env}(\cE_n) \simeq \disk_n^{\fr}$.
This identification can be read as $\disk_n^{\fr}$ being initial among symmetric monoidal $\infty$-categories equipped with an $\cE_n$-algebra therein.  
There is a likewise identification ${\sf Env}\bigl(\cE(\cB)\bigr)\simeq \disk(\cB)$ where $\cE(\cB)$ is an $\infty$-operad whose underlying $\infty$-category of 1-ary operations is $\cB$, and whose space $I$-ary morphisms from $(U_i)_{i\in I}\in \cB^I$ to $V\in \cB$ is the space of $\cB$-structured embeddings from the $I$-fold disjoint union: $\Map_{\disk(\cB)}\bigl(\underset{i\in I} \bigsqcup U_i  , V\bigr)$.  

\end{remark}

The Definition~\ref{coend} of factorization homology for $n$-manifolds is adequately formal, defined through a universal property, to imitate for singular manifolds.
\begin{definition}\label{def.snglr.fact}
For $A\in \Alg_{\disk(\cB)}(\cV)$, and $X$ a $\cB$-manifold, the \emph{factorization homology} of $A$ over $X$ is the colimit in $\cV$,
\[
\int_X A~:=~\colim\bigl( \disk(\cB)_{/X} \to \disk(\cB) \xra{~A~} \cV\bigr)~,
\]
should it exist.  
\end{definition}

Restriction along the fully-faithful symmetric monoidal functor $\iota\colon \disk(\cB) \hookrightarrow \mfld(\cB)$ defines a functor
\begin{equation}\label{51}
\Alg_{\disk(\cB)}(\cV)
\longleftarrow
\Fun^{\ot}\bigl(\mfld(\cB) , \cV\bigr)  \colon \iota^\ast~.
\end{equation}
\begin{prop}\label{def.singular.fact}
Factorization homology exists, and defines a left adjoint to~(\ref{51}):
\[
\int \colon \Alg_{\disk(\cB)}(\cV)
~\rightleftarrows~
\Fun^{\ot}\bigl(\mfld(\cB) , \cV\bigr) \colon \iota^\ast~.
\] 
Furthermore, the a priori lax-commutative diagram among $\infty$-categories involving a left adjoint $\iota_!$ of the restriction functor $\iota^\ast$,
\[
\xymatrix{
\Alg_{\disk(\cB)}(\cV)  \ar[rr]^-{\int}  \ar[d]_-{\rm forget}
&&
\Fun\bigl(\disk(\cB) , \cV \bigr)  \ar[d]^-{\rm forget}
\\
\Fun\bigl(\disk(\cB) , \cV\bigr)  \ar[rr]^-{\iota_!}
&&
\Fun\bigl(\mfld(\cB) , \cV\bigr) ,
}
\]
is in fact a commutative diagram. 

\end{prop}

The Definition~\ref{def.collar-gluing} of a collar-gluing for $n$-manifolds can be imitated for singular manifolds: for $X$ a singular manifold, a \emph{collar-gluing} of $X$ is a continuous map $f\colon X\to [-1,1]$ for which the restriction $f_{|}\colon f^{-1}(-1,1) \to (-1,1)$ is a fiber bundle among singular manifolds.  
We denote such a collar-gluing as $X_- \underset{X_0\times \RR}\bigcup X_+ \cong X$ where we understand that $X_- := f^{-1}[-1,1)$, $X_+ := f^{-1}(-1,1]$, and $X_0 := f^{-1}(0)$, while $X_0\times \RR \cong X_0 \times (-1,1) \cong f^{-1}(-1,1)$.
Such a collar-gluing of a $\cB$-manifold $X$ determines an algebra in the symmetric monoidal $\infty$-category $\mfld(\cB)$ over the multi-category ${\sf Assoc}^{\sf RL}$:
\[
\bigl( X_0 ; X_+ , X_-\bigr)
~\in~\Alg_{{\sf Assoc}^{\sf RL}}\bigl(\mfld(\cB)\bigr)~.
\]
Thereafter, each symmetric monoidal functor $\cF\colon \mfld(\cB) \to \cV$ determines, for each such collar-gluing, a canonical morphism in $\cV$ from the 2-sided bar construction:
\begin{equation}\label{52}
\cF(X_-) \underset{\cF(X_0)}\ot \cF(X_+)
\longrightarrow
\cF(X)~.
\end{equation}

\begin{definition}\label{def.homology.singular}
The $\infty$-category of \emph{homology theories} for $\cB$-manifolds is the full $\infty$-subcategory 
\[
{\mathcal{\bf H}}\bigl(\mfld(\cB) , \cV\bigr)
~\subset~
\Fun^{\ot}\bigl(\mfld(\cB) , \cV \bigr)
\]
consisting of those symmetric monoidal functors $\cF$ that are \emph{$\ot$-excisive}:
\begin{itemize}
\item[~] 
For which, for each collar-gluing $X_- \underset{X_0\times \RR}\bigcup X_+ \cong X$ of a $\cB$-manifold, the canonical morphism in $\cV$,
\[
\cF(X_-) \underset{\cF(X_0)}\ot \cF(X_+)
\xra{~(\ref{52})~}
\cF(X)
\]
is an equivalence.
\end{itemize}

\end{definition}

The following foundational characterizing factorization homology is result is proved in~\cite{aft2}.
\begin{theorem}\label{singular.alpha.thm}
Factorization homology defines an equivalence between $\infty$-categories,
\[
\int\colon \Alg_{\disk(\cB)}(\cV)
~\simeq~
{\mathcal{\bf H}}\bigl(\mfld(\cB) , \cV\bigr) ~,
\]
with inverse given by restriction.

\end{theorem}

\subsection{Characterizing some $\disk(\cD_U)$-algebras}
Recall from~\S\ref{sec.structures} that, for $B=\ast$ so that a $B$-framing on an $n$-manifold is a framing thereof, then a $\disk_n^{\fr}$-algebra was precisely an $\cE_n$-algebra, which is a rather algebraic entity.  
In the same spirit, we characterize some $\disk(\cB)$-algebras in algebraic terms.  
Through Theorem~\ref{singular.alpha.thm}, this gives algebraic input for invariants of $\cB$-manifolds.

In this subsection, we fix a commutative ring spectrum $\Bbbk$, and consider its symmetric monoidal $\infty$-category $(\m_{\Bbbk},\underset{\Bbbk}\ot)$ of $\Bbbk$-modules and tensor product over $\Bbbk$ among them.

The next result is an immediate consequence of the main result of~\cite{thomas} after the observation that $\disk(\cD_{n-1,n}^{\fr})$ is the symmetric monoidal envelope of the swiss cheese operad.  
To state the result, recall that, for $B$ a $\cE_k$-algebra in $(\m_{\Bbbk}, \underset{\Bbbk}\ot)$, its (derived) center is
\[
\sZ(B)~:=~\Hom_{\int_{S^{k-1}}B}(B,B)~,
\]
the endomorphisms of $B$ as a module over the $\cE_1$-algebra $\int_{S^{k-1}} B$.
Deligne's conjecture, as stated and proved in~\S5 of~\cite{dag}, endows $\sZ(B)$ with a canonical structure of a $\cE_{k+1}$-algebra structure in $(\m_{\Bbbk}, \underset{\Bbbk}\ot)$.
\begin{prop}\label{identify.boundary}
Recall from Example~\ref{ex.boundary} the right fibration $\cD_{n-1,n}^{\fr}\to \bsc$.
A $\disk(\cD_{n-1,n}^{\fr})$-algebra in $(\m_{\Bbbk}, \underset{\Bbbk}\ot)$ is equivalent to the following data.
\begin{enumerate}
\item
An $\cE_n$-algebra $A$ in $(\m_{\Bbbk},\underset{\Bbbk}\ot)$.

\item
An $\cE_{n-1}$-algebra $B$ in $(\m_{\Bbbk},\underset{\Bbbk}\ot)$.

\item
An action of $A$ on $B$, instantiated as a morphism between $\cE_n$-algebras
\[
A 
\longrightarrow
\sZ(B)
\]
to the (derived) center of $B$.

\end{enumerate}

\end{prop}

The above result is typical of its kind.  
The next result appears in~\cite{aft2}.
\begin{prop}
Consider the right fibration $\cD_{d\subset n}^{\fr}\to \bsc$ of Example~\ref{ex.fr.defects}.
A $\disk(\cD_{d\subset n}^{\fr})$-algebra in $(\m_{\Bbbk}, \underset{\Bbbk}\ot)$ is equivalent to the following data.
\begin{enumerate}
\item
A $\cE_n$-algebra $A$ in $(\m_{\Bbbk},\underset{\Bbbk}\ot)$.

\item
A $\cE_d$-algebra $B$ in $(\m_{\Bbbk},\underset{\Bbbk}\ot)$.

\item
An action of $\int_{S^{n-d-1}} A$ on $B$, instantiated as a morphism between $\cE_{d+1}$-algebras
\[
\int_{S^{n-d-1}} A
\longrightarrow
\sZ(B)~.
\]

\end{enumerate}

\end{prop}

\begin{example}
Consider the case $(d,n)=(1,3)$.
A compact $\cD_{1\subset 3}^{\fr}$-manifold is a link in a 3-manifold, equipped with a framing of the 3-manifold and a splitting of this framing along the link via the first coordinate.
A $\disk(\cD_{1\subset 3}^{\fr})$-algebra is an $\cE_3$-algebra $A$, a $\cE_1$-algebra $B$, and a morphism between $\cE_2$-algebras
\[
\alpha\colon {\sf HH}_\bullet(A)
\longrightarrow
{\sf HH}^\bullet(B)
\]
from the Hochschild chains to the Hochschild cochains.  
Through Theorem~\ref{singular.alpha.thm}, such algebraic data determines, via factorization homology, a $\Bbbk$-module associated to each framed link in a framed 3-manifold, $L\subset M$:
\[
\int_{(L\subset M)} (A,B,\alpha)~ \in ~\m_{\Bbbk}~.
\]

\end{example}

The above two results are specific instances of a more general paradigm.
\begin{prop}\label{classify.algebras}
Let $U =\RR^d \times \sC(L)$ be a basic.
Consider the right fibration $\cD_U\to \bsc$ from Example~\ref{ex.fr.general}.
Consider the right fibration $\cD_{>U} \to \bsc$ from Observation~\ref{DU.restrict}.
A $\disk(\cD_U)$-algebra in $(\m_{\Bbbk},\underset{\Bbbk}\ot)$ is equivalent to the following data.
\begin{enumerate}
\item
A $\disk(\cD_{>U})$-algebra $A$ in $(\m_{\Bbbk},\underset{\Bbbk}\ot)$.
\begin{itemize}
\item[~]
Note how restriction along~(\ref{50}) in Observation~\ref{DU.restrict}, followed by factorization homology, defines a $\cE_{d+1}$-algebra in $(\m_{\Bbbk},\underset{\Bbbk}\ot)$:
\[
\int_L A\colon \disk_{d+1}^{\fr}
\underset{(\ref{50})}{\xra{~\times L~}}
\mfld(\cD_{>U})
\xra{~\int A~}
\m_{\Bbbk}
.
\]
\end{itemize}

\item
A $\cE_d$-algebra $B$ in $(\m_{\Bbbk},\underset{\Bbbk}\ot)$.

\item
An action of $\int_L A$ on $B$, instantiated as a morphism between $\cE_{d+1}$-algebras
\[
\int_L A
\longrightarrow
\sZ(B)~.
\]

\end{enumerate}

\end{prop}

The next result references Example~\ref{ex.U.boundary}, notably the right fibration $\cD_{\langle n\rangle}^{\fr}\to \bsc$ for which a $\cD_{\langle n\rangle}^{\fr}$-manifold is a framed $\langle n\rangle$-manifold.  
\begin{cor}\label{corner.data}
The following data canonically determines a $\disk_{\langle n\rangle}^{\fr}$-algebra in $(\m_{\Bbbk},\underset{\Bbbk}\ot)$.
\begin{enumerate}
\item
For each $0\leq i\leq n$, a $\cE_i$-algebra $A_i$ in $(\m_{\Bbbk},\underset{\Bbbk}\ot)$.

\item
For each $0<i\leq n$, a morphism between $\cE_i$-algebras
\[
\alpha_i\colon A_i
\longrightarrow
\sZ(A_{i-1})~.  
\]

\end{enumerate}

\end{cor}

\begin{remark}
Suppose $\Bbbk$ is an algebraically closed field whose characteristic is 0.
Recall from~\cite{PTVV} the notions of a shifted symplectic algebraic stack over $\Bbbk$, and of a Lagrangian map to one.
We also use the result therein which states that the derived intersection of two Lagrangian maps to an $n$-shifted symplectic stack canonically inherits the structure of a $(n-1)$-shifted symplectic stack.  
We expect that the data of Corollary~\ref{corner.data} can be supplied through a parametrized version of deformation quantization from the following input.
\begin{itemize}
\item 
A sequence
\[
L_0
\longrightarrow
L_1
\longrightarrow
\dots
\longrightarrow
L_n
\]
of morphisms between algebraic stacks over $\Bbbk$.
\begin{itemize}
\item
Now, denote $X_n := L_n$, and for each $0<i<n$, inductively denote the derived intersection $X_i:= L_i \underset{X_{i+1}}\times L_i$.

\end{itemize}

\item
$L_n$ is equipped with an $n$-shifted structure.

\item
For each $0\leq i<n$, the diagonal map
\[
L_{i-1}
\longrightarrow
L_{i}\underset{X_{i+1}} \times L_{i}
\]
is equipped with a Lagrangian structure.  

\end{itemize}

\end{remark}

\end{document}